\documentclass[10pt]{article}

\usepackage{amssymb,amsthm}
\usepackage[centertags]{amsmath}
\usepackage{MnSymbol}
\usepackage[mathscr]{euscript}
\usepackage{tikz-cd}
\usepackage{hyperref}
\hypersetup{
    colorlinks=true,
    linkcolor=[HTML]{0D47A1},		
    citecolor=[HTML]{0D47A1},		
    urlcolor=[HTML]{0D47A1}		    
}
\usepackage{fullpage}
\usepackage{rotating}
\usepackage[capitalize,nameinlink]{cleveref}
\usepackage{enumitem}
\usepackage{bbm}

\usepackage[shortcuts]{extdash}

\usepackage{tabularx}

\definecolor{mat}{HTML}{FFD6AD}
\definecolor{eric}{HTML}{ffadad}
\definecolor{andre}{HTML}{a9f5ae}
\definecolor{georg}{HTML}{faffad}

\newtheorem{thm}{Theorem}[subsection]
\newtheorem{theorem}[thm]{Theorem}
\newtheorem{lem}[thm]{Lemma}
\newtheorem{lemma}[thm]{Lemma}

\newtheorem{proposition}[thm]{Proposition}
\newtheorem{cor}[thm]{Corollary}
\newtheorem{corollary}[thm]{Corollary}
\theoremstyle{definition}

\newtheorem{definition}[thm]{Definition}

\newtheorem{exmp}[thm]{Example}
\newtheorem{exmps}[thm]{Examples}
\newtheorem{rem}[thm]{Remark}
\newtheorem{remark}[thm]{Remark}

\newlist{exmpenum}{enumerate}{1}
\setlist[exmpenum]{label=\alph*), ref=\theproposition~(\alph*)}
\crefalias{exmpenumi}{exmp} 

\newtheorem{thmintro}{}
\newenvironment{thm-intro}[1]
  {\renewcommand\thethmintro{#1}\thmintro\itshape}
  {\endthmintro}
\newenvironment{thm-introref}[2]
  {\renewcommand\thethmintro{#1}\thmintro(#2)\itshape}
  {\endthmintro}
\newtheorem{defnintro}{}
\newenvironment{defn-intro}[1]
  {\renewcommand\thedefnintro{#1}\defnintro}
  {\enddefnintro}

\newcommand{\dom}[1]{\mathrm{dom}(#1)}

\newcommand\pbmark{\ar[dr, phantom, "\ulcorner" very near start, shift right=1ex]}

\newcommand\pomarkk{\ar[ull, phantom, "\lrcorner" very near start, shift right=1ex]}

\newcommand\fperp{\upModels}

\DeclareMathOperator*{\colim}{colim}
\newcommand\Map[2]{\mathrm{Map}\left(#1,#2\right)}
\newcommand\Loc[2]{\mathrm{Loc}\left(#1,#2\right)}
\newcommand\Mod[2]{\mathrm{Mod}\left(#1,#2\right)}
\newcommand\Prsh[1]{\mathrm{PSh}\left(#1\right)}
\newcommand\Sh[2]{\mathrm{Sh}\left(#1,#2\right)}
\newcommand\Shsite[2]{\mathrm{Sh}\left[#1,#2\right]}

\newcommand\cA{\mathscr{A}}
\newcommand\cB{\mathscr{B}}
\newcommand\cC{\mathscr{C}}
\newcommand\cD{\mathscr{D}}
\newcommand\cE{\mathscr{E}}
\newcommand\cF{\mathscr{F}}
\newcommand\cG{\mathscr{G}}

\newcommand\cI{\mathscr{I}}

\newcommand\cK{\mathscr{K}}
\newcommand\cL{\mathscr{L}}

\newcommand\cQ{\mathscr{Q}}
\newcommand\cR{\mathscr{R}}
\newcommand\cS{\mathscr{S}}
\newcommand\cT{\mathscr{T}}
\newcommand\cU{\mathscr{U}}

\newcommand\cW{\mathscr{W}}
\newcommand\cX{\mathscr{X}}
\newcommand\cY{\mathscr{Y}}

\newcommand\bA{\mathbf{A}}
\newcommand\ZZ{\mathbb{Z}}
\newcommand\bone{\mathbf{1}}

\newcommand\xto\xrightarrow

\newcommand\Tau{\mathrm{T}}
\newcommand\Fin{\mathscr{F}\mathrm{in}}
\newcommand\Cat{\mathsf{Cat}}
\newcommand\CAT{\mathsf{CAT}}
\newcommand\op{^\mathrm{op}}
\newcommand\slice[1]{_{/#1}}
\newcommand\ho[1]{\mathsf{ho}(#1)}
\newcommand\fun[2]{\left[#1,#2\right]}
\newcommand\LOC[2]{#1[#2^{-1}]}
\newcommand\LOCcc[2]{#1[#2^{-1}]_{\mathrm{cc}}}
\newcommand\LOCcclex[2]{#1[#2^{-1}]_{\mathrm{cclex}}}

\newcommand\Toposop{\mathsf{Topos_{alg}}}

\newcommand\cc{_\mathrm{cc}}
\newcommand\lex{^\mathrm{lex}}
\newcommand\cclex{_\mathrm{cclex}}
\newcommand\oo{$\infty$\=/}
\newcommand\ooo{$(\infty,1)$\=/}
\newcommand\pointed{^{\bullet}}

\newcommand\arr{^\to}
\newcommand\id{{\rm id}}
\newcommand\sat{^\mathsf{s}}
\newcommand\ssat{^\mathsf{ss}}
\newcommand\ac{^\mathsf{a}}
\renewcommand\cong{^\mathsf{c}}
\newcommand\diag{^\Delta}
\newcommand\bc{^\mathsf{bc}}

\newcommand\Iso{\mathsf{Iso}}
\newcommand\All{\mathsf{All}}
\newcommand\Mono{\mathsf{Mono}}
\newcommand\Surj{\mathsf{Surj}}
\newcommand\Trunc{\mathsf{Trunc}}
\newcommand\Conn{\mathsf{Conn}}

\newcommand\im[1]{\mathsf{im}\left(#1\right)}

\newcommand\truncated[1]{^{\leq #1}}
\newcommand\connected[1]{_{> #1}}
\newcommand\ptruncated[1]{^{\bullet, \leq #1}}
\newcommand\pconnected[1]{^\bullet _{> #1}}

\makeatletter
\pgfqkeys{/tikz/commutative diagrams}{
  row sep/.code={\tikzcd@sep{row}{#1}{}},
  column sep/.code={\tikzcd@sep{column}{#1}{}},
  bo row sep/.code={\tikzcd@sep{row}{#1}{between origins}},
  bo column sep/.code={\tikzcd@sep{column}{#1}{between origins}},
  bo column sep/.default=normal,
  bo row sep/.default=normal,
}
\def\tikzcd@sep#1#2#3{
  \pgfkeysifdefined{/tikz/commutative diagrams/#1 sep/#2}%
    {\pgfkeysalso{/tikz/#1 sep={\ifx\\#3\\1*\else1.7*\fi\pgfkeysvalueof{/tikz/commutative diagrams/#1 sep/#2},#3}}}%
    {\pgfkeysalso{/tikz/#1 sep={#2,#3}}}}
\makeatother

\title{
Left-exact Localizations of \texorpdfstring{$\infty$-Topoi}{Infinity-Topoi} I:\\
Higher Sheaves
}
\author{Mathieu Anel\footnote{Department of Philosophy, Carnegie Mellon University, mathieu.anel@protonmail.com} ,
Georg Biedermann\footnote{Departamento de Matem\'{a}ticas y Estad\'{i}stica, Universidad del Norte, gbm@posteo.de} ,
Eric Finster\footnote{Cambridge University, ericfinster@gmail.com} ,
and Andr\'{e} Joyal\footnote{CIRGET, UQ\`AM. joyal.andre@uqam.ca} 
}
\date{}

\begin{document}

\maketitle

\begin{abstract}
We are developing tools for working with arbitrary left-exact localizations of \oo topoi.
We introduce a notion of higher sheaf with respect to an arbitrary set of maps $\Sigma$ in an $\infty$-topos $\cE$. 
We show that the full subcategory of higher sheaves $\Sh \cE \Sigma$ is an $\infty$-topos, and that the sheaf reflection $\cE\to \Sh \cE \Sigma$ is the left-exact localization generated by $\Sigma$.
The proof depends on the notion of congruence, which is a substitute for the notion of Grothendieck topology in 1-topos theory.
\end{abstract}

\setcounter{tocdepth}{2}
\tableofcontents

\section{Introduction}

The theory of \oo topoi ressembles ordinary topos theory in many aspects, but there are some important differences.
A striking difference is the insufficiency of Grothendieck topologies for encoding all left-exact localizations.
Recall that every Grothendieck 1-topos $\cE$ arises as a reflective left-exact localization of the category of presheaves of sets $\mathrm{PSh}_0(\cC) := [\cC\op, \mathsf{Set}]$ on a small category $\cC$, and that any such localization uniquely corresponds to a {\it Grothendieck topology} on $\cC$.
The objects in the reflective subcategory of $\mathrm{PSh}_0(\cC)$ are then the {\it sheaves} for the topology.
Furthermore, the Grothendieck topology itself may be viewed as a distinguished class of {\it monomorphisms} in $\mathrm{PSh}_0(\cC)$, and an equivalent characterization of the sheaves is that they are exactly the objects which are local for this class.
In other words, the left-exact localizations of $\mathrm{PSh}_0(\cC)$ can always be presented by inverting a set of monomorphisms.


In the theory of \oo topoi, the category of sets is replaced by the \oo category of \oo groupoids $\cS$.
It remains true that every \oo topos $\cE$ is a left-exact localization of a category of presheaves $\Prsh{\cC} := [\cC\op,\cS]$ for some small \oo category $\cC$, but the class of all such localizations can no longer be completely described using Grothendieck topologies.
More precisely, it is not true anymore that every left-exact localization of $\Prsh{\cC}$ is generated by inverting monomorphisms.
Recall from \cite[Proposition 6.5.2.19.]{HTT}, that every left-exact localization $\Prsh{\cC}\to \cE$ is the composite of a {\it topological} localization $\Prsh{\cC}\to \cE'$ followed by a {\it cotopological} localization $\cE' \to \cE$.
While the topological part can indeed still be described entirely as inverting a class of monomorphisms, the cotopological part is a localization which inverts {\it no} monomorphisms.
The classical theory of sheaves and Grothendieck topologies can be applied to the topological localizations, but not to the cotopological ones.

\bigskip

Our goal in the present paper is to provide tools which are suited for arbitrary left-exact localizations.
\begin{enumerate}

\item We introduce a general notion of {\it sheaf} with respect to {\it an arbitrary set of maps $\Sigma \subseteq \cE$} (not necessarily monomorphisms) in an \oo topos $\cE$.
Our main theorem (\cref{univlexloc}) says that the left-exact localization generated by $\Sigma$ is exactly the category of sheaves with respect to $\Sigma$ (see below).

\item The definition of sheaves is obtained through the study of classes of maps inverted by cocontinuous left-exact functors $\phi:\cE\to \cF$ between \oo topoi, called {\it congruences}.
Our second main result (\cref{cor:generation-congruence}) is an explicit description of the congruence generated by an arbitrary class of maps $\Sigma\subseteq \cE$ (see below).
\end{enumerate}

\bigskip
We will now describe in details our results.
As the rest of the introduction will be solely concerned with the higher categorical situation, we will from now on drop the prefix ``$\infty$'' when referring to \oo topoi and \ooo categories, and speak explicitly of $1$-categories and $1$-topoi if the occasion arises.
See \cref{sec:convents} for a more complete description of our conventions with respect to higher category theory.

\bigskip

Given an arbitrary set of maps $\Sigma$ in a topos $\cE$, it was proved in \cite{HTT} that there exists a cocontinuous left-exact localization $\cE\to \cE_\Sigma$ inverting $\Sigma$ universally.
The objects of $\cE_\Sigma$ are defined as the local objects with respect to the smallest strongly saturated class closed under base changes $\overline\Sigma$ containing $\Sigma$.
A fundamental problem is to give a more direct description of the full subcategory $\cE_\Sigma\subseteq \cE$ in terms of $\Sigma$.
To explain our answer, 
we will need the following construction: for a map $u:A\to B$ in $\cE$, we denote by $\Delta(u):A\to A\times_B A$ its diagonal and by $\Delta^n(u)$ its $n$-th iterated diagonal (with the convention that $\Delta^0(u):=u$).  
We define the \emph{diagonal closure} of a set $\Sigma$ of maps by 
$\Sigma\diag=\{\Delta^n(u) \ | \ u\in \Sigma, n\in \mathbb N\}$.

\begin{defn-intro}{\cref{defn:sigmasheaf}}
Let $\Sigma$ be a set of maps in a topos $\cE$.
We say that an object $X \in \cE$ is a \emph{$\Sigma$-sheaf} if for every base change $u' : A' \to B'$ of a map $u : A \to B \in \Sigma\diag$, the map $\Map{u'} X: \Map{B'} X \to \Map{A'} X $ is invertible.
We write $\Sh \cE \Sigma$ for the full subcategory of $\Sigma$-sheaves.
\end{defn-intro}

\noindent Our main theorem is then the following:

\begin{thm-intro}{\cref{univlexloc}}
Let $\Sigma$ be a set of maps in a topos $\cE$.
Then the full subcategory $\Sh \cE \Sigma$ is reflective and the reflector $\rho : \cE\to \Sh \cE \Sigma$ is left-exact.
In particular, $\Sh \cE \Sigma$ is a topos.
Furthermore, the reflector $\rho$ inverts the maps in $\Sigma$ universally among cocontinuous left-exact functors.
\end{thm-intro}

In practice, we will view \cref{defn:sigmasheaf} as a composite of several, simpler definitions, and the proof of \cref{univlexloc} arises from a careful examination of the interactions of these more basic notions.
For example, the reader will no doubt have recognized the condition on the equivalence induced on mapping spaces as an instance of the notion of \emph{local object}, well known in the literature on localizations.
Recall that an object $X$ of a category $\cE$ is said to be \emph{local} for a set of maps $\Sigma\subseteq \cE$ if the map $\Map u X :\Map B X \to \Map A X$ is invertible for every map $u:A\to B$ in $\Sigma$.
We write $\Loc \cE \Sigma \subseteq \cE$ for the full subcategory of $\cE$ spanned by the $\Sigma$-local objects.
If $\Sigma$ is a set of maps in a presentable category $\cE$, then the inclusion functor $\Loc \cE \Sigma \subseteq \cE$ has a left adjoint $\rho: \cE\to \Loc \cE \Sigma$ called the {\it reflector}.

\medskip

Perhaps less well known is the notion of a {\it modal} object in a topos.
We say that an object $X\in \cE$ is {\it modal} with respect to a set of maps $\Sigma\subseteq \cE$ if it is local with respect to every base change of every map in $\Sigma$ (\cref{defmodalobject}).
We write $\Mod \Sigma \cE \subseteq \cE$ for the full subcategory spanned by the $\Sigma$-modal objects.
It is worth noting that if $\cG\subseteq \cE$ is a set of generators of the presentable category $\cE$, then it suffices to consider only the base changes along maps having their domain in $\cG$.
It follows that the modal objects with respect to a set of maps $\Sigma$ can themselves be described as the local objects with respect to another set of maps, albeit a slightly larger one.

In view of the preceding discussion, our definition of $\Sigma$-sheaf in a topos $\cE$ may now be rephrased as follows: an object $X \in \cE$ is a $\Sigma$-sheaf if it is \emph{modal for $\Sigma\diag$, the diagonal closure of $\Sigma$}.
Furthermore, we have a nested sequence of full subcategories
\[
\Sh \cE \Sigma
\ \subseteq\ 
\Mod \cE \Sigma
\ \subseteq\ 
\Loc \cE \Sigma
\ \subseteq\ 
\cE
\]
each of which is reflective in $\cE$.

\medskip

An important technical tool throughout this work is the connection between the three reflective subcategories just described and the theory of \emph{factorization systems}.
Let us recall briefly that a factorization system $(\cL,\cR)$ on a category $\cE$ consists of two classes of maps $\cL$ and $\cR$, called the \emph{left class} $\cL$ and the \emph{right class} $\cR$, which are \emph{orthogonal} and for which every map $f : A \to B \in \cE$ admits an essentially unique factorization
$f = A \xto{u} E \xto{v} B$ where $v \in \cL$ and $u \in \cR$ (see \cref{sec:fact-syst} for precise definitions). 
We shall put $\|f\|=E$ and this defines a functor $\|-\|:\cE\arr\to \cE$,
where $\cE\arr$ is the arrow category of $\cE$.
Let us write $\cR[B] \subseteq \cE\slice{B}$ for the full subcategory of the slice category of $\cE$ at $B$ whose objects $(A,f)$ have a structure map $f:A\to B$ in $\cR$.
Then the association $(f:A\to B) \mapsto (v:\|f\|\to B)$ determines a reflector $\|-\|_B:\cE\slice{B} \to \cR[B]$.
In other words, the factorization system $S$ provides us with a family of reflective subcategories $\cR[B]$, one for each slice $\cE\slice{B}$ of the category $\cE$.
In particular, if $\cE$ has a terminal object $1$, we obtain a reflective subcategory $\cR[1]\subseteq \cE$.

\medskip

A fundamental observation, then, is that when $\cE$ is a topos and $\Sigma$ is a set of maps of $\cE$, then each of the full subcategories $\Sh \cE \Sigma$, $\Mod \cE \Sigma$ and $\Loc \cE \Sigma$ introduced above can be realized as a subcategory of the form $\cR[1]$ for a particular type of factorization system generated by $\Sigma$.
We shall review how these different types of factorization systems $(\cL,\cR)$ can be distinguished by the closure properties of their left classes
$\cL$.
The situation is summarized in \cref{table:FS}.

\begin{table}[htbp]
\begin{center}
\label[table]{table:FS}
\caption{Conditions of saturation}
\medskip
\renewcommand{\arraystretch}{1.6}
\begin{tabularx}{\textwidth}{
|>{\hsize=1\hsize\linewidth=\hsize\centering\arraybackslash}X|
|>{\hsize=1\hsize\linewidth=\hsize\centering\arraybackslash}X
|>{\hsize=1\hsize\linewidth=\hsize\centering\arraybackslash}X
|>{\hsize=1\hsize\linewidth=\hsize\centering\arraybackslash}X
|>{\hsize=1\hsize\linewidth=\hsize\centering\arraybackslash}X|}
\hline
Class of maps $\cL$ 
& {\it Saturated class}
& {\it Strongly saturated class}
& {\it Acyclic class}
& {\it Congruence}
\\
\hline Closure properties
& composition, isomorphisms, colimits
& saturated + 3-for-2
& saturated + base~change
& saturated + finite~limits
\\
\hline
Class generated by $\Sigma$
& $\Sigma\sat$
& $\Sigma\ssat$
& $\Sigma\ac$
& $\Sigma\cong$
\\
\hline
Corresponding factorization system $(\cL,\cR)$
& general factorization system
& reflective factorization system
& modality
& left-exact modality
\\
\hline
$\cR[1]$
& \multicolumn{2}{c|}{$\Loc \cE \Sigma$ }
& $\Mod \cE \Sigma $
& $\Sh \cE \Sigma$
\\

& \multicolumn{2}{c|}{Local objects}
& Modal objects
& Sheaves
\\
\hline
$\cE\to \cR[1]$
& \multicolumn{2}{c|}{ $\cE\to \Loc \cE \Sigma$ }
& $\cE\to \Mod \cE \Sigma$ 
&  $\cE\to \Sh \cE \Sigma$
\\
Localization inverting $\Sigma$
& \multicolumn{2}{c|}{universal for cocontinuous functors}
& 
& universal for lex cocontinuous functors
\\
\hline
\end{tabularx}
\end{center}  
\end{table}

\medskip
As a first step, let us note that a general factorization system $(\cL,\cR)$ in a cocomplete category $\cE$ has the property that its left class $\cL$ contains all the isomorphisms, is closed under composition and formation of colimits in the arrow category of $\cE$ (\cref{leftorthissat}).
Following \cite[Definition 5.5.5.1]{HTT}, we refer to a class of morphisms satisfying these properties as a \emph{saturated class}.
Every class of maps $\Sigma$ in a cocomplete category $\cE$ is contained in a smallest saturated class $\Sigma\sat$, the saturated  class \emph{generated} by $\Sigma$.
If the category $\cE$ is presentable and $\Sigma$ is a set of maps, then $\Sigma\sat$ is the left class of a factorization system $(\cL,\cR)$, the factorization system \emph{generated} by $\Sigma$.
It is easily seen in this case that the full subcategory $\cR[1]$ coincides with $\Loc \cE \Sigma$.  
 
\medskip

A related notion is that of {\it strongly saturated class} \cite[Definition 5.5.4.5]{HTT}.
A class is strongly saturated if it is saturated and has the 3-for-2 property.
We denote by $\Sigma\ssat$ the smallest strongly saturated class containing a class of maps $\Sigma$.
The class of maps $\cW\subseteq \cE$ inverted by a cocontinuous functors between cocomplete categories $\cE\to \cF$ is always strongly saturated.
If $\Sigma$ is a set of maps in a presentable category $\cE$, then $\Sigma\ssat$ is precisely the class of maps inverted by the reflector $\rho: \cE\to \Loc \cE \Sigma$.
Under these hypotheses, the class $\Sigma\ssat$ is also the left class of a factorization system $(\cL,\cR)$.
And we shall say that a factorization system $(\cL,\cR)$ is {\it reflective} if $\cL$ is strongly saturated.
Strongly saturated classes are a fundamental tool in the study of localizations (for example in \cite{HTT}) but
the closure of a class of maps for the 3-for-2 property is often difficult to describe explicitly.
We shall see below how our methods avoid this difficulty.

\medskip
We say that a factorization system $(\cL,\cR)$ in a topos $\cE$ is a {\it modality} if its left class $\cL$ is closed under base change \cite{RSS,ABFJ1,ABFJ2}.  
Since the right class is always closed under base change, this is equivalent to asking that the factorization of a map be stable under base change (\cref{sec:modalities}).
The full subcategory $\Mod \cE \Sigma$ from above may then be recovered as the full subcategory $\cR[1]$ associated to the modality generated by $\Sigma$ (\cref{thmgenerationmodality}).
In addition to being saturated, the left class of a modality is closed under base change by definition.
We call a saturated class with this property \emph{acyclic}, and it can be shown that the left class of the modality generated by $\Sigma$ is its acyclic closure, which we denote $\Sigma\ac$ (\cref{sec:acyclic-classes}).
A modality $(\cL,\cR)$ is said to be generated by a set $\Sigma$ if $\cL=\Sigma\ac$.
We discuss some motivation for this approach at the end of the introduction.

\medskip

A factorization system $(\cL,\cR)$ is said to be 
\emph{a left-exact modality} if the functor
$\|-\|:\cE\arr\to \cE$ is left-exact,
or equivalently, if its left class $\cL$ is closed under finite limits in the arrow category $\cE\arr$ \cref{def:lex-modality}.
This condition actually implies that $\cL$ is closed under base change and that the factorization system is indeed a modality (see the dual of \cref{colexrightcancel}).
Left-exact modalities are important because they are in bijective correspondence with left-exact localizations.
The bijection sends a left-exact localization $\phi:\cE\to \cF$ to a left-exact modality $(\cL_\phi,\cR_\phi)$ where $\cL_\phi$ is the class of maps inverted by $\phi$ (see \cref{lexlocvslexmod} and \cref{thm:bij-congruence-lexloc}).
This correspondence suggests to axiomatize the properties of the left class of left-exact modalities.
We define a {\it congruence} to be a saturated class of maps closed under finite limits.
Any set of maps $\Sigma$ in a topos $\cE$ generates a smallest congruence denoted $\Sigma\cong$.
The class of maps inverted by a cocontinuous left-exact functor $\phi:\cE\to \cF$ between topoi is a congruence.
The converse is true under some small generation assumption: a congruence $\cW=\Sigma\cong$, where $\Sigma$ is a set of maps in a topos $\cE$, will be the class of maps inverted by the  cocontinuous left-exact localization $\phi:\cE\to \Sh \cE \Sigma$ of \cref{univlexloc}.
This provides a bijection between congruences of small generation and accessible left-exact localizations.

\medskip

This equivalence is due to Lurie (see \Cref{sec:Lurie}).
As we shall see in \cref{charac-cong}, the notion of congruence is equivalent to the notion of a strongly saturated class of maps closed under base change introduced in \cite[6.2.1]{HTT}.
\Cref{charac-cong} provides also another description of congruences as an acyclic class $\cL$ closed under diagonals ($\Delta(\cL)\subseteq \cL$).
This description has the advantage of not referring to the 3-for-2 condition.
It has also the advantage of replacing in the definition of congruences the condition of closure under finite limits by the simpler condition of closure under base change and diagonals.
The equivalence between the acyclic classes closed under diagonals and strongly saturated classes closed under base change can be explained by the following ``reorganization trick'' of closure properties:
\begin{center}
\renewcommand{\arraystretch}{1.6}
\begin{tabular}{rl}
congruence 
&= strongly saturated + base change \\
&= (saturated + 3-for-2) + base change\\
&= (saturated + base change) + 3-for-2\\
&= acyclic + 3-for-2 \\
&= acyclic + diagonals.
\end{tabular}
\end{center}  

This reformulation leads to an explicit description of the congruence generated by a set of maps $\Sigma$ and to the notion of $\Sigma$-sheaf introduced above.

\begin{thm-intro}{\cref{cor:generation-acyclic-class} {\rm and} \cref{congruencegenerated}}
Given a set of generators $\cG$ of $\cE$, we denote by $\Sigma\bc$ the closure of $\Sigma$ under $\cG$-base changes.
For $\Sigma$ a set of maps in $\cE$, we have 
\[
\Sigma\ac = (\Sigma\bc)\sat
\qquad\textrm{and}\qquad
\Sigma\cong = (\Sigma\diag)\ac.
\]    
\end{thm-intro}

Putting the two formulas together, we get $\Sigma\cong = ((\Sigma\diag)\bc)\sat$. 
The characterization of $\Sigma$-sheaves and \cref{univlexloc} follow easily.
This description is the main reason we would like to advocate the study of left-exact localizations by the flexible tools of modalities, acyclic classes and congruences.

\bigskip
From a technical point of view, the proof of the fundamental formula $\Sigma\cong = (\Sigma\diag)\ac$ (\cref{congruencegenerated}) is a consequence of the following important recognition criteria for left-exact modalities.

\begin{thm-intro}{\cref{thm:recognition}}
Let $(\cL,\cR)$ be a modality generated by a set $\Sigma$ of maps in a topos $\cE$.
If $\Delta(\Sigma)\subseteq \cL$, then the modality $(\cL,\cR)$ is left-exact.
\end{thm-intro}

\noindent Since the proof of \cref{thm:recognition} occupies a large part of this paper, we detail here its main steps.
We will need the following definition:

\begin{defn-intro}{\cref{def:requiv}}
Let $(\cL,\cR)$ be a modality in a category with finite limits $\cE$.
We shall say that a map $u:A\to B$ in $\cE$ is an $\cR$-{\it equivalence} if the functor $u^\star:\cR[B]\to \cR[A]$ is an equivalence of categories.
We shall say that $u$ is a {\it fiberwise $\cR$-equivalence} if every base change of $u$ is an $\cR$-equivalence.
\end{defn-intro}

\noindent The utility of this definition lies in the following alternative characterization of left-exact modalities:
\begin{thm-intro}{\cref{lexcharacterization}}
Let $\cE$ be a category with finite limits.
Then a modality $(\cL, \cR)$ in $\cE$ is left-exact if and only if the functor $u^\star:\cR[B]\to \cR[A]$ is an equivalence of categories for every map $u:A\to B$
in $\cL$.
\end{thm-intro}

\noindent 
In a topos $\cE$, the class of $\cR$-equivalences of a modality $(\cL,\cR)$ is always strongly saturated (\cref{saturationRequivalence}) and every strongly saturated class $\cW$ in $\cE$ contains a largest acyclic class $\cA\subseteq \cW$ (\cref{largestacyclic}).
If $\cW$ is the class of $\cR$-equivalences of a modality $(\cL,\cR)$, then $\cA$ is the class of fiberwise $\cR$-equivalences.
Hence the class of fiberwise $\cR$-equivalences is acyclic (\cref{saturationRequivalence}).
We prove also in \cref{basechangesequiv} that a map $u:A\to B$ is a fiberwise $\cR$-equivalence if and only if $u\in \cL$ and $\Delta(u)\in \cL$.
By combining these results, we obtain \cref{thm:recognition} stated above.

\medskip

The rest of \cref{sec:lex-mod-hs} take advantage of our definition of $\Sigma$-sheaves to propose definitions for the notions of site and of sheaves on a site suited for $\infty$-topoi.
A \emph{site} is defined to be a pair $(\cK,\Sigma)$, where $\cK$ is a small category and $\Sigma$ is an arbitrary set of maps in the category of presheaves $\Prsh \cK=[\cK\op,\cS]$ (\cref{defsite}).
A presheaf $X$ on $\cK$ is a \emph{sheaf} on $(\cK,\Sigma)$ if $X$ is a $\Sigma$-sheaf in $\Prsh \cK$ (local with respect to every base change of the maps in $\Sigma\diag$).
We are denoting the category of sheaves on $(\cK,\Sigma)$ by $\Shsite \cK \Sigma$.
The category $\Shsite \cK \Sigma$ is a topos, and every topos is equivalent to the category of sheaves on a site \cref{prop:any-topos}. 
In \cref{sec:examples}, we provide some applications of our notions of $\Sigma$-sheaves and congruences by describing explicitly various left-exact localizations (such as the topoi classifying $n$-truncated or $n$-connected objects).

\bigskip

Finally, we would like to make some remarks of the use of the concept of modality in this work.
Indeed, the appearance of the intermediate notion of modal object and modality may seem surprising given that our stated goal is to understand left-exact localizations, and modalities do not have a simple analog in terms of localization theory.
There are at least two reasons, however, to recommend the approach taken here.
A first is that, though modalities have not received a great deal of attention in the literature on category theory and homotopy theory, they do arise quite naturally from the logical perspective on topos theory, a point of view whose higher analog goes by the name \emph{Homotopy Type Theory}.
From the perspective of the internal language, modalities appear as operations on types similar to the operators of modal logic, which is where the concept gets its name.
Indeed, many of the early properties of modalities were first worked out in this setting \cite{HoTT,RSS}.
Furthermore, our original approach to \cref{thm:recognition} was very much inspired by ideas from type theory, and in fact, the theorem can be given a completely type-theoretic proof.
Of course, other approaches are also possible: the language of the present paper is entirely categorical, as is that of \cite[Thm 3.4.16]{AS}, which presents an account based on the small object argument and sheafification techniques.

\medskip

A second motivation is that, owing to the ubiquity with which they appear in various parts of homotopy theory and higher category theory, we feel modalities deserve to be better known.
For example, the factorization system given by factoring a map as an $n$-connected map followed by an $n$-truncated one, well known in classical homotopy, is
a modality.
As is the factorization system obtained from Quillen's $+$-construction, whose left class consists of the so-called \emph{acyclic maps} \cite{Raptis}.
Even the $\mathbb{A}^1$-local objects of motivic homotopy theory \cite{MV} are part of a modality on the Nisnevich topos generated by $\mathbb{A}^1$ (i.e. $\mathbb{A}^1$-local objects can be viewed as $\mathbb{A}^1$-modal objects).
Furthermore, arbitrary modalities in a topos may be seen as generalizations of the notion of \emph{closed class} introduced by Dror-Farjoun \cite{DF95} (see \cref{rem:cc}).
In short, these objects appear frequently (if not always explicitly) in the literature, and we hope the present paper will contribute to their adoption and study.

\bigskip

\noindent {\bf Acknowledgments:} 
The authors would like to thank 
Dan Christensen, 
Simon Henry,
Egbert Rijke, and 
Mike Shulman,
for useful discussions on the material of this paper.
The first author gratefully acknowledges the support of the Air Force
Office of Scientific Research through grant FA9550-20-1-0305.

\section{Preliminaries}
\label{sec:preliminaries}

\subsection{Conventions and notations}
\label{sec:convents}

Throughout the paper, we use the language of higher category theory.
The word \emph{category} refers to \ooo category, and all constructions are assumed to be homotopy invariant.
When necessary, we shall refer to an ordinary category as a \emph{1-category} and to an ordinary Grothendieck topos as a \emph{1-topos}.
Furthermore, we work in a model independent style, which is to say, we do not choose an explicit combinatorial model for \ooo categories such as quasicategories, but rather give arguments which we feel are robust enough to hold in any model.
We will refer to the work of Lurie \cite{HTT} for the general theory of \oo categories and \oo topoi.
Other references are \cite{Cisinski} and \cite{RV}.

\medskip

We use the word \emph{space} to refer generically to a homotopy type or $\infty$-groupoid.
We shall denote by $\cC(A,B)$ or by $\mathrm{Map}_{\cC}(A,B)$ the space of maps between two objects $A$ and $B$ of a category $\cC$ and write $f:A\to B$ to indicate that $f \in \cC(A,B)$.
We write $A \in\cC$ to indicate that $A$ is an object of $\cC$.  

\medskip

In our article we are dealing with classes of different sizes and we will explain here our conventions.
We fix an inaccessible cardinal $\kappa$.
We call a class that is $\kappa$-small a {\it set}. 
We call a class with cardinality $\le \kappa$ a {\it collection}.
Everything that is bigger we will call {\it class}. 
A space $X$ will be called {\it small} if $\pi_0X$ and all homotopy groups $\pi_n(X,x_0)$ are ($\kappa$-small) sets.
We denote the category of such spaces by $\cS$.
A category will be called {\it locally small} if it is enriched over $\cS$, that is, if all its morphism spaces are small.
A category will be called {\it small} if it is locally small and its underlying class of isomorphism classes of objects is small.
We denote the category of small categories by $\Cat$.
The category $\Cat$ is large in the sense that its class of isomorphism classes of objects is a proper collection.
Unless otherwise specified all categories in this paper are assumed to be large in this sense: they will be locally small, and the class of objects will be a collection.
We have one exception to this rule: in \cref{fullfaithlim,descentforF} we will consider the category of categories $\CAT$ (denoted $\widehat{\cC\mathrm{at}}_\infty$ in \cite{HTT}). It is the category of those categories that are locally small and whose isomorphism classes of objects form a collection.
Clearly, $\CAT$ itself is not locally small and its class of objects is a proper class.
 
\medskip

The \emph{opposite} of a category $\cC$ is denoted $\cC\op$ and defined by the fact that $\cC\op(B,A):=\cC(A,B)$ with its category structure inherited from $\cC$.
We write $\cC\slice{A}$ for the slice category of $\cC$ over an object $A$. 
If $f : X \to A$ is a morphism of $\cC$, we often write $(X,f) \in \cC\slice{A}$, as it is frequently convenient to have both the object and structure map visible when working in a slice category.
If a category $\cC$ has a terminal object, we denote it by $1$.

\medskip

Every category $\cC$ has a {\it homotopy category} $\ho\cC$ which is a 1-category with the same objects as $\cC$, but where $ho\cC(A,B)=\pi_0\cC(A,B)$.
We shall say that a morphism $f:A\to B$ in $\cC$ is {\it invertible}, or that it is an {\it isomorphism} if the morphism is invertible in the homotopy category $\ho\cC$. 
We make a small exception to this terminology with regard to categories and spaces: for these objects, we continue to employ the more traditional term \emph{equivalence}.

\medskip

We shall say that a functor $F:\cC\to \cD$ is {\it essentially surjective} if for every object $X\in\cD$ there exists an object $A\in\cC$ together with an isomorphism $X\simeq FA$.
We shall say that $F$ is {\it fully faithful} if the induced map $\cC(A,B)\to \cD(FA,FB)$ is invertible for every pair of objects $A,B\in\cC$.
And $F$ is an {\it equivalence} (of categories) if it is fully faithful and essentially surjective.
We denote the category of functors from $\cC$ to $\cD$ alternatively by $[\cC,\cD]$ or $\cD^{\cC}$ as seems appropriate from the context.
For a small category $\cC$, we will write $\Prsh{\cC} := [\cC\op,\cS]$ for the category for presheaves on $\cC$.
Recall that the {\it Yoneda functor} $Y:\cC\to \Prsh{\cC}$ is defined by putting $Y(A)(B):=\cC(B,A)$ for every objects $A,B\in \cC$.

\medskip

We say an object is \emph{unique} if the space it inhabits is contractible.
For example, the inverse of an isomorphism is unique in this sense.
We assume that all full subcategories and classes of maps in a category are defined by properties which are invariant under isomorphism, and consequently {\it we adopt the convention that all full subcategories are replete}.

\medskip

A {\it diagram} in a category $\cE$ is defined to be a functor $\cK\to \cE$, where $\cK$ is a small category.
The
{\it diagonal functor} $\delta:\cE\to [\cK,\cE] $ takes an object $A\in\cE$ to the constant diagram $\delta A:\cK\to \cE$ with value $A$.
A  {\it cone} with {\it base} $D:\cK\to \cE$ and {\it apex} $A\in \cE$ is defined to be a natural transformation $\alpha:D\to \delta A$; the cone is a {\it colimit cone} if for every other cone $\beta:D\to \delta B$ there exists a unique map $u:A\to B$ such that $\beta =\delta(u)\alpha$. 
We shall often denote a colimit cone by $\gamma(D):D\to \delta  \colim D$.
A category $\cE$ is {\it cocomplete} if, for any small category $\cK$, any diagram $\cK\to \cE$ admits a colimit. 
There are dual notion of {\it $\cK$-limits} and {\it complete} categories.
A category $\cE$ is {\it finitely complete} or {\it lex} if it has a terminal object and fiber products.
We call a functor ({\it co-}){\it complete} if it preserves all (co-)limits. 
A functor between lex categories is called {\it left-exact} or {\it lex} if it preserves terminal objects and fiber products.

\medskip

We will denote by $[1]$ the poset $\{0<1\}$, regarded as a category.
If $\cC$ is a category, then the category of arrows of $\cC$ is the functor category $\cC\arr:=\cC^{[1]}$.
By construction, an object of $\cC\arr$ is a map $u:A\to B$ in $\cC$, and a morphism $\alpha:u\to u'$ in $\cC\arr$ is a pair of maps $(f,g)$ fitting in a commutative square
\[
    \begin{tikzcd}
      A\ar[d, "u"'] \ar[r, "f"] & A' \ar[d, "{u'}"]  \\
      B\ar[r, "g"] & B'\,.
    \end{tikzcd}
\]

\medskip
 
Given a commutative square as above in a category $\cC$ with finite limits we refer to the canonically induced map $(f,g) : A \to C \times_D B$ as the \emph{cartesian gap map} of the square.
A square is cartesian (= is a pullback) if and only if its  cartesian gap map is invertible.

\medskip

For an object $A$ of a category $\cC$ with finite limits, we will write $\Delta(A)=(1_A,1_A):A\to A\times A$ for the canonical map, which we refer to as the \emph{diagonal of $A$}.
More generally, the \emph{diagonal} of a map $u:A\to B$ is defined to be the canonical map $\Delta(u)=(1_A,1_A):A\to A\times_B A$
\[
  \begin{tikzcd}
    A \ar[dr, "\Delta(u)"] \ar[drr, "1_A", bend left] \ar[ddr, "1_A"', bend right] && \\
    & A\times_B A\ar[r, "p_2"] \ar[d, "p_1"'] \pbmark & A\ar[d, "u"] \\
    & A \ar[r, "u"']  & B
  \end{tikzcd}
\]
induced by the universal property of the pullback.
This construction can be iterated, and we use the notation $\Delta^k(u)$ for the $k$-th iterated diagonal of a map or object.
Let $S^k$ be the $k$-sphere in $\cS$. 
Since it is a finite space, the cotensor $A^{S^k}$ of $A$ by $S^k$ exist in any category with finite limits.
Using this operation, the iterated diagonal are maps $\Delta^k(u) : A\to A^{S^{k-1}}\times_{B^{S^{k-1}}}B$.

\subsection{Localizations, reflections, and presentability}
\label{sec:refl-local}

Recall that a functor $F:\cE\to \cF$ is said to {\it invert} a class of maps $\Sigma \subseteq \cE$ if the map $F(f)\in \cF$ is invertible for every map $f\in \Sigma$.
The functor $F$ is said to be a $\Sigma$-{\it localization} if it is initial in the category of functors which invert $\Sigma$.
More precisely, for any category $\cG$, let us denote by $[\cE,\cG]^\Sigma$ the full subcategory of $[\cE,\cG]$ spanned by the functors $\cE\to \cG$ inverting the class $\Sigma$.
If $F$ inverts $\Sigma$, then the composition $(-)\circ F: [\cF,\cG] \to [\cE,\cG]$ induces a functor
\[
  (-)\circ F: [\cF,\cG]\to [\cE,\cG]^\Sigma.
\]
The functor $F$ is said to be a $\Sigma${\it -localization} if the induced functor is an equivalence of categories for every category $\cG$.
We shall say that $F$ is a {\it localization} if it is a $\Sigma$-localization with respect to some class of maps $\Sigma\subseteq \cE$ (equivalently, if it is a localization with respect to the class of all maps inverted by $F$).
If $\Sigma$ is a class of maps in a category $\cE$, then the codomain $\cF$ of any $\Sigma$-localization $\cE\to \cF$ is unique up to equivalence of categories, and we denote the codomain $\cF$ generically by $\LOC \cE \Sigma$.

\medskip

While arbitrary localizations can be difficult to describe in general, those which arise from \emph{reflective subcategories} are more tractable. 
If $\cE'$ is a full subcategory of a category $\cE$, we shall say that a map $r:X\to X'$ {\it reflects the object $X$ into $\cE'$}, or simply that $r$ is a {\it reflecting map}, if $X' \in \cE'$ and the map 
\[
\Map r Z : \Map{X'} Z \to \Map X Z 
\]
is invertible for every object $Z\in \cE'$.
We shall say that the full subcategory $\cE'$ is {\it reflective} if for every object $X\in \cE$ there exists a map $r:X\to X'$ which reflects $X$ into $\cE'$.
A full subcategory $\cE'$ of $\cE$ is reflective if and only if the inclusion functor $\iota:\cE'\hookrightarrow \cE$ has a left adjoint $\rho:\cE\to \cE'$ called the {\it reflector}, or the {\it reflection}.  
Indeed, the choice of a reflecting map $\eta(X):X\to \rho(X)$ for each object $X\in\cE$ determines an endofunctor $\rho:\cE\to \cE$ together with a natural transformation $\eta:\id \to \rho$ (= the unit of the adjunction $\rho \vdash \iota$).
More generally, we shall say that a functor $\rho:\cE\to \cF$ is a {\it reflector} if it has a fully faithful right adjoint $\iota: \cF\to \cE$.

If $\eta:\id \to \iota\rho$ is the unit of the adjunction $\rho\dashv \iota$, then the reflector $\rho:\cE\to \cF$ is also localization with respect to the class $\Sigma$ of unit maps  $\eta(X) : X \to \iota\rho(X)$ for $X\in \cE$.

\begin{proposition}
\label{prop:loc-vs-reflection}
If a localization $\phi:\cE\to \cF$ has a right adjoint $\iota$, then $\iota$ is fully faithful and $\phi$ is a reflector.
Conversely, if $\phi:\cE\to \cF$ is a reflector with right adjoint $\iota$, then it is the localization of $\cE$ by the class $\Sigma$ of unit maps $\eta(X):X\to \iota\phi(X)$.
\end{proposition}

\begin{proof} 
It suffices to show that the counit $\epsilon: \phi\circ \iota\to \id_{\cF}$ of the adjunction $\phi\vdash \psi$ is invertible.
If $\cG$ is a category, then the adjunction $\phi\vdash \iota$ induces an adjunction $[\iota, \cG]\vdash [\phi, \cG]$.
If $\cF=\LOC \cE \Sigma$, then the functor $[\phi, \cG]:[\cF,\cG]\to [\cE,\cG]$ is the composite of an equivalence $[\cF,\cG]\to [\cE,\cG]^\Sigma$ followed by the inclusion of a full subcategory $[\cE,\cG]^\Sigma \subseteq [\cE,\cG]$.
Hence the functor $[\phi, \cG]$ is fully faithful.
It follows that the counit $[\epsilon,\cG]$ of the adjunction $[\iota, \cG]\vdash [\phi, \cG] $ is invertible.
This is true in particular if $\cG:=\cF$; but the value of the counit $[\epsilon, \cF]$ at $\id_{\cF}\in [\cF,\cF]$ is the natural transformation $\epsilon: \phi\circ \iota\to \id_{\cF}$.
This proves that $\epsilon$ is invertible, and hence that the functor $\iota$ is fully faithful.

Conversely, let us show first that $\phi$ inverts the units maps $\eta(X):X\to \iota\phi(X)$.
Using $\phi\iota = id$, we have that $\phi(\eta):\phi\to \phi\iota\phi$ is invertible by the  adjunction identities of $\phi\dashv \iota$.
This proves that the functor $\phi$ factors as $\cE\xto \lambda \LOC \cE \Sigma \xto \mu \cF$ because $\phi\iota\phi=\phi$.
We want to show that $\mu$ is an equivalence of categories.
Since $\phi\iota = id_\cF$, we have $\mu(\lambda\iota) = id_{\cF}$.
We need to prove that conversely  $(\lambda\iota)\mu \simeq id_{\LOC\cE\Sigma}$.
By the universal property of the localization $\lambda$, this is equivalent to prove the existence of an isomorphism $\lambda \simeq (\lambda\iota\mu)\lambda$.
We've seen above that all units $\eta(X):X\to \iota\phi(X)$ are inverted by $\phi$.
Hence $\lambda(\eta):\lambda \to \lambda\iota\phi$ is invertible.
Using that $\lambda\iota\mu\lambda = \lambda\iota\phi$ this gives an isomorphism $\lambda \simeq (\lambda\iota\mu)\lambda$.
\end{proof}

\begin{remark}
\cite[Definition 5.2.7.2]{HTT} defines a \emph{localization} to be what we have here called a \emph{reflection}.
In favorable cases, the two notions coincide as we will see below. We prefer, in any case, to distinguish them (see \cite[Warning 5.2.7.3]{HTT}).
\end{remark}

The notion of localization introduced in the preceding paragraph is the most basic, but certain variations also appear naturally.
For example, if $\cE$ and $\cF$ are cocomplete categories and $\Sigma \subseteq \cE$ is a class of maps in $\cE$, then a cocontinuous functor $F : \cE \to \cF$ may invert $\Sigma$ universally \emph{among cocontinuous functors}.
More precisely, if we denote by $[\cE,\cF]_{\mathrm{cc}}$ the category of cocontinuous functors $\cE\to \cF$ and if $[\cE,\cF]\cc^\Sigma: =[\cE,\cF]\cc \cap [\cE,\cG]^\Sigma$, the we will say that $F$ is a \emph{cocontinuous $\Sigma$-localization} if the induced functor
\[
  (-)\circ F: [\cF,\cG]_{\mathrm{cc}}\to [\cE,\cG]_{\mathrm{cc}}^\Sigma
\]
is an equivalence of categories for every cocomplete category $\cG$.
More prosaically, this means that if a cocontinuous functor $G:\cE\to \cG$ with values in a cocomplete category 
inverts every maps in $\Sigma$, then there exists
a unique (=homotopy unique) pair $(G',\alpha)$
where $G':\cF\to \cG$ is a 
cocontinuous functor and $\alpha$ is an isomorphism
$G\simeq G'\circ F$.
If $\Sigma$ is a class of maps in a cocomplete category $\cE$, then the codomain of any cocontinuous $\Sigma$-localization $\cE\to \cF$ is unique up to equivalence of categories and we denote
this codomain generically by $\LOCcc \cE \Sigma$.

\medskip
When $\cE$ is a topos and $\Sigma$ is a class of maps of $\cE$, then it is also natural to consider the notion of a \emph{cocontinuous left-exact $\Sigma$-localization} $\cE\to \cF$ which inverts $\Sigma$ universally among cocontinuous left-exact functors (=cocontinuous functors preserving  finite limits), and is defined analogously.
In this case, we denote the target of such a localization by $\LOCcclex \cE \Sigma$.

\begin{rem}
\label{rem:ccvslex}
The previous notions of localization need not coincide.
For example, if $\Sigma = \{ S^{n+1} \to 1 \}$ where $S^{n+1}$ is the $(n+1)$-sphere in $\cS$, then $\LOCcc \cS \Sigma$ is the category of $n$-truncated spaces, while $\LOCcclex \cS \Sigma$ is the terminal category.
\end{rem}

For any class of maps $\cA$ in a category $\cC$, we shall write $\underline{\cA}$ for the full subcategory of $\cC\arr$ whose objects are the maps in $\cA$.
This construction permits us to succinctly describe closure properties of the class $\cA$ in terms of the associated full subcategory $\underline{\cA}$.
For example, we say that $\cA$ is closed under colimits or limits if this is true of $\underline{\cA}$.

Recall that a class $\cW$ of maps in a category $\cE$ is said to have the {\it 3-for-2 property} if any commutative triangle with two maps in $\cW$
\[
  \begin{tikzcd}
   & B\ar[dr, "g"] &     \\
 A \ar[rr, "h"]  \ar[ru, "f"]   && C
  \end{tikzcd}
\]
has its third map in $\cW$.

\begin{definition}[Strongly saturated class]
\label{defstrongsaturated}
Let $\cE$ be a cocomplete category.
We shall say that a class of maps $\cW\subseteq \cE$ is {\it strongly saturated} if the following conditions hold:
\begin{enumerate}[label=\roman*)]
\item  $\cW$ contains the isomorphisms and it has the 3-for-2 property;
\item  ${\cW}$ is closed under colimits.
\end{enumerate}
\end{definition}

The definition of strongly saturated class in \cite[Definition 5.5.4.5]{HTT} includes the additional condition that the class $\cW$ be closed under cobase change.
The following lemma, however, shows that the additional condition is automatic.

\begin{lemma}[{\cite{AS}}]
\label{saturatedareclosedundercob}
Let $\cQ$ be a class of maps in a category with finite colimits $\cE$.
If $\cQ$ contains the isomorphisms and the full subcategory $\underline{\cQ}\subseteq \cE\arr$ is closed under pushouts, then $\cQ$ is closed under cobase change.
\end{lemma}

\begin{proof} 
If the square
\[
\begin{tikzcd}
A\ar[d, "u"'] \ar[rr, "f"] && A' \ar[d, "{u'}"]  \\
B\ar[rr, "g"] && B'
\end{tikzcd}
\]
is a pushout in the category $\cE$,  then the square
\[
\begin{tikzcd}
1_A\ar[d, "{(1_A,u)}"'] \ar[rr, "{(f,f)}"] && 1_{A'} \ar[d, "{(1_{A'},u')}"]  \\
u\ar[rr, "{(f,g)}"] && {u'}
\end{tikzcd}
\]
is also a pushout in the category $\cE\arr$.
We have $1_A,1_{A'}\in \underline{\cQ}$, since the class $ \cQ$ contains the isomorphisms.
Thus, $u\in \underline{\cQ}$ implies that $u'\in \underline{\cQ}$, since the full subcategory $\underline{\cQ}\subseteq \cE\arr$ is closed under pushouts.
\end{proof}

Let $\phi:\cE\to \cF$ be cocontinuous functor between cocomplete categories.
Then the class of maps $\cW_\phi\subseteq \cE$ inverted by the functor $\phi$ is strongly saturated.
More generally, the inverse image of a strongly saturated class by such a $\phi:\cE\to \cF$ is always strongly saturated.

\begin{definition}
\label{defsaturatedclassgen} 
Let $\cE$ be a cocomplete category.
Then every class of maps $\Sigma \subseteq \cE$ is contained in a smallest strongly saturated class $\Sigma\ssat$.
We shall say that $\Sigma\ssat$ is the strongly saturated class  \emph{generated} by the class $\Sigma \subseteq \cE$.
We shall say that a strongly saturated class of map $\cL \subseteq \cE$ is of {\it small generation} if $\cL=\Sigma\ssat$ for a set of maps $\Sigma \subseteq \cE$.
\end{definition}

\begin{definition}[{\cite[Definition 5.5.4.1]{HTT}}]
An object $X$ in a category $\cE$ is said to be \emph{local} with respect to a map $u:A\to B$ in $\cE$ if the map 
\[
\Map u X : \Map B X \to \Map A X
\]
is invertible.
The object $X$ is said to be local with respect to a class of maps $\Sigma\subseteq \cE$ if it is local with respect to every map in $\Sigma$.
We shall denote by $\Loc \cE \Sigma$ the full subcategory of $\cE$ spanned by the $\Sigma$-local objects.
\end{definition}

\noindent {\bf Warning.}
The full subcategory
$\Loc \cE \Sigma \subseteq \cE$
may not be reflective. 
Even if it is reflective, the reflector
$\rho:\cE\to \Loc \cE \Sigma$,
which is always a localization
by \cref{prop:loc-vs-reflection},
may not be the localization
with respect to $\Sigma$.

\medskip

\begin{definition}[Accessibility]
\label{def:accessible}
We shall say that a reflector $\phi:\cE\to \cF$ is {\it accessible} if $\cF = \Loc \cE \Sigma$ for a set of maps $\Sigma$ in $\cE$.
We shall say that such a category $\cF$ is an accessible reflection of $\cE$.
A category $\cE$ is said to be \emph{presentable} if it is an accessible reflection of a presheaf category $\Prsh\cK$ for a small category $\cK$.
In particular $\Prsh\cK$ is presentable.    
\end{definition}

When $\cE$ is a presentable category, the notion of accessible reflector is equivalent to the notion of accessible localizations defined in \cite{HTT}.     

\begin{proposition}[{\cite[Propositions 5.5.4.15 and 5.5.4.20]{HTT}}]
\label{localizationofpresentable}
If $\Sigma$ is a set of maps in a presentable category $\cE$, then the full subcategory $\Loc \cE \Sigma \subseteq \cE$ of $\Sigma$-local objects is presentable, reflective, and the reflector $\cE\to \Loc \cE \Sigma$ is a cocontinuous localization $\cE\to \LOCcc \cE \Sigma$.
Moreover $\Sigma\ssat$ is the class of maps inverted by $\cE\to \Loc \cE \Sigma$.
\end{proposition}

\begin{remark}
\label{localizationvsssat2}
It follows from \cref{prop:loc-vs-reflection,localizationofpresentable} that 
\[
\Loc\cE \Sigma
=\LOCcc \cE \Sigma
=\LOC \cE {(\Sigma\ssat)}.
\]
\end{remark}

\begin{proposition}[{\cite[Proposition 5.5.4.16]{HTT}}]
\label{stronglysatcocont}
Let $\phi:\cE\to \cF$ be cocontinuous functor between presentable categories.
Then the class of maps $\cW\subseteq \cE$ inverted by the functor $\phi$ is strongly saturated and of small generation.
\end{proposition}

\bigskip
We finish with some lemma on generators that will be used later in the paper.
If $\cE$ is a cocomplete category, we shall say that a class of objects $\cG\subseteq \cE$ {\it  generates} the category $\cE$, or that
$\cG$ is a class of {\it generators}, if every objects in $\cE$ is the colimit of a diagram of objects in ${\cG}$.
For example, the category of presheaves $\Prsh \cK$ on a small category $\cK$ is generated by the set of representable presheaves $\{Y(k) | \ k\in \cK\}$.
Every presentable category $\cE$ admits a set of generators.

\medskip

If $\cE$ is a cocomplete category, then so is the category $\cE\slice{B}$ for any object $B\in \cE$.  If $\cG\subseteq \cE$ is a class of objects, let us denote by $\cG\slice{B}$ the class of objects $(G,g)\in \cE\slice{B}$ with $G\in \cG$.

\begin{lemma}
\label{generatinglemma}
If a cocomplete category $\cE$ is generated by a class of objects $\cG\subseteq \cE$, then the cocomplete category $\cE\slice{B}$ is generated by the class $\cG\slice{B}\subseteq \cE\slice{B}$ for every object $B\in \cE$.
\end{lemma}

\begin{proof}
If $(X,f)$ is an object of $\cE\slice{B}$, then the object $X$ is the colimit of a diagram $F:\cK\to \cE$ of objects of $\cG$, since $\cG$ generates $\cE$.
If $u_k:F(k)\to X$ denotes the inclusion for each object $k\in \cK$, consider the functor $F':\cK\to \cE\slice{B}$ defined by putting $F'(k)=(F(k),fu_k)$ for every object $k\in \cK$.
The family of morphisms $u_k:(F(k),fu_k)\to (X,f)$ is a colimit cone $F'\to \delta(X,f)$ since the family of morphisms $u_k:F(k)\to X$ is a colimit cone $F\to \delta X$.
Thus, $(X,f)$ is the colimit of the diagram $F':\cK\to \cE\slice{B}$.
We have $(F(k),fu_k)\in \cG\slice{B}$ for every object $k\in \cK$, since $F(k)\in \cG$ for every object $k\in \cK$.
This shows that $(X,f)$ is the colimit of a diagram of objects in $\cE\slice{B}$.
We have proved that $\cE\slice{B}$ is generated by the class $\cG\slice{B}$.
\end{proof}

\subsection{Descent and topoi}
\label{sec:topoi}

Recall that if $\cE$ is a category, the {\it push-forward functor} $u_!: \cE\slice{A}\to \cE\slice{B}$ along a map $u:A\to B$ in $\cE$ is defined by putting $u_!(X,f)=(X,uf)$ for every object $(X,f)\in \cE\slice{A}$.
If the category $\cE$ has pullbacks, the functor $u_!$ has a right adjoint $u^\star: \cE\slice{B}\to \cE\slice{A}$ called the  {\it pullback functor along $u$}, or the  {\it base change functor along $u$}.
By construction $u^\star(Y,g)=(A\times_B Y,p_1)$ for every object $(Y,g)\in \cE\slice{B}$.
\[
\begin{tikzcd}
A\times_B Y\ar[d, "{p_1}"']\ar[rr, "{p_2}"]  && Y  \ar[d, "{g}"]  \\
A \ar[rr, "{u}"] && B
\end{tikzcd}
\]
Let us now suppose that the category $\cE$ has finite limits and colimits.
We shall say that colimits in $\cE$ are {\it universal} if the base change functor $u^\star: \cE\slice{B}\to \cE\slice{A}$ cocontinuous for every map $u:A\to B$.

\medskip
 
Recall that if $\cK$ is a small category, then a natural transformation $\alpha:F\to G$ between $\cK$-diagrams in $\cE$ is said to be {\it cartesian} if the naturality square
\[
\begin{tikzcd}
F(j)\ar[d, "{\alpha(j)}"'] \ar[rr, "{F(r)}"] &&  F(k) \ar[d, "{\alpha(k)}"] \\
G(j)\ar[rr, "{G(r)}"]   && G(k)
\end{tikzcd}
\]
is cartesian for every morphism $r:j\to k$ in $\cK$. 
If $\gamma(F):F\to  \delta\colim(F)$ and $\gamma(G):G\to \delta\colim(G)$ are colimit cones, then the square
\begin{equation}
\label{comsq50078}
\begin{tikzcd}
F(k)\ar[rr,"\gamma(F)(k)"] \ar[d, "\alpha(k)"'] && \colim(F)  \ar[d, "\colim(\alpha)"] \\
G(k) \ar[rr,"\gamma(G)(k)"] && \colim(G)
\end{tikzcd}
\end{equation}
commutes for every object $k\in \cK$.
The colimits in $\cE$ are said to be {\it effective} if for any {\it cartesian} transformation $\alpha:F\to G$ the square \eqref{comsq50078} is cartesian for every object $k\in \cK$.

\medskip

\begin{definition}[Topos \cite{Rezk:topos}]
\label{Rezkdescent} 
If a category $\cE$ is cocomplete and finitely complete, then the {\it descent principle} is said to {\it hold} in $\cE$ if every colimit in $\cE$ is universal and effective.
A presentable category $\cE$ is said to be a {\it topos} if the descent principle holds in $\cE$.
\end{definition}

Examples of topoi are the category $\cS$ of spaces, the category of diagrams $[\cC,\cE]$ where $\cC$ is a small category and $\cE$ a topos, as well as the slice category $\cE\slice{A}$ for any object $A$ of a topos $\cE$.
In particular, any presheaf category $\Prsh \cK$ is a topos. 

\medskip
We shall say that an accessible reflection in the sense of \cref{def:accessible} is left-exact if the corresponding relfector is left-exact.

\begin{theorem}[\cite{HTT,Rezk:topos,TV05}]
\label{Rezkcharacterization}
A presentable category $\cE$ is a topos if and only if it is an accessible left-exact reflection of a presheaf category.
\end{theorem}

\medskip

We shall need in the proof of \cref{univlexloc} the following formulation of the descent principle.
Let $\CAT$ be the category of categories as defined in~\cref{sec:convents}.
If $\cE$ is a category with finite limits, consider the functor $T:\cE\op\to \CAT$ which takes an object $A\in \cE$ to the category $T(A)=\cE\slice{A}$ and takes a map $u:A\to B$ in $\cE$ to the base change functor $T(u):=u^\star:\cE\slice{B}\to \cE\slice{A}$.

\begin{theorem}[{\cite[Theorem 6.1.3.9]{HTT}}]
\label{descenttheorem} 
A presentable category $\cE$ is a topos if and only if the contravariant functor  $T:\cE\op\to \CAT$ takes colimits to limits.
\end{theorem}

\subsubsection{Surjective families and local classes}
\label{sec:cover-local-class}

\begin{definition}[Surjection]
\label{defsurjcoverage}
We shall say that a map $f: X \to Y$ in a topos $\cE$ is \emph{surjective}, or that $f$ is a \emph{surjection}, if the base change functor $ f^\star: \cE\slice{Y} \to \cE\slice{X} $ is conservative.
We shall say that a family of maps $\{f_i : X_i \to Y \}_{i\in I}$ is  \emph{surjective} if the total map 
\[
(f_i):\bigsqcup_{i\in I} X_i\to Y
\]
is surjective.
Equivalently, a family of maps $\{f_i : X_i \to Y \}_{i\in I}$ is surjective if and only if the total base change functor 
\[ (f^\star_i):
\cE\slice{Y} \to \prod_{i \in I} \cE\slice{X_i}
\]
is conservative.
\end{definition}

\begin{rem}
Surjective maps are called \emph{effective epimorphisms} in \cite{HTT} and \emph{covers} in \cite{ABFJ2}.
\end{rem}

\begin{exmps} \label{exempsurjective} 
Some useful examples of the previous definition are the following:
\begin{exmpenum}
\item \label{exsurjection} A map between two spaces $f:X\to Y$ is surjective if and only if the map $\pi_0(f):\pi_0(X)\to \pi_0(Y)$ is surjective.
A pointed object $(X,x)$ in a topos $\cE$ is connected if and only if the map $x:1\to X$ is surjective.
\item \label{spacecoverage} If $X \in \cS$ is a space, then the family of its points $\{x : 1 \to X\}_{x \in X}$ is surjective.
\item \label{colimcoverage}
If $F : \cK \to \cE$ is a diagram in a topos $\cE$, then the set of inclusions
\[
\{ \gamma(k) : F(k) \to \colim(F) \}_{k \in Ob(\cK)}
\]
is surjective.
\end{exmpenum}
\end{exmps}

If $f:A\to B$ is a map in a topos $\cE$.
We define the {\it nerve} of $f$ to be the simplicial diagram in $N(f):\Delta\op \to \cE\slice B$
sending $[n]$ to $(A,f)^{\times n+1}$, the $(n+1)$-iterated product of $(A,f)$ in $\cE\slice B$ (i.e. the iterated fiber product over $B$ in $\cE$).
The colimit of $N(f)$ is denoted $\im f$ and called the {\it image} of $f$.
A map $f:X\to Y$ in $\cE$ is a {\it monomorphism} if the square
\[
\begin{tikzcd}
X \ar[r, equal] \ar[d, equal] & X \ar[d, "f"] \\
X\ar[r, "f"'] \ar[r] & Y
\end{tikzcd}
\]
is a pullback.
It can be proven that the image of $f$ is a monomorphism \cite[Proposition 6.2.3.4]{HTT}.

\begin{proposition}[{\cite[6.2.3]{HTT}, \cite[Lecture 4]{Rezk:course}}]
\label{prop:charac-surjection}
If $f:A\to B$ is a map in a topos $\cE$. 
The following conditions are equivalent
\begin{enumerate}
\item $f$ is surjective;
\item the colimit of $N(f)$ is terminal in $\cE\slice B$;
\item $f$ is left orthogonal to all monomorphisms in $\cE$.
\end{enumerate}
\end{proposition}

\begin{cor}
\label{cor:naturality-surjection}
Let $\phi:\cE\to \cF$ be a cocontinuous left-exact functor between topoi.
Then for any surjective map $f\in \cE$, the map $\phi(f)$ is surjective.
\end{cor}
\begin{proof}
The functor $\phi$ being cocontinuous and left exact, it preserves the construction of the nerve and its colimit.
\end{proof}

\begin{definition}[Local class {\cite[Proposition 6.2.3.14]{HTT}}]
\label{defn:localclass}
We shall say that a class of maps $\cA$ in a topos $\cE$ is {\it local} if the following two conditions hold:
\begin{enumerate}[label=\roman*)]
\item \label{defn:localclass:1} $\cA$ is closed under base change;
\item \label{defn:localclass:2} If $\{g_i : A_i \to B \}_{i\in I}$ is a surjective family of maps in $\cE$ and the base change $A_i\times_B X\to A_i$ of a map $f:X\to B$ along every map $g_i: A_i \to B$ belongs to $\cA$, then $f$ belongs to $\cA$.
\end{enumerate}
We shall refer to condition \eqref{defn:localclass:2} of the definition by saying that a local class \emph{descends along surjective families}.
\end{definition}

\begin{exmps} \quad
\phantom{  } 
\begin{exmpenum}
\item \label{lsoislocal} The class of isomorphisms in a topos $\cE$ is local: if a family $\{g_i : A_i \to B \}_{i\in I}$ is surjective and the base change map $g_i^\star(X)\to A_i$ of a map $f:X\to B$ is invertible for every $i\in I$, then $f$ is invertible.
\item \label{localclassesinSpaces} In view of \cref{spacecoverage}, if $\cA$ is a local class in the category of spaces $\cS$, then a map $f:X\to Y$ belongs to $\cA$ if and only if the map $f^{-1}(y)\to 1$ belong to $\cA$ for every $y:1\to Y$.
In particular, a map $f:X\to Y$ in $\cS$ is invertible if and only if the map $f^{-1}(y)\to 1$ is invertible for every point $y\in Y$.
\end{exmpenum}
\end{exmps}

\section{Factorization systems and modalities}
\label{sec:hs-and-lex-locs}

\subsection{Factorization systems and saturated classes}
\label{sec:fact-syst}

\begin{definition}
\label{defintionorth}
Let $u:A\to B$ and $f:X\to Y$ be two maps in a category $\cE$.
We say that $u$ is \emph{left orthogonal} to $f$, or that $f$ is \emph{right orthogonal} to $u$, if every commutative square
\[
\begin{tikzcd}
A \ar[rr,"x"] \ar[d,"{u}"']  && X \ar[d,"f"] \\
B \ar[rr,"y"]  && Y
\end{tikzcd}
\]
has a unique diagonal filler $B\to X$.
We shall denote this relation by $u\upvdash f$. 
Equivalently, the condition $u\upvdash f$ means that the square 
\begin{equation}\label{eq:def-orthogonality}
\begin{tikzcd}
\Map B X \ar[rr,"{\Map u X}"] \ar[d,"{\Map B f}"']  && 
\Map A X \ar[d,"{\Map A f}"] \\
\Map B 1 \ar[rr,"{\Map u Y}"]  && {\Map A Y}
\end{tikzcd} \,
\end{equation}
in $\cS$ is cartesian.
If $\cE$ has a terminal object $1$, we shall say that an object $A\in \cE$ is {\it left orthogonal} to a map $f:X\to Y$, and write $A\upvdash f$, if the map $A\to 1$ is left orthogonal $f$.
We shall say that a map $u:A\to B$ is {\it left orthogonal} to an object $X\in \cE$, and write $u\upvdash X$, if $u$ left orthogonal to the map $X\to 1$.
Finally, we shall say that an object $A\in \cE$ is {\it left orthogonal} to an object $X\in \cE$, and write $A\upvdash X$, if the map $A\to 1$ is left orthogonal the map $X\to 1$.
\end{definition} 

If the category $\cE$ is cartesian closed then the morphism spaces in Diagram~\eqref{eq:def-orthogonality}, given by the enrichement of $\cE$ over spaces, can be replaced by the internal hom of $\cE$. In this way, one obtains a stronger relation that we call {\it internal orthogonality} (see~\cite[Definition 3.2.5]{ABFJ1}). Here it will be used only in \cref{orthexmp2}.

\begin{exmps}\quad
\phantom{  }
\begin{exmpenum}
\item \label{orthexmp1} Let $S^0=1\sqcup 1$ be the $0$-sphere in the category of spaces $\cS$.
Then a map $f:X\to Y$ in $\cS$ is right orthogonal to the map $S^0\to 1$ if and only if the following square is cartesian
\[
\begin{tikzcd}
X \ar[r, "\Delta(X)"] \ar[d,"f"'] & 
X\times X \ar[d, "f\times f"] \\
Y\ar[r, "\Delta(Y)"] & Y\times Y
\end{tikzcd}
\]
if and only if the map $f:X\to Y$ is a {\it monomorphism}.
Equivalently, $f:X\to Y$ is a monomorphism if it is isomorphic to the inclusion $Y'\subseteq Y$ of a union $Y'$ of connected components of $Y$.
 
\item  \label{orthexmp2} 
In the category of small categories $\Cat$, consider the inclusion $u:S^0\to [1]$ of the 0-sphere $S^0=1\sqcup 1$ into the category $[1]=\{0<1\}$. 
Then a functor $\phi:\cX\to \cY$ is internally right orthogonal to $u$ if and only if the square
\[
\begin{tikzcd}
\cX\arr \ar[r, "{(s,t)}"] \ar[d,"\phi\arr"'] & 
\cX\times \cX \ar[d, "\phi\times \phi"] \\
\cY\arr \ar[r, "{(s,t)}"] & \cY\times \cY
\end{tikzcd}
\]
is cartesian if and only if $\phi$ is {\it fully faithful}.
The same observation can be made for functors in the category of large categories $\CAT$.

\item  \label{orthexmp3} 
Still in the category $\Cat$, we consider the projection $p:[1]\to 1$ where $[1]=\{0<1\}$.
Then a functor $\phi:\cX\to \cY$ is internally right orthogonal to $p$ if and only if the square
\[
\begin{tikzcd}
\cX \ar[r, "id_-"] \ar[d,"\phi"']
& \cX\arr \ar[d, "\phi\arr"] \\
\cY \ar[r, "id_-"] & \cY\arr
\end{tikzcd}
\]
is cartesian, if and only if $\phi$ is {\it conservative}.

\end{exmpenum}
\end{exmps}

\begin{lemma}
\label{proporthvslocal}
Let $\cE$ be a category with a terminal object $1$.
Then an object $X\in \cE$ is right orthogonal to a map $u:A\to B$ if and only if $X$ is local with respect to the map $u$.
\end{lemma}

\begin{proof}
A map $u:A\to B$ is left orthogonal to the map $p:X\to 1$ if and only if the square
\[
\begin{tikzcd}
{\Map B X}\ar[rr,"{\Map u X}"] \ar[d,"{\Map B p}"']  && 
{\Map A X}    \ar[d,"{\Map A p}"] \\
{\Map B 1}\ar[rr,"{\Map u 1}"]  && {\Map A 1}
\end{tikzcd}
\]
is cartesian.
But the map $\Map u 1$ is invertible, since $\Map B 1 = 1 = \Map A 1$.
Hence the square is cartesian if and only if the map $\Map u X$ is invertible.
This shows that $u\upvdash p$ if and only if the object $X$ is local with respect to the map $u$.
\end{proof}

The following lemma is left to the reader.

\begin{lemma} \label{orthadj}
Let $\phi:\cA\rightleftarrows \cB:\psi$ be a pair of adjoint functors $\phi\dashv \psi$.
If $u$ is a map in $\cA$ and $f$ is a map in $\cB$, then $u\upvdash \psi(f)\Leftrightarrow \phi(u)\upvdash f$.
\end{lemma}

If $\cA$ and $\cB$ are two classes of maps in a category $\cE$, we shall write $\cA\upvdash \cB$ to mean that we have $u\upvdash v$ for every $u\in \cA$ and $v\in \cB$.  We shall denote by $\cA^\upvdash$ (resp.  $ {}^\upvdash\cA$ ) the class of maps in $\cE$ that are right orthogonal (resp. left orthogonal) to every map in $\cA$.
We have
\[
\cA\subseteq {}^\upvdash\cB 
\quad \Leftrightarrow \quad
\cA\upvdash \cB 
\quad \Leftrightarrow  \quad
\cA^\upvdash \supseteq \cB
\]

Recall that a class of objects $\cU$ (resp. a full subcategory $\cU$) in a category $\cE$ is said to be {\it replete} if every object of $\cE$ which is isomorphic
to an object in $\cU$ belongs to $\cU$.
Recall that if $\cA$ is a class of maps in a category $\cE$, then $\underline{\cA}$ denotes the full subcategory of $\cE\arr$ whose objects are the maps in $\cA$. 
We shall say that a class of maps $\cA$ is  {\it replete} if the full subcategory $\underline{\cA}$ is replete.
Remember that we say that $\cA$ is \emph{closed under limits} (resp. \emph{closed under colimits}) is the full subcategory $\underline{\cA}\subseteq \cE\arr$ is closed under limits (resp. closed under colimits).
The closure of $\cA$ under finite limits, or finite colimits, etc. is defined similarly.

\begin{lemma}[{\cite[Proposition 5.2.8.6]{HTT}}]
\label{omnibus0} 
Let $\cA$ be a class of maps in a category $\cE$.
\begin{enumerate}
\item The class $\cA^{\upvdash}$ is replete.
\item The class $\cA^{\upvdash}$ contains the isomorphisms and is closed under composition.
\item The class $\cA^{\upvdash}$ is closed under base change.
\item The class $\cA^{\upvdash}$ is left cancellable ($vu,v\in \cA^{\upvdash} \Rightarrow u\in \cA^{\upvdash}$).
\item The class $\cA^{\upvdash}$ is closed under limits.
\end{enumerate}
\end{lemma}

\begin{definition}
\label{defFactSystem}
A pair $(\cL, \cR)$ of classes of maps in a category $\cE$ is said to be a \emph{factorization system} if the following three conditions hold:
\begin{enumerate}[label=\roman*)]
\item the classes $\cL$ and $\cR$ are replete;
\item  $\cL\upvdash \cR$;
\item every map $f:X\to Y$ in $\cE$ admits a factorization $f=pu:X\to E\to Y$ with $u\in \cL$ and $p\in \cR$.
\end{enumerate}
We shall see in \cref{compdeffact} below that these conditions imply $\cR=\cL^{\upvdash}$ and $\cL={}^{\upvdash}\cR$.
\end{definition}

The class $\cL$ of a factorization system $(\cL,\cR)$ is said to be the \emph{left class} of the factorization system and the class $\cR$ to be the \emph{right class}.
The data of a factorization $f=pu:X\to E\to Y$ with $u\in \cL$ and $p\in \cR$ is unique, and we therefore write $\cL(f):=u$, $ \cR(f):=p$ and $\|f\|:=E$ for a specific choice depending on $f$ (the context will make clear with respect to which factorization $\|f\|$ is defined).
\[
\begin{tikzcd}
X\ar[dr, "\cL(f)"'] \ar[rr, "f"] & & Y  \\
&\|f\| \ar[ur, "\cR(f)"'] &
\end{tikzcd}
\]
Moreover, the operations $f\mapsto \cL(f)$, $f\mapsto \cR(f)$ and $f\mapsto \|f\|$ are \emph{functorial} \cite[Proposition 5.2.8.17.(2)]{HTT}.
When the category $\cE$ has a terminal object $1$, the we shall simply write $\|X\|$ for $\|X \to 1\|$.

\begin{exmps}\quad
\phantom{  }
\begin{exmpenum}

\item \label{factsysexemp0} In a category $\cE$, if $\All$ denotes the class of all maps and $\Iso$ denotes the class of isomorphisms, then the pair $(\All, \Iso)$ and the pair $(\Iso,\All)$ are (trivial) examples of factorization systems.

\item \label{factsysexemp1} If $\Mono$ is the class of monomorphisms in $\cS$ (see \cref{orthexmp1}) and $\Surj$ is the class of surjective maps in $\cS$ (see \cref{exsurjection}) then the pair $(\Surj,\Mono)$ is a factorization system in $\cS$.

\item \label{factsysexemp2} More generally, if $\Mono$ is the class of monomorphisms in $\cE$ and $\Surj$ is the class of surjective maps (see \cref{sec:cover-local-class}), then the pair $(\Surj,\Mono)$ is a factorization system in $\cE$ \cite[Proposition 6.2.3.4]{HTT}.
Thus, every map $f:X\to Y$ in $\cE$ admits a factorization $f=up:X\to J\to Y$ with $p:X\to J$ a surjective map and $u: J\to Y$ a monomorphism,
\[
\begin{tikzcd}
X\ar[dr, "p"'] \ar[rr, "f"] &&Y  \\
&J\ar[ur, "u"'] &
\end{tikzcd}
\]
The subobject $(J,u)$ of $Y$ is the {\it image} of $f$.
It can be computed as the colimit of the nerve of $f$ (see \cref{sec:cover-local-class}).

\item \label{factsysexemp4} Recall from \cite[3.3]{ABFJ2} that an object $A$ in a topos $\cE$ is {\it $n$-truncated} if the diagonal map $A\to A^{S^{n+1}}$ is invertible. The full subcategory of $n$-truncated objects $\cE^{\leq n}$ is reflective and the reflector $\tau_n:\cE \to \cE^{\leq n}$ takes an object $X$ to its {\it $n$-truncation} $\tau_n(X)$. An object $X\in \cE$ is said to be {\it $n$-connected} if $\tau_n(X)=1$. An object $A$ is $n$-connected if and only if the diagonal map $A\to A^{S^k}$ is surjective for every $-1\leq k\leq n$ ($S^{-1} =\emptyset$) \cite[Proposition 6.5.1.18]{HTT}. We shall say that a map $f:A\to B$ in $\cE$ is {\it $n$-truncated} (resp. {\it $n$-connected}) if the object $(A,f)$ of $\cE\slice{B}$ is $n$-truncated (resp. $n$-connected). If $\Trunc_n$ (resp. $\Conn_n$) denotes the class of $n$-truncated maps (resp. $n$-connected maps), then the pair $(\Conn_n,\Trunc_n)$ is a factorization system \cite[Proposition 3.3.6]{ABFJ2}.

{\bf Warning.} An $n$-connected map in our sense is $(n+1)$-connected in the conventional topological indexing and is called $(n+1)$-connective in \cite{HTT}.
 
\item \label{factsysexemp3} Recall that a functor $\phi:\cC\to \cD$ is said to be {\it essentially surjective} if for every object $B\in \cD$ there exists an object $A\in \cC$ together with an isomorphism $\phi(A)\simeq B$.
If $\mathsf{EssSur}$ is the class of essentially surjective functors in the category $\Cat$ 
and $\mathsf{FullFaith}$ is the class of fully faithful functors, then the pair 
$(\mathsf{EssSur},\mathsf{FullFaith})$ is a factorization system in the category $\Cat$.
The same factorization system also exist in $\CAT$

\item \label{factsysexemp5} A functor $\phi:\cC\to \cD$ between small categories is said to be {\it a long localization} if it is a $\omega$-chain of localizations.
If $\mathsf{LongLoc}$ is the class of long localizations in the category $\Cat$ and $\mathsf{Cons}$ is the class of conservative functors (\cref{orthexmp3}), then the pair $(\mathsf{LongLoc},\mathsf{Cons})$ is a factorization system in the category $\Cat$ \cite{Joy}.

\end{exmpenum}
\end{exmps}

The proof of the following lemma is left to the reader.

\begin{lemma}
\label{inducedfactsystem}
A factorization system $(\cL,\cR)$ in a category $\cE$ induces a factorization system $(\cL\slice{A},\cR\slice{A})$ in the category $\cE\slice{A}$ for every object $A\in \cE$.
By definition, a morphism $f:(X,p)\to (Y,q)$ in $\cE\slice{A}$ belongs to $\cL\slice{A}$ (resp. $\cR\slice{A}$) if and only if the map $f:X\to Y$ belongs to $ \cL$ (resp. $\cR$).
\end{lemma}

\begin{lemma}
\label{compdeffact}
If $(\cL, \cR)$ is a factorization system in a category $\cE$, then $\cR=\cL^{\upvdash}$ and $\cL={}^{\upvdash}\cR$.
\end{lemma}

\begin{proof}
We have $\cR\subseteq \cL^{\upvdash}$, since $\cL\upvdash \cR$. Conversely, if a map $f:X\to Y$ belongs to  $\cL^{\upvdash}$, let us show that $f\in \cR$.
For this, we choose a factorization $f=pu:X\to E\to Y$ with $u\in \cL$ and $p\in \cR$.
The following square has a unique diagonal filler $d:E\to X$, since we have $u\upvdash f$ by the hypothesis on $f$.
\[
\begin{tikzcd}
X\ar[r, "{1_X}"] \ar[d, "u"'] &X  \ar[d, "f"]  \\
E \ar[r, "p"'] & Y
\end{tikzcd}
\]
Let us see that $d$ is invertible. 
For this, it suffices to show that $ud=1_E$, since $du=1_X$.
But $ud$ is a diagonal filler of the square
\[
\begin{tikzcd}
X\ar[r, "{u}"] \ar[d, "u"'] &E  \ar[d, "p"]  \\
E \ar[r, "p"'] & Y
\end{tikzcd}
\]
since $udu=u$ and $pud=fd=p$.
This square has a unique diagonal filler since $u\upvdash p$, thus $ud=1_E$.
This shows that $d$ is invertible, and hence that the object $f:X\to Y$ of the arrow category $\cE\arr$ is isomorphic to the object $p:E\to Y$.
So $f\in \cR$, since $p\in \cR$ and the class $\cR$ is replete.
This proves $\cR=\cL^{\upvdash}$.
The proof of $\cL={}^{\upvdash}\cR$ is similar.
\end{proof}
  
\begin{proposition}
\label{factsysfromlex}
Let $\cE$ be a category with finite limits and $\phi:\cE\to \cE'\subseteq \cE$ be a left-exact reflector.
Then the class of maps $\cL_\phi\subseteq \cE$
inverted by the functor $\phi$ is the left class of a factorization system $(\cL_\phi,\cR_\phi)$.
If $\eta:Id \to \phi$ is the unit of the reflection, then a  map $f:X\to Y$ belongs to $\cR_\phi$ if and only if the naturality square 
\begin{equation}
\label{squareforR}
\begin{tikzcd}
X\ar[d, "f"'] \ar[r, "\eta(X)"] &\phi(X) \ar[d, "\phi(f)"]  \\
Y \ar[r, "\eta(Y)"] & \phi(Y)
\end{tikzcd}
\end{equation}
is cartesian.
\end{proposition}  

\begin{proof}
Let $u:A\to B$ be a map in $\cL_\phi$ and $f:X\to Y$ be a map in $\cR_\phi$.
We first prove that $u\upvdash f$.
The functor $\phi:\cE\to \cE'$ is left adjoint to the inclusion $\mathsf{in}:\cE'\hookrightarrow \cE$.
We have $\phi(u)\upvdash \phi(f)$, since $\phi(u)$ is invertible.
It follows by \cref{orthadj} that we have $u\upvdash \mathsf{in}(\phi(f))$.
But the map $f:X\to Y$ is a base change of the map $\mathsf{in} (\phi(f))$, since the square \eqref{squareforR} is cartesian.
Hence we have  $u\upvdash f$ by \cref{omnibus0}.

Let us now show that every map $f:X\to Y$ in $\cE$ admits a factorization $f=pv:X\to E\to Y$ with $v\in \cL_\phi$ and $p\in \cR_\phi$.
Let $v:X\to E$ be the cartesian gap map of the square \eqref{squareforR},
\begin{equation} \label{gapmapfact}
\begin{tikzcd}
X \ar[dr, "v"] \ar[drr, "{\eta(X)}", bend left] \ar[ddr, "{f}"', bend right] && \\
& E\ar[r, "p_2"] \ar[d, "p_1"'] \pbmark & \phi(X) \ar[d, "{\phi (f)}"] \\
& Y \ar[r, "{\eta(Y)}"'] & \phi(Y)
\end{tikzcd}
\end{equation}
Since $f=p_1v$, it is enough to prove that $v\in \cL_\phi$ and $p_1\in \cR_\phi$.
The functor $\phi$ preserves cartesian gap maps because it is left-exact.
Hence the map $\phi(v)$ is the cartesian gap map of the image of the square \eqref{squareforR} by the functor $\phi$.
The maps $\phi(\eta(X))$ and $\phi(\eta(Y))$ are invertible, since the natural transformation $\eta:Id \to \phi$ is the unit of the reflection.
Hence the image of the square \eqref{squareforR} by $\phi$ is cartesian. 
Therefore $\phi(v)$ is invertible and hence $v\in \cL_\phi$.
We are left to prove that $p_1\in \cR_\phi$.
Observe that $\phi(p_2)$ is invertible, since both $\phi(v)$ and $\phi(p_2)\phi(v)=\phi(p_2v)=\phi(\eta(X))$ are. 
Hence the map $p_2:E\to \phi(X)$ is reflecting the object $E$ into $\cE'$.
We may suppose that $p_2=\eta(E)$ in which case we have $\phi(p_1)=\phi(f)$ since the square in the diagram \eqref{gapmapfact} commutes.
It follows that $p_1\in \cR_\phi$, since the same square is cartesian.
\end{proof}

We record here for future reference a number of standard facts about the closure properties of the right and left classes of a factorization system.  

\begin{proposition}
\label{omnibus}
Let $(\cL, \cR)$ be a factorization system in a category $\cE$.
\begin{enumerate}
\item  $\cL^\upvdash= \cR$ and $\cL={}^\upvdash \cR$.
\item  The classes $\cL$ and $\cR$ contain the isomorphisms and are closed under composition.
\item  The class $\cL$ is closed under cobase change.
\item  The class $\cR$ is closed under base change.
\item  The class $\cL$ is right cancellable ($vu, u\in \cL \Rightarrow v\in \cL$).
\item  The class $\cR$ is left cancellable ($vu, v\in \cR \Rightarrow u\in \cR$).
\item  The class ${\cL}$ is closed under colimits.
\item  The class ${\cR}$ is closed under limits.
\item  The intersection $\cL\cap \cR$ is the class of isomorphisms.
\end{enumerate}
\end{proposition}

\begin{proof}
This follows from \cref{omnibus0,compdeffact}.
The last item follows from the equivalence $(f\upvdash f )\Leftrightarrow (f {\rm\ is\ invertible})$.
See also \cite[Propositions 5.2.8.6 and 5.2.8.11]{HTT}.
\end{proof}

\subsubsection{Saturated classes}
\label{sec:saturated-classes}

\begin{definition}[Saturated class]
Let $\cE$ be a cocomplete category.  We shall say that a class of maps $\cL\subseteq \cE$ is {\it saturated} if the following conditions
hold:
\begin{enumerate}[label=\roman*)]
\item  $\cL$ contains the isomorphisms and is closed under composition;
\item  ${\cL}$ is closed under colimits.
\end{enumerate}
\end{definition}

\begin{definition}
If $\cE$ is a  cocomplete category,  then every class of maps $\Sigma \subseteq \cE$ is contained in a smallest saturated class of maps $\Sigma\sat  \subseteq \cE$.
We shall say that ${\Sigma}\sat$ is the \emph{saturated class generated} by $\Sigma$.
A saturated class of maps $\cL\subseteq \cE$ is of {\it small generation} if $\cL={\Sigma}\sat$ for a set of maps $\Sigma\subseteq \cE$.
\end{definition}

The definition of saturated class given as \cite[Definition 5.5.5.1]{HTT} includes the additional condition that a saturated class be closed under cobase change.
However, the additional condition is automatic: 

\begin{proposition}
\label{saturatedclosedbaseandcancel}
A saturated class $\cL$ is closed under cobase change and it has the right cancellation property: if $u$ and $vu$  belongs to $\cL$, then $v\in \cL$.
\end{proposition}

\begin{proof}
The first statement follows from \cref{saturatedareclosedundercob}.
The second statement follows from \cref{colexrightcancel} below.
\end{proof}

\begin{lemma}[\cite{AS}]
\label{colexrightcancel}
Let $\cE$ be a category with finite colimits and let $\cQ\subseteq \cE$ be a class of maps which contains the isomorphisms.
If the full subcategory $\underline{\cQ}\subseteq \cE\arr$ is closed under finite colimits, then the class $\cQ$ has the right cancellation property.
\end{lemma}

\begin{proof}
Observe that if $u:A\to B$ and $v:B\to C$ are two maps in $\cE$, then the following square in the category $\cE\arr $ is a pushout:
\[
\begin{tikzcd}
u\ar[d, "{(u,1_B)}"'] \ar[rr, "{(1_A,v)}"] && vu \ar[d, "{(u,1_C)}"]  \\
1_B\ar[rr, "{(1_B,v)}"'] && {v}
\end{tikzcd}
\]
We have $1_B \in \underline{\cQ}$, since the class ${\cQ}$ contains the isomorphisms.
Thus, if $u$ and $vu$ belong to $\underline{\cQ}$, then $v\in \underline{\cQ}$ since the full subcategory $\underline{\cQ}\subseteq \cE\arr$ is closed under pushouts.
\end{proof}

\begin{lemma}
\label{leftorthissat}
If $\cB$ is a class of maps in a cocomplete category $\cE$, then the class ${}^\upvdash \cB$ is saturated.
The left class $\cL$ of any factorization system $(\cL,\cR)$ in $\cE$ is saturated.
\end{lemma}

\begin{proof}
The first statement follows from \cref{omnibus0}.
The second statement follows from the first, since $\cL={}^\upvdash \cR$ by \cref{compdeffact}. 
\end{proof}

\begin{lemma}
\label{orthsaturatedclass}
If $\Sigma$ is a class of maps in a cocomplete category $\cE$, then $(\Sigma\sat)^\upvdash=\Sigma^\upvdash$.
\end{lemma}

\begin{proof}
We have $(\Sigma\sat)^\upvdash\subseteq \Sigma^\upvdash$, since $\Sigma\subseteq \Sigma\sat$.
Conversely, the class ${}^\upvdash(\Sigma^\upvdash)$ is saturated by \cref{leftorthissat}.
Thus $\Sigma\sat\subseteq {}^\upvdash(\Sigma^\upvdash)$, since $\Sigma\subseteq {}^\upvdash(\Sigma^\upvdash)$.
It follows that $\Sigma\sat \upvdash \Sigma^\upvdash$ and hence that $\Sigma^\upvdash\subseteq (\Sigma\sat)^\upvdash$.
\end{proof}

\begin{proposition}
\label{genfactsyst}
Let $\Sigma$ be a set of maps in a presentable category $\cE$.
If ${\Sigma}\sat\subseteq \cE$ is the saturated class generated by $\Sigma$, then the pair $({\Sigma}\sat,\Sigma^\upvdash)$ is a factorization system.
\end{proposition}

\begin{proof}
The saturated class ${\Sigma}\sat$ is of small generation, since $\Sigma$ is a set.
Hence the pair $({\Sigma}\sat,({\Sigma}\sat)^\upvdash)$ is a factorization system by \cite[Proposition 5.5.5.7]{HTT}.
By \cref{orthsaturatedclass}, we have $(\Sigma\sat)^\upvdash=\Sigma^\upvdash$ and this prove the statement.
\end{proof} 

We shall say that the factorization system $(\Sigma\sat,\Sigma^\upvdash)$ of \cref{genfactsyst} is \emph{generated} by $\Sigma\subseteq \cE$.

\subsubsection{Reflective subcategories from factorization systems}
\label{sec:fact-syst-refl}

Let $(\cL,\cR)$ be a factorization system in a category $\cE$.
If the category $\cE$ has a terminal object $1\in \cE$, we denote by $\cR[1]$ the full subcategory of $\cE$ spanned by the objects $X$ for which the map $X\to 1$ belongs to $\cR$.
In general, for $X\in \cE$, we denote by $p_X$ the unique map $X\to 1$, and write $\| X\|$ for the object $\|p_X\|$ and we define $\eta(X):=\cL(p_X):X\to \| X\|$.
\[
\begin{tikzcd}
X\ar[dr,"p_X"'] \ar[r, "\eta(X)"] & \| X\| \ar[d] \\
& 1 
\end{tikzcd}
\]
We have $\| X\| \in \cR[1]$, since  $\cR(p_X)\in \cR$.

\begin{lemma}
\label{reflectionR(1)} 
Let $(\cL,\cR)$ be a factorization system in a category $\cE$ with a  terminal object $1$.
The following properties hold:
\begin{enumerate}
\item \label{reflectionR(1):1}  The map $\eta(X):X\to \| X\|$ reflects the object $X\in \cE$ into $\cR[1]$.
\item \label{reflectionR(1):2} The full subcategory $\cR[1]\subseteq \cE$ is reflective and if $\| -\|:\cE\to \cR[1]$ is the reflector, then the square
\begin{equation}
\label{reflectorfactsystem}
\begin{tikzcd}
X\ar[d, "f"'] \ar[rr, "{\eta(X)}"] &&   \| X\|  \ar[d, " {\|f\|}"]  \\
Y \ar[rr, "{\eta(Y)}"] &&   \|Y\|
\end{tikzcd}
\end{equation}
commutes for every map $f:X\to Y$.
\item \label{reflectionR(1):3} A map $r: X\to X'$ in $\cE$ reflects an object $X\in \cE$ into $\cR[1]$ if and only if $r\in \cL$ and $X'\in \cR[1]$.
\item \label{reflectionR(1):4}  An object  $X\in \cE$ belongs to $\cR[1]$ if and only if it is $\cL$-local.
\item \label{reflectionR(1):5} Every map in $\cR[1]$ belongs to $ \cR$.
\item \label{reflectionR(1):6} The functor $ \| -\| :\cE\to \cR[1]$ inverts every map in $\cL$  and $\cR[1]=\LOC \cE \cL$.
\end{enumerate}
\end{lemma}

\begin{proof}
\eqref{reflectionR(1):1}
Let us first show that if a map $r:X\to X'$ belongs to $\cL$ and $X'\in \cR[1]$, then $r$ reflects the object $X$ into the full subcategory $\cR[1]$. 
If $Y\in \cR[1]$, then for every map $f:X\to Y$ the following commutative square has a unique diagonal filler   
\[
\begin{tikzcd}
X\ar[d,"{r}"'] \ar[rr, "{f}"]   && Y  \ar[d,"p_Y "]      \\
X'  \ar[rr, "{p_{X'}}"] &  & 1
\end{tikzcd}
\]
since $r\upvdash p_Y$.
This shows that there exists a unique $g:X'\to Y$ such that $gr=f$. 
Thus, $r$ reflects the object $X$ into the full subcategory $\cR[1]$. 
In particular, the map $\eta(X):=\cL(p_X):X\to \| X\|$ reflects the object $X$ into the full subcategory $\cR[1]$.

\smallskip
\noindent \eqref{reflectionR(1):2}
It follows that the full subcategory $\cR[1]$ is reflective. 
Moreover, the square \eqref{reflectorfactsystem} commutes by construction of the reflector $\| -\|:\cE\to \cR[1]$.

\smallskip
\noindent \eqref{reflectionR(1):3}
It remains to show that if a map $r:X\to X'$ reflects an object $X$ into $\cR[1]$, then $r\in \cL$. But we saw above that the map $\eta(X):=\cL(p_X):X\to \| X\|$ is reflecting the object $X$ into $\cR[1]$.
It follows that there exists an isomorphism $h:\| X\|\to X'$ such that $h\eta(X)=r$. 
Thus, $r\in \cL$, since $\eta(X)\in \cL$.

\smallskip
\noindent \eqref{reflectionR(1):4}
By definition, an object $X\in \cE$ belongs to $\cR[1]$ if and only if the map $p_X:X\to 1$ belongs to $\cR={\cL}^\upvdash$.  
But the map $p_X:X\to 1$ belongs to $\cL^\upvdash$ if and only if the object $X$ is $\cL$-local by \cref{proporthvslocal}.

\smallskip
\noindent \eqref{reflectionR(1):5}
If $f:X\to Y$ is a map in $\cR[1]$ then the maps $p_X:X\to 1$ and $p_Y:Y\to 1$ belongs to $\cR$.
It follows that $f\in \cR$, since $p_Yf=p_X$ and the class $ \cR$ is left cancellable by \cref{omnibus}.

\smallskip
\noindent \eqref{reflectionR(1):6}
If $f:X\to Y$ belongs to $\cL$, let us show that the map $ \|f\|:\|X\|\to \|Y\|$ is invertible.  
We have $\|f\| \eta(X)=\eta(Y)f$, since the square \eqref{reflectorfactsystem} commutes.
But we have $\eta(Y)f \in \cL$, since $f\in \cL$ and $\eta(Y)\in \cL$.
Thus, $ \|f\| \eta(X)=\eta(Y)f\in \cL$.
It follows that $\|f\|\in \cL$, since $\eta(X)\in \cL$ and the class $\cL$ is right cancellable by \cref{omnibus}.
Thus, $ \| f\|\in \cL\cap \cR$, since we have $\| f\|\in \cR$ by \eqref{reflectionR(1):5}.
This shows that $ \| f\|$ is invertible by \cref{omnibus}.
Let us see now that $\cR[1]=\LOC \cE \cL$.
Let $\Sigma\subseteq \cE$ be the class of maps $\eta(X):X\to \|X\|$ for $X\in \cE$.
The maps in $\Sigma$ are units of the reflector $ \| -\| :\cE\to \cR[1]$ by \eqref{reflectionR(1):1}. 
Thus, the reflector $ \| -\| :\cE\to \cR[1]$ is a localization with respect to $\Sigma$ by \cref{prop:loc-vs-reflection}.
It is therefore a localization with respect to $\cL$, since $\Sigma\subseteq \cL$ by \eqref{reflectionR(1):3} and we just saw that the reflector $ \| -\| :\cE\to \cR[1]$ inverts every map in $\cL$.
\end{proof}

In general, if $(\cL,\cR)$ is a factorization system in a category $\cE$, then for every object $A\in \cE$ we shall denote by $\cR[A]$ the full subcategory of $\cE\slice{A}$ spanned by the objects $X=(X,f)$ with a structure map $f:X\to A$ in $\cE$.
Let us put $\|X\|_A=(\|f\|,\cR(f))\in \cR[A]$. 
\[
\begin{tikzcd}
X\ar[dr,"f"'] \ar[r, "\cL(f)"] &\|X\|_A \ar[d,"{\cR(f)}"] \\
& 1 
\end{tikzcd}
\]
By \cref{inducedfactsystem}, the factorization system $(\cL,\cR)$ induces a factorization system $(\cL\slice{A},\cR\slice{A})$ in the category $\cE\slice{A}$.
Moreover, $\cR\slice{A}(1_A)=\cR[A]$.
It then follows from \cref{reflectionR(1)} that the full subcategory $\cR[A]\subseteq \cE\slice{A}$ is reflective:
for every object $X=(X,f)\in \cE\slice{A}$ the map $\cL(f):X\to \| X\|_A$ reflects the object $X=(X,f)$ in the full subcategory $\cR[A]$.
We shall denote the reflector by
\begin{equation}
\label{reflectorA}
\|-\|_A:\cE\slice{A}\to \cR[A]
\end{equation}
The properties of the functor $\|-\|_A$ can be deduced from \cref{reflectionR(1)}.

\begin{proposition}
\label{lemmaR1pres}
Let $(\cL,\cR) = (\Sigma\sat,\Sigma^\perp)$ be the factorization system generated by a set of maps $\Sigma$ in a presentable category $\cE$.
The reflection $\cE\to \Loc \cE \Sigma$ of \cref{localizationofpresentable} is equivalent to the reflection $\|-\|:\cE\to \cR[1]$ associated to the  factorization system $(\cL,\cR)$.
In particular, the category $\cR[1]$ is presentable.
\end{proposition}

\begin{proof}
The factorization system $({\Sigma}\sat,\Sigma^\upvdash)$ is build in \cref{genfactsyst}.
An object $X\in \cE$ belongs to $\cR[1]$ if and only if the map $X\to 1$ belongs to $\cR = \Sigma^\perp$. 
But this is exaclty the definition of an object local with respect to $\Sigma$.
Thus, $\cR[1]=\Loc \cE \Sigma$ as full subcategories of $\cE$, and the reflection of $\cE$ into them must coincide.
Finally, the category $\Loc \cE \Sigma$ is presentable by \cref{localizationofpresentable}, hence so is $\cR[1]$.
\end{proof}

\begin{proposition}
\label{lem:equiv-ccloc}
For $\Sigma$ a set of maps in a presentable category $\cE$, we have canonical equivalences
\[
\Loc \cE \Sigma
= \LOCcc \cE {\Sigma} 
= \LOC \cE {(\Sigma\ssat)} 
= \LOC \cE {(\Sigma\sat)} 
.
\]
\end{proposition}
\begin{proof}
The equivalences $\LOC \cE {(\Sigma\ssat)} = \LOCcc \cE {\Sigma} = \Loc \cE \Sigma$ are given by \cref{localizationofpresentable}.
By \cref{lemmaR1pres}, $\Sigma\sat$ is the left class of a factorization system and we get the equivalence with $\Loc \cE \Sigma = \LOC \cE {(\Sigma\sat)}$ from \cref{reflectionR(1)}\eqref{reflectionR(1):6}.
\end{proof}

If $u:A\to B$ is a map in a category $\cE$, then the {\it pushforward functor} $u_!:\cE\slice{A}\to \cE\slice{B}$ is defined by putting $u_!(X,f)=(X,uf)$ for every map $f:X\to A$.
If $(\cL,\cR)$ is a factorization system in $\cE$, then the pushforward functor descends to a functor $u_\sharp:\cR[A]\to \cR[B]$ defined by putting $u_\sharp(X,f):=(\|uf\|,\cR(uf)) $ for every map $f:X\to A$ in $\cR$.
\[
\begin{tikzcd}
X\ar[rr, "\cL(uf)"] \ar[d, "f"'] &&\|uf\| \ar[d, "\cR(uf)"] \\
A \ar[rr, "u"] && B
\end{tikzcd}
\]
It follows from the definition that the following square commutes 
\[
\begin{tikzcd}
\cE\slice{A}\ar[rr, "u_!"] \ar[d] &&\cE\slice{B} \ar[d] \\
\cR[A] \ar[rr, "u_\sharp "] && \cR[B]
\end{tikzcd}
\]
where the vertical functors are reflectors.

\medskip

If the category $\cE$ has finite limits, then the base change functor $u^\star:\cE\slice{B}\to \cE\slice{A}$ (defined by putting $u^\star(Y,g)=(A\times_B Y,p_1)$ for every map $g:Y\to B$) is right adjoint to the functor $u_!:\cE\slice{A}\to \cE\slice{B}$.
If $(\cL,\cR)$ is a factorization system in $\cE$, then the functor $u^\star:\cE\slice{B}\to \cE\slice{A}$ induces a base change functor $u^\star:\cR[B]\to \cR[A]$, since the class $\cR$ is closed under base change by \cref{omnibus}.

\begin{proposition}
\label{propusharpleftadjoint} 
Let $(\cL,\cR)$ be a factorization system in a category with finite limits $\cE$.
For any map $u:A\to B$ in $\cE$, the functor $u^\star:\cR[B]\to \cR[A]$ is right adjoint to the functor $u_\sharp:\cR[A]\to \cR[B]$.
\end{proposition}

\begin{proof} 
The adjunction $u_! \dashv u^\star$ is a natural equivalence 
\[
\mathrm{Map}_B(u_!(X),Y) \xto{\simeq} \mathrm{Map}_A(X,u^\star(Y))
\]
for every $X=(X,f) \in \cE\slice{A}$ and $Y=(Y,g) \in \cE\slice{B}$.
If $(Y,g) \in \cR[B]$, then $u^\star(Y,g)\in \cR[A]$.
By \cref{reflectionR(1)}, the map $\cL(uf):u_!(X)\to u_\sharp(X)$ reflects the object $u_!(X)=(X,uf)$ of $ \cE\slice{B}$ into $\cR[B]$.
Hence the map
\[
(-)\circ  \cL(uf):\mathrm{Map}_B(u_\sharp(X),Y) \xto{\simeq} \mathrm{Map}_B(u_!(X),Y)
\]
is invertible.
The adjunction $u_\sharp \dashv u^\star$ follows.
\end{proof}

\begin{corollary}
\label{corollpusharpleftadjoint} 
Let $(\cL,\cR)$ be a factorization system in a category with finite limits $\cE$.
If $u:A\to B$ is a map in $\cE$, then the functor $u_\sharp:\cR[A]\to \cR[B]$ is fully faithful if and only if the square 
\begin{equation}
\label{remusharp}
\begin{tikzcd}
  X\ar[rr, "\cL(uf)"] \ar[d,"f"'] && \|uf\| \ar[d,"\cR(uf)"] \\
  A\ar[rr, "u "] && B
\end{tikzcd}
\end{equation}
is cartesian for every map $f:X\to A$ in $\cR$.
\end{corollary}
  
\begin{proof}
The functor $u_\sharp$ is fully faithful if and only if the unit $X\to u^\star \|uf\|$ of the adjunction $u_! \dashv u^\star$ is invertible for every object $(X,f)\in \cR[A]$.
Since this unit is the cartesian gap map of the square \eqref{remusharp}, it is invertible if and only if \eqref{remusharp} is cartesian.
\end{proof}

\subsection{Modalities and acyclic classes}
\label{sec:modalities}

\begin{definition}[Modality]
\label{defmodality}
Let $\cE$ be a category with finite limits.
We shall say that a factorization system $(\cL, \cR)$ in $\cE$ is a \emph{modality} if its left class $\cL$ is closed under base change.
\end{definition}
 
The right class of a factorization system is always closed under base change by \cref{omnibus}.
Hence \emph{both} classes of a modality $(\cL, \cR)$ are closed under base change and it follows that a factorization system is a modality if and only if the factorization of maps is stable by base change.

\begin{exmps}
\label{exmpmodality}
Let $\cE$ be a topos.

\begin{exmpenum}
\item \label{exmpmodality:1} The factorization systems $(\All,\Iso)$ and $(\Iso,\All)$ of \cref{factsysexemp0} are modalities.
\item \label{exmpmodality:2} The factorization system $(\Surj,\Mono)$ of \cref{factsysexemp2} is a modality.
\item \label{exmpmodality:3} More generally, for any $-2\leq n< \infty$, the factorization system $(\Conn_n,\Trunc_n)$ of \cref{factsysexemp4} is a modality \cite[Example 3.4.2]{ABFJ2}.
\end{exmpenum}
\end{exmps}

Let $\cE$ be a category with finite limits and let $\phi:\cE\to \cE'\subseteq \cE$ be a left-exact reflector. 
We saw in \cref{factsysfromlex} that the class  of maps $\cL_\phi\subseteq \cE$ inverted by the functor $\phi$ is the left class of a factorization system $(\cL_\phi,\cR_\phi)$.

\begin{proposition}
\label{modfromlexreflector}
The factorization system $(\cL_\phi,\cR_\phi)$ of \cref{factsysfromlex} is a modality.
\end{proposition}
  
\begin{proof}
Let us show that the class $\cL_\phi$ is closed under base change.
If the map $g\in \cE$ is a base change of a map $f\in  \cL_\phi$, then the map $\phi(g)\in \cE'$ is a base change of the map $\phi(f)\in \cE'$, since the functor $\phi:\cE\to \cE'$ is left-exact.
But the map $\phi(f)$ is invertible, since $f\in \cL_\phi$.
Hence the map $\phi(g)$ is invertible. 
This shows that $g\in \cL_\phi$.
\end{proof}

\begin{definition}[Fiberwise orthogonality]
\label{deffibleftorth}
Let $\mathcal{E}$ be a category with finite limits.
We shall say that a map $u:A\to B$ in $\mathcal{E}$ is \emph{fiberwise left orthogonal} to another map $f:X\to Y$, or that $f$ is \emph{fiberwise right orthogonal} to $u$, if every base change $u'$ of $u$ is left orthogonal to $f$.
We shall denote this relation by $u\fperp f$.
\end{definition} 

If $\cA$ and $\cB$ are two classes of maps in a category with finite limits $\cE$, we shall write $\cA\fperp \cB$ to mean that we have $u\fperp f$ for every $u\in \cA$ and $f\in \cB$.
We shall denote by $\cA^\fperp$ (resp. ${}^\fperp \cA$) the class of maps in $\cE$ that are fiberwise right orthogonal (resp.  fiberwise left orthogonal) to every map in $\cA$.
We have
\[
\cA\subseteq {}^\fperp\cB 
\quad \Leftrightarrow \quad
\cA\fperp \cB
\quad \Leftrightarrow
\quad \cA^\fperp \supseteq \cB
\]

\begin{definition}[Modal object]
\label{defmodalobject}
If $u:A\to B$ is a map in a topos $\cE$, we shall say that an object $X\in \cE$ is $u$-{\it modal} if it is local with respect to every base change $u'$ of $u$.
If $\Sigma$ is a class of maps in $\cE$, we shall say that an object $X\in \cE$ is $\Sigma$-{\it modal} if it is $u$-modal for every map $u\in \Sigma$, that is if the map $X\to 1$ belongs to $\Sigma^{\fperp}$.
We shall denote by $\Mod \cE \Sigma$ the full subcategory of $\cE$ spanned by $\Sigma$-modal objects.
\end{definition}

\begin{lemma}[{\cite[Proposition 2.6.2]{ABFJ1}}]
\label{fiberwiseorthogonalitylemma}
Let $\cE$ be a category with finite limits.
Then a factorization system $(\cL, \cR)$ in $\cE$ is a modality if and only if $\cL \fperp \cR$, in which case $\cR=\cL^\fperp$ and $\cL={}^\fperp \cR$.
\end{lemma}

Recall from \cref{defn:localclass} that a class of maps is said to be \emph{local} if it is closed under base change and descends along surjective families of maps.
The following proposition says that modalities provide us with a rich source of local classes:

\begin{proposition}[{\cite[Proposition 3.6.5]{ABFJ2}}]
\label{prop-localclass}   
In a topos $\cE$, the left and the right classes of a modality $(\cL,\cR)$ are both local.
\end{proposition}

\begin{proof}
The class $\cR$ is closed under base change by \cref{omnibus}.
Let $(g_i:A_i\to B |\ i\in I) $ be a surjective family of maps.
If the base change $g^\star(f):X_i\to A_i$ of a map $f:X\to B$ belongs to $\cR$ for every $i\in I$, let us show that the map $f$ also belongs to $\cR$.
We choose a factorization $f=pu:X\to E\to Y$ with $p\in \cR$ and $u\in \cL$.
The map $f$ will be in $\cR$ if and only if $u$ is invertible.
We have $g_i^\star(p)\in \cR$, since the class $\cR$ is closed under base change by \cref{omnibus}.
And we have $g_i^\star(u)\in \cL$ since the class $\cL$ is closed under base change, being the left class of a modality.
We have $g_i^\star(f)=g_i^\star(p)g_i^\star(u)$ and $g_i^\star(f)\in \cR$ by hypothesis.
Thus, $g_i^\star(u)$ is invertible by uniqueness of the $(\cL,\cR)$-factorization.
Therefore, $u$ is invertible since the family of maps $(g_i:Y_i\to Y |\ i\in I)$ is surjective (\cref{defsurjcoverage}).
This shows that $f\in \cR$, since $p\in \cR$ and $f=pu$.
This proves that $\cR$ is local.
The proof is similar for $\cL$.
\end{proof}

\subsubsection{Acyclic classes and generation of modalities}
\label{sec:acyclic-classes}

\begin{definition}[Acyclic class]
\label{defacyclicclass}
Let $\cE$ be a topos.
We shall say that a class of maps $\cL\subseteq \cE$ is {\it acyclic} if the following conditions hold:
\begin{enumerate}[label=\roman*)]
\item the class $\cL$ contains the isomorphisms and is closed under composition;
\item the class ${\cL}$ is closed under colimits;
\item the class ${\cL}$ is closed under base change.
\end{enumerate} 
\end{definition}

Equivalently, a class of maps $\cL$ is acyclic if and only if it is saturated and closed under base change.
In particular, every acyclic class is closed under cobase change and has the right cancellation property by \cref{saturatedclosedbaseandcancel}.

\begin{definition}
\label{defcoreant}
If $\cE$ is a topos, then every class of maps $\Sigma\subseteq \cE$ is contained in a smallest acyclic class $\Sigma\ac$. 
We shall say that $\Sigma\ac$ is the acyclic class {\it generated} by $\Sigma$.
We shall say that an acyclic class $\cL\subseteq \cE$ is of {\it small generation} if $\cL=\Sigma\ac$ for a \emph{set} of maps $\Sigma$.
\end{definition}

\begin{rem}
\label{rem:cc}
Every acyclic class is local by \cref{prop-localclass}. 
In the category of spaces $\cS$, the concept of an acyclic class of maps $\cL$ is equivalent to that of a \emph{closed class of spaces} introduced in \cite{DF95}.
In view of \cref{spacecoverage}, a map $f:X\to Y$ belongs to $\cL$ if and only if the map $f^{-1}(y)\to 1$ belongs to $\cL$ for every $y:1\to Y$.  That is, an acyclic class $\cL$ is determined entirely by a class of spaces.
\end{rem}

\begin{proposition}
\label{leftmodacyclic}
A factorization system $(\cL,\cR)$ in a topos $\cE$ is a modality if and only if its left class $\cL$ is acyclic.
\end{proposition}

\begin{proof}
Let $(\cL,\cR)$ be a factorization system in a topos $\cE$.
If $(\cL,\cR)$ is a modality, $\cL$ is saturated by \cref{leftorthissat} and closed under base change by \cref{defmodality}.
Reciprocally, if $\cL$ is acyclic, it is closed by base change and hence $(\cL,\cR)$ is a modality.
\end{proof}

\begin{exmps}
\label{exmp-acyclic}
Let $\cE$ be a topos.
\begin{exmpenum}
\item \label{exmp-acyclic:1} The classes $\Iso$ and $\All$ are acyclic. They are respectively the smallest and largest such classes.
\item \label{exmp-acyclic:2} For $-2\leq n<\infty$, the class $\Conn_n$ of $n$-connected map (in particular the class $\Surj=\Conn_{-1}$) are acyclic since they are left classes of modalities (\cref{exmpmodality:3}).
\item \label{exmp-acyclic:3} Recall from that a map is \oo connected if it is $n$-connected for every $n$. The intersection of acyclic classes is always acyclic. Hence the class $\Conn_\infty = \bigcap_{n}\Conn_n$ of \oo connected maps is acyclic.
\item \label{exmp-acyclic:4} If $\phi:\cE\to \cF$ is a cocontinuous left-exact functor between topoi, and $\cA$ is an acyclic class in $\cF$, the class $\phi^{-1}(\cA) =\{f\in \cE\ |\ \phi(f)\in \cA\}$ is acyclic.
\end{exmpenum}
\end{exmps}

\begin{lemma}
\label{largestacyclic}
In a topos $\cE$, every saturated class $\cL\subseteq \cE$ contains a largest acyclic class $\cA\subseteq \cL$.
A map $u$ belongs to $\cA$ if and only if every base change $u'$ of $u$ belongs to $\cL$.
\end{lemma}

\begin{proof}
Let $\cA$ be the class of maps $u:A\to B$ in $\cE$ having all their base changes in $\cL$.
Obviously, $\cA\subseteq \cL$ and $\cA$ is closed under base change by construction.
We need to show that the class $\cA$ is saturated.
The class $\cA$ contains the isomorphisms and is closed under composition, since this is true of the class $\cL$.
We are left to prove that the full subcategory $\underline{\cA}\subseteq \cE\arr$ is closed under colimits.

If $\cK$ is a small category, then a diagram $\cK\to \cE\arr$ is the same thing as a natural transformation $u:A\to B$ between two diagrams $A,B:\cK\to \cE$.
If the map $u(k):A(k)\to B(k)$ belongs to $\cA$ for every object $k\in \cK$, let us see that the map $\colim(u):\colim A \to \colim B$ belongs to $\cA$. 
Consider the following commutative square of natural transformations
\begin{equation}
\label{comsquare1235}
\begin{tikzcd}
A\ar[rr, "{\gamma(A)}"] \ar[d, "u"'] &&\delta \colim(A) \ar[d, "{\delta \colim(u)}"] \\
B \ar[rr, "{\gamma(B)}"]&& \delta \colim(B)
\end{tikzcd}
\end{equation}
where $\delta \colim(A)$ is the constant diagram $\cK\to \cE$ with value $\colim(A)\in \cE$, and similarly for $\delta \colim(B)$.
The natural transformations $\gamma(A)$ and $\gamma(B)$ are colimit cones.
The map $u(k):A(k)\to B(k)$ belongs to $\cL$ for every object $k\in \cK$, since $\cA\subseteq \cL$.
Thus, $\colim(u) \in \cL$, since the full subcategory $\underline{\cL}$ is closed under colimits.
It remains to show that every base change of the map $\colim(u)$ belongs to $\cL$.
Consider a pullback square
\[
\begin{tikzcd}
D\ar[rr,"{w}"] \ar[d, "{}"'] &&C \ar[d, "{g}"] \\
\colim(A) \ar[rr, "{\colim(u)}"]&& \colim(B)
\end{tikzcd}
\]
and let us show that map $w$ belongs to $\cL$.
By pulling back the square \eqref{comsquare1235} along the map $\delta(g):\delta C\to \delta \colim(B)$ we obtain the following cube of natural transformations, in which every vertical face is a cartesian square.
\[
\begin{tikzcd}[bo column sep=large, row sep=large]
A' \ar[rr, "{u'}"] \ar[dd, dotted] \ar[dr, "{\gamma(A')}"]
& &B' \ar[dd, near end, "{}",dotted] \ar[dr,"{\gamma(B')}"] & \\
& \delta D \ar[rr, crossing over, near start, "\delta(w)"] & & \delta C \ar[dd, "\delta(g)"] \\
A \ar[rr, near end, "u"] \ar[dr, "\gamma(A)"] & & B \ar[dr, "\gamma(B)"]  & \\
& \delta \colim(A) \ar[rr, crossing over, "\delta \colim(u)"'] \ar[from=uu, crossing over, near start, "{}"', dotted] & & \delta \colim(B)
\end{tikzcd}
\]
The natural transformations $\gamma(A'):A'\to \delta D$ and $\gamma(B'):B'\to \delta C$ are colimit cones, since colimit are universal in a topos.
It follows that $w=\colim(u')$.
The map $u'(k):A'(k)\to B'(k)$ is a base change of the map $u(k):A(k)\to B(k)$ for every object $k\in \cK$.
Hence the map $u'(k)$ belongs to $\cL$, since $u(k)$ belongs to $\cA$.
Thus, $w=\colim(u')$ belongs to $\cL$, since the full subcategory $\underline{\cL}\subseteq \cE\arr$ is closed under colimits.
This shows that $\colim(u)\in \cA$ and that $\underline{\cA}\subseteq \cE\arr$ is closed under colimits.
The class $\cA$ is then saturated, therefore acyclic because it is closed under base change.
Finally, we show that $\cA$ is maximal.
If $\cA'\subseteq \cL$ is an acyclic class, we need to show that $\cA'\subseteq \cA$.
If $u:A\to B$ is a map in $ \cA'$, then $\cA'$ contains every base change $u'$ of the $u$, since $\cA'$ is closed under base change.
Thus, $\cL$ contains every base change $u'$ of $u$, since $ \cA'\subseteq \cL$.
Hence we have $u\in \cA$ and this proves that $\cA'\subseteq \cA$.
\end{proof}

\begin{proposition}
\label{leftorthacyclic}
If $\cB$ is a class of maps in a topos $\cE$, then the class ${}^\fperp \cB$ is acyclic.
\end{proposition}

\begin{proof}
The class ${}^\upvdash \cB$ is saturated by \cref{leftorthissat}.
By \cref{largestacyclic} the saturated class ${}^\upvdash \cB$ contains a largest acyclic class $({}^\upvdash \cB)'$.
Moreover, a map $u:A\to B$ belongs to $({}^\upvdash \cB)'$ if and only if every base change $u':A'\to B'$ of $u$ belongs to ${}^\upvdash \cB$.
In other words, $({}^\upvdash \cB)'={}^\fperp \cB$. 
Hence the class ${}^\fperp \cB$ is acyclic.
\end{proof}

\begin{lemma}
\label{orthacyclicclass}
If $\Sigma$ is a class of maps in a topos $\cE$, 
then $\Sigma^\fperp=(\Sigma\ac)^\fperp$,
$\Sigma\ac = {}^\fperp(\Sigma^\fperp)$, 
and $\Mod \cE \Sigma = \Mod \cE {\Sigma\ac}$.
\end{lemma}

\begin{proof}
We have $(\Sigma\ac)^\fperp\subseteq \Sigma^\fperp$, since $\Sigma\subseteq \Sigma\ac$.
Conversely, the class ${}^\fperp(\Sigma^\fperp)$ is acyclic by \cref{leftorthacyclic}.
Thus $\Sigma\ac\subseteq {}^\fperp(\Sigma^\fperp)$, since $\Sigma\subseteq {}^\fperp(\Sigma^\fperp)$.
It follows that $\Sigma\ac \fperp \Sigma^\fperp$ and hence that $\Sigma^\fperp\subseteq (\Sigma\ac)^\fperp$.
We now prove $\Sigma\ac = {}^\fperp(\Sigma^\fperp)$.
The class ${}^\fperp(\Sigma^\fperp)$ is acyclic by \cref{leftorthacyclic}.
Hence $\Sigma\ac \subseteq {}^\fperp(\Sigma^\fperp)$ since $\Sigma \subseteq {}^\fperp(\Sigma^\fperp)$.
We just proved that $(\Sigma\ac)^\fperp = \Sigma^\fperp$.
Then ${}^\fperp(\Sigma^\fperp) \subset \Sigma\ac$ is deduced from $\Sigma^\fperp \subseteq (\Sigma\ac)^\fperp$.
Finally, an object $X\in \cE$ belongs to $\Mod \cE \Sigma$ if and only if the map $p_X:X\to 1$
belongs to $\Sigma^{\fperp}$ if and only if $p_X$ belongs to $(\Sigma\ac)^{\fperp}$ if and only if $X$ belongs to $\Mod \cE {\Sigma\ac}$.
\end{proof}

Let $u:A\to B$ be map in a topos $\cE$.
Consider the functor $u^\flat: \cE\slice{B}\to \cE\arr$ which takes an object $(X,f)$ of $\cE\slice{B}$ to the base change $u^\flat(f)$ of the map $u$ along the map $f$:
\[
\begin{tikzcd}
X\times_B A \ar[rr,"{p_2}"] \ar[d, "u^\flat(f)"'] && A \ar[d, "{u}"] \\
X \ar[rr, "{f}"]&& B \, .
\end{tikzcd}
\]

\begin{lemma}
\label{phiiscc}
The functor $u^\flat: \cE\slice{B}\to \cE\arr$ is cocontinuous.
\end{lemma} 

\begin{proof}
It suffices to show that the functor $u^\flat$ has a right adjoint.
Let us first consider the case where $B=1$.
Consider the diagonal functor $\delta:\cE\to \cE\arr$ (which takes an object $X\in \cE$ to the map $1_X:X\to X$).
If $p_A:A\to 1$, then for every $X\in \cE$, we have $(p_A)^\flat(X)=\delta(X)\times p_A$.
Hence the functor $(p_A)^\flat: \cE \to \cE\arr$ is the composite of the functor $\delta:\cE\to \cE\arr$ followed by the functor $(-)\times p_A:\cE\arr\to \cE\arr$.
The functor $\delta$ is left adjoint to the domain functor $\mathrm{dom}:\cE\arr\to \cE$.
Moreover, the functor $(-)\times p_A:\cE\arr\to \cE\arr$ has a right adjoint $[p_A,-]:\cE\arr\to \cE\arr$ since the category $\cE\arr$ is cartesian closed.
It follows that the functor $g\mapsto \dom {[p_A,g]}$ is right adjoint to the functor $p_A^\flat$.
This proves the lemma in the case $B=1$.
In the general case, observe first that the functor $(A,u)^\flat: \cE\slice{B}\to (\cE\slice{B})\arr$ defined by the object $(A,u)$ of $\cE\slice{B}$ has a right adjoint by the first part of the proof.
Moreover, if $F_B: (\cE\slice{B})\arr\to \cE\arr$ is the forgetful functor, then we have $u^\flat=F_B\circ (A,u)^\flat$.
The forgetful functor $F_B$ is left adjoint to the base change functor $B^\star: \cE\arr\to (\cE\slice{B})\arr$ which takes a map $g:U\to V$ in $\cE$ to the map $B\times g: B\times U\to B\times V$ of $\cE\slice{B}$.
This shows that the functor $u^\flat=F_B\circ (A,u)^\flat$ has a right adjoint.
\end{proof}

\begin{definition}
\label{defG-basechange}
Let $\cG$ be a class of objects in a topos $\cE$.
We shall say that the base change $u':A'\to B'$ of a map $u:A\to B$ in $\cE$ is a $\cG$-\emph{base change} if $B'\in \cG$.
\end{definition}

\begin{proposition}
\label{saturatedacyclic} 
Let $\cG$ be a class of generators in a topos $\cE$.
If a class of maps $\Sigma\subseteq \cE$ is closed under $\cG$-base change, then $\Sigma\ac=\Sigma\sat$ and $\Sigma^\fperp =\Sigma^\upvdash$.
\end{proposition}

\begin{proof}
We have always $\Sigma\sat \subseteq \Sigma\ac$, so we only need to show $\Sigma\ac \subseteq \Sigma\sat$.
The saturated class $\cL:=\Sigma\sat$ contains a largest acyclic class $\cL'\subseteq \cL$ by \cref{largestacyclic}.
Let us see that $\Sigma \subseteq \cL'$.
If a map $u:A\to B$ belongs to $\Sigma$, then every $\cG$-base change of $u$ belongs to $\cL$, since $\Sigma \subseteq \cL$.
Let $\cG\slice{B}\subseteq \cE\slice{B}$ be the class of objects $(G,g)$ of $\cE\slice{B}$ with $G\in \cG$.
By \cref{generatinglemma}, the category $\cE\slice{B}$ is generated by the class $\cG\slice{B}$ since the category $\cE$ is generated by the class $\cG$.
The functor $u^\flat:\cE\slice{B}\to \cE\arr$, which takes an object $X=(X,f)$ of $\cE\slice{B}$ to the map $u^\flat(X)=f^\star(u):X\times_B A\to X$, is cocontinuous by \cref{phiiscc}.
Let $\underline{\cL}$ be the full subcategory of $ \cE\arr$ spanned by the maps in $\cL$.
The subcategory $\underline{\cL}\subseteq \cE\arr$ is closed under colimits, since $\cL$ is saturated.
Hence the full subcategory $(u^\flat)^{-1}(\underline{\cL})\subseteq \cE\slice{B}$ is closed under colimits, since $u^\flat$ is cocontinuous.
We have $u^\flat(G,g)\in \underline{\cL}$ for every $(G,g)\in \cE\slice{B}$, since every $\cG$-base change of the map $u:A\to B$ belongs to $\cL$.
That is $\cG\slice{B} \subseteq (u^\flat)^{-1}(\underline{\cL})$.
But the class $\cG\slice{B}$ generates $\cE\slice{B}$ by \cref{generatinglemma}.
Thus, $\cE\slice{B}= (u^\flat)^{-1}(\underline{\cL})$.
 This shows that we have $u^\flat (X,f)\in \cL$ for every map $f:X\to B$.
In other words, the base change of the map $u:A\to B$ along any map $f:X\to B$ belongs to $\cL$.
Thus, $u\in \cL'$ by \cref{largestacyclic}, and this
shows that $\Sigma \subseteq \cL'$.
It follows that $\Sigma\ac\subseteq \cL'$, since the class $\cL'$ is acyclic.
Hence we have $\Sigma\ac \subseteq \Sigma\sat$,
since $\Sigma\ac \subseteq \cL'\subseteq \cL=\Sigma\sat$;
the equality $\Sigma\ac=\Sigma\sat$ is proved.
By \cref{orthsaturatedclass} we have $(\Sigma\sat)^\upvdash=\Sigma^\upvdash$ and by \cref{orthacyclicclass} we have $(\Sigma\ac)^\fperp=\Sigma^\fperp$.
Thus, $\Sigma^\fperp =\Sigma^\upvdash$.
\end{proof}

\begin{cor}
\label{cor:generation-acyclic-class}
Let $\cG$ be a class of generators in a topos $\cE$.
For $\Sigma$ a class of maps in $\cE$, 
let $\Sigma\bc$ be the class of all its $\cG$-base changes, then 
\[
\Sigma\ac = (\Sigma\bc)\sat.
\]    
\end{cor}
\begin{proof}
The class $\Sigma\bc$ is stable by $\cG$-base change by transitivity of $\cG$-base change.
Hence we can apply \cref{saturatedacyclic}.
\end{proof}

\begin{theorem}
\label{thmgenerationmodality}
\label{lemmamodalitypres}
Let $\Sigma$ be a set of maps in a topos $\cE$.
If $\Sigma\ac\subseteq \cE$ is the acyclic class generated by $\Sigma$, then the pair $(\cL,\cR):=(\Sigma\ac,\Sigma^\fperp)$ is a modality. 
Moreover, $\cR[1]=\Mod \cE \Sigma$ and the category $\cR[1]$ is presentable.
\end{theorem}

\begin{proof}
The topos $\cE$ admits a set of generators $\cG$ since it is presentable.
Let us denote by $\Lambda$ the closure of $\Sigma$ under $\cG$-base change (\cref{defG-basechange}).
The class $\Lambda$ is a set, since $\cG$ and $\Sigma$ are.
Hence the pair $(\Lambda\sat, \Lambda^\upvdash)$ is a factorization system by \cref{genfactsyst}.
But ${\Lambda}\ac={\Lambda}\sat$ and ${\Lambda}^\fperp={\Lambda}^\upvdash$ by \cref{saturatedacyclic}, since $\Lambda$ is closed under $\cG$-base change.
Therefore $(\Lambda\ac, \Lambda^\fperp) = (\Lambda\sat, \Lambda^\upvdash)$ is a factorization system.
The factorization system $(\Lambda\ac, \Lambda^\fperp)$ is a modality, since the class $\Lambda\ac$
is closed under base change.
It remains to show that $(\Sigma\ac,\Sigma^\fperp)= (\Lambda\ac, \Lambda^\fperp)$.
We have ${\Sigma}\ac\subseteq \Lambda\ac $, since $\Sigma\subseteq \Lambda$.
But we have $\Lambda \subseteq {\Sigma}\ac$,
since ${\Sigma}\ac$ is closed under base change.
Thus, $\Lambda\ac \subseteq {\Sigma}\ac$,
and it follows that ${\Sigma}\ac=\Lambda\ac$.
To prove $\Sigma^\fperp=\Lambda^\fperp$, we use that $\Sigma^\fperp=(\Sigma\ac)^\fperp$ by \cref{orthacyclicclass}.
Since $\Sigma\ac=\Lambda\ac$, it follows that $(\Sigma\ac)^\fperp=({\Lambda}\ac)^\fperp$.
Thus, $\Sigma^\fperp=({\Lambda}\ac)^\fperp= \Lambda^\fperp$, since the pair $(\Lambda\ac, \Lambda^\fperp)$ is a modality.
This shows that $(\Sigma\ac,\Sigma^\fperp)=(\Lambda\ac, \Lambda^\fperp)$.
Hence the pair $(\cL,\cR):=(\Sigma\ac,\Sigma^\fperp)$ is a modality.
Let us show that $\cR[1]=\Mod \cE \Sigma$.
By \Cref{defmodalobject}, an object $X\in \cE$ belongs to $\Mod \cE \Sigma$
if and only if the map $p_X:X\to 1$ belongs to
$\Sigma^\fperp=\cR$ if and only if $X$ belongs to $\cR[1]$.
We saw in the proof that $(\Sigma\ac,\Sigma^\fperp)=
(\Lambda\ac, \Lambda^\fperp)=(\Lambda\sat, \Lambda^\upvdash)$.
Hence the factorization system
$(\cL,\cR)$ is generated by the set $\Lambda$.
It follows that the category $\cR[1]$ is presentable
by \Cref{lemmaR1pres}.
\end{proof}

\begin{definition}
\label{defmodgenerated}
We shall say that the modality $(\Sigma\ac,\Sigma^\fperp)$ of \cref{thmgenerationmodality} is \emph{generated} by the set of maps $\Sigma\subseteq \cE$,
and that a modality $(\cL,\cR)$ is of {\it small generation} if it is of the type $(\Sigma\ac,\Sigma^\fperp)$ for some set $\Sigma$.
\end{definition}

\subsection{\texorpdfstring{$\cR$-equivalences}{R-equivalences}}
\label{sec:fib-r-eqv}

In this section, we give a criterion under which the modality generated by a set of maps $\Sigma$ is left-exact.
The theorem will follow from the careful analysis of the following class of maps:

\begin{definition}[$\cR$-equivalence]
\label{def:requiv}
Let $(\cL,\cR)$ be a modality in a category with finite limits $\cE$.
We shall say that a map $u:A\to B$ in $\cE$ is an $\cR$-{\it equivalence} if the functor $u^\star:\cR[B]\to \cR[A]$ is an equivalence of categories.
We shall say that $u$ is a {\it fiberwise $\cR$-equivalence} if every base change of $u$ is an $\cR$-equivalence.
\end{definition}

\begin{exmp}
If $(\cL,\cR)=(\Surj,\Mono)$ is the modality of surjections and monomorphisms in $\cS$, then the $\Mono$-equivalences are exactly the maps $X\to Y$ inducing an isomorphism $\pi_0(X) \simeq \pi_0(Y)$, and the fiberwise $\Mono$-equivalences are the maps $X\to Y$ satisfying the stronger condition of having connected fibers.
\end{exmp}

\begin{lemma} 
\label{basechangeleftclass}
Let $(\cL,\cR)$ be a modality in a category with finite limits $\cE$.
Then a map $u:A\to B$ in $\cE$ belongs to $\cL$ if and only if the functor $u^\star:\cR[B]\to \cR[A]$ is fully faithful.
\end{lemma}

\begin{proof} 
The functor $u^\star:\cR[B]\to \cR[A]$ has a left adjoint $u_\sharp:\cR[A]\to \cR[B]$ by \cref{propusharpleftadjoint}.
Hence the functor $u^\star:\cR[B]\to \cR[A]$ is fully faithful if and only if the counit of the adjunction $u_\sharp \dashv u^\star$ is invertible.
By definition, for every map $g:Y\to B \in \cR$ we have $u^\star(Y,g)=(A\times_B Y,p_1)$.
\[
\begin{tikzcd}
A\times_B Y \ar[rr, "p_2"] \ar[d, "p_1"'] && Y \ar[d, "g"] \\
A \ar[rr, "u"] && B
\end{tikzcd}
\]
Suppose that $u$ belongs to $\cL$. We will prove that the canonical map $u_\sharp u^\star(Y,g)\to (Y,g)$ is invertible.
As a base change of $u$ the map $p_2$ belongs to $\cL$.
But $g$ belongs to $\cR$ by hypothesis.
Thus, $\cL(up_1)=p_2$ and $\cR(up_1)=g$.
This shows that $u_\sharp(A\times_B Y,p_1)=(Y,g)$, so the canonical map $u_\sharp u^\star(Y,g)\to (Y,g)$ is the identity.

Conversely, if the counit $u_\sharp u^\star(Y,g)\to (Y,g)$ is invertible for every $(Y,g)\in \cR[B]$, then it is invertible in the special case where $(Y,g)=(B,1_B)$.
But $u^\star(B,1_B)=(A,1_A)$ and hence 
 the object 
 $u_\sharp u^\star(B,1_B)=u_\sharp(A,1_A)=(\|(u\|,\cR(u))$ 
is isomorphic to  $(B,1_B)$.
This means that the map $\cR(u)$ is invertible
and hence that the map $u=\cR(u)\cL(u)$ belongs to $\cL$.
\end{proof}

\begin{lemma}
\label{Deltainleft}
Let $(\cL,\cR)$ be a modality in a category with finite limits $\cE$ and let $u:A\to B$ be a map in $\cE$.
If $\Delta(u)\in \cL$, then the functor $u_\sharp:\cR[A]\to \cR[B]$ is fully faithful.
\end{lemma}
 
\begin{proof}
It suffices to show that for any $f : X \to A$, the canonical map $X \to u^* u_\sharp(X)$ is an isomorphism.
To this end, consider the following cubical diagram:
\[
\begin{tikzcd}[bo column sep=large, row sep=large]
u^*(X) \ar[rr, "u^*\cL(uf)"] \ar[dd, "u^*(f)"'] \ar[dr]
& & u^*(u_\sharp(X)) \ar[dd, near end, "u^*\cR(uf)"] \ar[dr] & \\
& X \ar[rr, crossing over, near start, "\cL(uf)"] & & u_\sharp(X) \ar[dd, "\cR(uf)"] \\
A \times_B A \ar[rr, near end, "p_2"] \ar[dr, "p_1"'] & & A \ar[dr, "u"]  & \\
& A \ar[rr, crossing over, "u"'] \ar[from=uu, crossing over, near start, "f"'] & & B
\end{tikzcd}
\]
The front face of the cube is a square in $\cE\slice{B}$ and the cube is obtained by pulling back this square along the map $u:A\to B$.
It follows immediately that the left, the right, the top and the bottom faces of the cube are cartesian squares.
Hence we have $p_1^\star(X)=u^\star(X)$, since the left face is cartesian.
Moreover, $u^*\cL(uf) \in \cL$ and $u^*\cR(uf) \in \cR$, since the classes $\cL$ and $\cR$ are closed under base change.
We deduce that
\begin{equation}
\label{someequation}
(p_2)_\sharp u^*(X) = u^*u_\sharp(X) 
\end{equation}
by the definition of $u_\sharp$. 
Observe now that the functor $\Delta(u)^*: \cR[A \times_B A]\to \cR[A]$ is fully faithful by \cref{basechangeleftclass}, since the map $\Delta(u):A\to A \times_B A$ belongs to $\cL$. 
Hence we have $\Delta(u)_\sharp\Delta(u)^*(Y) = Y$ for every $Y\in \cR[A \times_B A]$.
We now simply calculate:
\begin{align*}
X &= (p_2)_\sharp \Delta(u)_\sharp(X) && \text{since $p_2 \circ \Delta(u) = 1_A$} \\
&=(p_2)_\sharp\Delta(u)_\sharp \Delta(u)^* p_1^*(X) &&
\text{since $p_1 \circ \Delta(u) = 1_A$}  \\
&= (p_2)_\sharp p_1^*(X) && \text{since $\Delta(u)^*$ is fully faithful} \\	 &= (p_2)_\sharp u^*(X) && \text{since} \ p_1^\star(X)=u^\star(X)  \\
&= u^*u_\sharp(X) && \text{from \eqref{someequation} }
\end{align*} 
This shows that the canonical map  $X \to u^* u_\sharp(X)$ is an isomorphism. 
\end{proof}

\begin{proposition}
\label{basechangesequiv}
Let $(\cL,\cR)$ be a modality in a category with finite limits $\cE$.
Then a map $u:A\to B$ in $\cE$ is a fiberwise $\cR$-equivalence if and only if both $u$ and $\Delta(u)$ belong to $ \cL$.
\end{proposition}

\begin{proof}
Let $u:A\to B$ be an $\cR$-equivalence. Then $u\in \cL$ by \cref{basechangeleftclass}, since the functor $u^\star:\cR[B]\to \cR[A]$ is fully faithful.
It remains to show that $\Delta(u)\in \cL$.
The projection $p_1:A\times_B A \to B$ is a base change of the map $u:A\to B$, since the square 
\[
\begin{tikzcd}
A\times_B A\ar[r, "p_2"] \ar[d, "p_1"'] & A\ar[d, "u"] \\
A \ar[r, "u"'] \ar[r] & B
\end{tikzcd}
\]
is cartesian.
Since $u$ is a fiberwise $\cR$-equivalence, the map $p_1:A\times_B A \to B$ is one as well. 
Therefore the functor $p_1^\star:\cR[A]\to \cR [A\times_B A]$ is an equivalence of categories.
But we have $\Delta(u)^\star p_1^\star =\id$, since $p_1\Delta(u)=1_A$.
Hence the functor $\Delta(u)^\star: \cR [A\times_B A]\to \cR[A]$ is an equivalence of categories.
From \cref{basechangeleftclass} then follows that $\Delta(u)\in \cL$.

Conversely, if the maps $u$ and $\Delta(u)$ belong to $ \cL$, we need to demonstrate that $u$ is a fiberwise $\cR$-equivalence.
We will first prove that $u$ is an $\cR$-equivalence.
The functor $u^\star$ is fully faithful by \cref{basechangeleftclass}, since $u\in \cL$.
This means that the counit of the adjunction $u_\sharp\dashv u^\star$ is invertible.
But $\Delta(u)\in \cL$ and by \cref{Deltainleft} the functor $u_\sharp$ is also fully faithful. 
Hence the unit of the adjunction $u_\sharp\dashv u^\star$ is invertible.
Therefore $u^\star:\cR[B]\to \cR[A]$ is an equivalence of categories and $u$ is an $\cR$-equivalence.

To finish the proof it remains to show that every base change $u'$ of $u$ is also an $\cR$-equivalence.
If $u'$ is a base change of $u$ along a map $p:B'\to B$, then the map $\Delta(u')$ is a base change of the map $\Delta(u)$, since the base change functor $p^\star: \cE\slice{B}\to \cE\slice{B'}$ preserves limits.
Hence the maps $u'$ and $\Delta(u')$ belong to $\cL$.
With the same argument as above it follows the map $u'$ is an $\cR$-equivalence, whence $u$ is a fiberwise $\cR$-equivalence.
\end{proof}

The following lemma shows that a limit of fully faithful functors is fully faithful.

\begin{lemma} \label{fullfaithlim}
Let $\alpha:F\to G$ be a natural transformation between two diagrams of categories $F,G:\cK\to \CAT$.
If the functor $\alpha(k):F(k)\to G(k)$ is fully faithful for every object $k\in \cK$, then so is the functor 
\[
\lim_k \alpha(k):\lim_k F(k)\to \lim_k G(k).
\]
\end{lemma}

\begin{proof}
The class of fully faithful functors $\mathsf{FullFaith}\subseteq \CAT$ is the right class of a factorization system in the category $\CAT$ by \cref{factsysexemp3}.
Hence the class $\mathsf{FullFaith}$ is closed under small limits by \cref{omnibus}. 
\end{proof}

For a topos $\cE$ consider the functor $T:\cE\op\to \CAT$ defined by putting $T(A)=\cE\slice{A}$ for every object $A\in \cE$, and by putting $T(u)=u^\star:\cE\slice{B}\to \cE\slice{A}$ for every map $u:A\to B$.
The functor $T$ preserves limits by \cref{descenttheorem}, since $\cE$ is a topos. 

\smallskip

Let $\cF$ be a local class of maps in a topos $\cE$.
For each $A\in \cE$, we denote by ${\cF}(A)$ the full subcategory of $\cE\slice{A}$ spanned by the objects $(X,f)$ with a structure map $f:X\to A$ in $\cF$.
If $u:A\to B$ is a map in $\cE$, then $u^\star({\cF}(B))\subseteq {\cF}(A)$, since the class $\cF$ is closed under base change.
We denote by ${\cF}(-):\cE\op\to \CAT$ the subfunctor of the functor $T:\cE\op\to \CAT$ so defined.

\begin{lemma}[{\cite[Proposition 6.2.3.14]{HTT}}]
\label{descentforF}
Let $\cF$ be a local class of maps in a topos $\cE$.
Then the functor ${\cF}(-):\cE\op\to \CAT$ defined above preserves small limits.
\end{lemma}

\begin{proof}
Let $D:\cK \to \cE$ be a  diagram in $\cE$ with colimit $A\in \cE$ and let $u(k):D(k)\to A$ be the structure map for every $k\in \cK$.
The functor 
\[
u^\star: \cE\slice{A} \to \lim_k  \cE\slice{D(k)}
\]
defined by putting $u^\star(X)(k)=u(k)^\star(X)\in \cE\slice{D(k)}$ for every $k\in \cK $ is an equivalence of categories, since the functor $T$ preserves limits by \cref{descenttheorem}.
Now we consider the functor
\[
\tilde{u}^\star:   \cF(A) \to \lim_k  \cF(D(k))
\]
defined by putting $\tilde{u}^\star(X)(k)=u(k)^\star(X)\in  \cF(D(k))$ for $X\in  \cF(A)$ and $k\in \cK$. 
The claim is that $\tilde{u}^\star$ is an equivalence of categories.

We will first prove that $\tilde{u}^\star$ is fully faithful.
To achieve this we consider the following commutative square of functors
\begin{equation}
\label{squareforlim}
\begin{tikzcd}
\cF(A)\ar[rr,"{\tilde{u}^\star}"] \ar[d, "{J_A}"'] && \lim_k \cF(D(k))
\ar[d, "\lim_k {J_{D(k)}}"] \\
\cE\slice{A} \ar[rr, "u^\star"] && \lim_k  \cE\slice{D(k)} 
\end{tikzcd}
\end{equation}
where $J_B$ is the inclusion functor $ \cF(B)\subseteq  \cE\slice{B}$. 
The functor $\lambda:=\lim_k {J_{D(k)}}$ is fully faithful by \cref{fullfaithlim}, since the inclusion functor $J_{D(k)}$ is fully faithful for every $k\in \cK$.
The functor $u^\star \circ J_A$ is fully faithful, since $J_A$ is fully faithful and $u^\star$ is an equivalence of categories.
Hence $\lambda \circ \tilde{u}^\star$ is fully faithful, since the square \eqref{squareforlim} commutes.
The class of fully faithful functors has the left cancellation property by \cref{factsysexemp3} and \cref{omnibus}.
So $\tilde{u}^\star$ is fully faithful.

Now it remains to show that the functor $\tilde{u}^\star$ is essentially surjective.
An object $(Y,g)$ of the category $ \lim_k \cF(D(k))$ is a diagram $Y:\cK\to \cE$ equipped with a cartesian natural transformation $g:Y\to D$ such that $g(k)\in \cF$ for every $k\in \cK$.
There exists an object $(X,f)\in \cE\slice{A}$ such that $u(k)^\star(X,f)=(Y(k),g(k))$ for every $k\in \cK$, since the functor $u^\star$ is an equivalence of categories.
The family of maps $u(k):D(k)\to A$ is surjective by \cref{colimcoverage}, since $A=\colim_k  D(k)$.
Thus, $f\in \cF$, since for every $k\in \cK$ we have $g(k)\in \cF$ and since the class $\cF$ is local.  
It follows that $(X,f)\in \cF(A)$ and $\tilde{u}^\star(X,f)=(Y,g)$, since $u^\star(X,f)=(Y,g)$.
This shows that the functor $\tilde{u}^\star$ is an equivalence of categories. Therefore the functor ${\cF}(-):\cE\op\to  \CAT$ preserves small limits.
\end{proof}

Let $\cF$ be a local class of maps in a topos $\cE$.
We shall say that a map $u:A\to B$ is an $\cF$-{\it equivalence} if the functor $u^\star:\cF(B)\to \cF(A)$ is an equivalence of categories.

\begin{lemma}
\label{saturationFequivalence} 
Let $\cF$ be a local class of maps in a topos $\cE$.
Then the class of $\cF$-equivalences is strongly saturated.
\end{lemma}

\begin{proof} 
The functor $\cF(-):\cE\op\to \CAT$ preserves (small) limits by \cref{descentforF}.
It then follows from \cref{stronglysatcocont} that the class of $\cF$-equivalences is strongly saturated.
\end{proof}

\begin{theorem}
\label{saturationRequivalence}
Let $(\cL,{\cR})$ be a modality in a topos $\cE$.
Then the class $\cW$ of $\cR$-equivalences is strongly saturated and the class $\cA$ of fiberwise $\cR$-equivalences is acyclic.
\end{theorem}

\begin{proof}
The first statement follows from \cref{saturationFequivalence}, since the class $\cR$ is local by \cref{prop-localclass}.  
Moreover, a map $u:A\to B$ belongs to $\cA$ if and only if every base change $u':A'\to B'$ of $u$ belongs to $\cW$.
It follows by \cref{largestacyclic} that the class $\cA$ is acyclic.
\end{proof}

\begin{remark}
\label{rem:susp}
Putting together \cref{basechangesequiv,saturationRequivalence} we get that the class $\mathsf S(\cL) = \{f \,|\, f\in \cL, \Delta f \in \cL\}$ is acyclic for any acyclic class $\cL$.
Then, by \cref{charac-cong}, an acyclic class is a congruence if and only if $\mathsf S(\cL)= \cL$.
\end{remark}

\section{Left-exact modalities and higher sheaves}
\label{sec:lex-mod-hs}

\subsection{Left-exact modalities}
\label{sec:lexmods}

Among the modalities $(\cL,\cR)$ in a topos $\cE$, we may distinguish those for which the full subcategory $\cR[1]$ is again a topos.
We shall see that this is the case as soon as the associated reflector preserves finite limits.

\medskip

Recall that if $(\cL,\cR)$ is a factorization system
in a category $\cE$,
the factorization of a map $f:X\to Y$
\[
  \begin{tikzcd}
    X\ar[dr, "\cL(f)"'] \ar[rr, "f"] & & Y  \\
    &\|f\| \ar[ur, "\cR(f)"'] &
  \end{tikzcd}
\]
defines three functors  
$\cL:\cE\arr\to \cE\arr$,
$\cR:\cE\arr\to \cE\arr$, 
and $\|-\|:\cE\arr\to \cE$.

\begin{definition}[Left-exact factorization systems]
\label{def:lex-modality}
Let $\cE$ be a category with finite limits.
We shall say that a factorization system $(\cL,\cR)$ is \emph{left-exact}, or simply
\emph{lex}, if the functor $\|-\|:\cE\arr\to \cE$ is lex.
\end{definition}

\begin{lem}
\label{lem:def-lex-mod}
Let $\cE$ be a category with finite limits and $(\cL,\cR)$ a factorization system.
The following conditions are equivalent
\begin{enumerate}
\item \label{lem:def-lex-mod:1} the factorization system is lex ($\|-\|:\cE\arr\to \cE$ is lex);
\item \label{lem:def-lex-mod:1bis} the functor $\|-\|:\cE\arr\to \cE$ preserves fiber products;
\item \label{lem:def-lex-mod:2} both functors $\cL:\cE\arr\to \cE\arr$ and
$\cR:\cE\arr\to \cE\arr$ are lex (i.e. the factorization is stable under finite limits);
\item \label{lem:def-lex-mod:3} either of the functors $\cL:\cE\arr\to \cE\arr$ or
$\cR:\cE\arr\to \cE\arr$ is lex;
\item \label{lem:def-lex-mod:4} the class $\cL$ is stable by finite limits;
\item \label{lem:def-lex-mod:4bis} the class $\cL$ is stable by fiber products.
\end{enumerate}
\end{lem}
\begin{proof}
\noindent \eqref{lem:def-lex-mod:1}$\Leftrightarrow$\eqref{lem:def-lex-mod:1bis}
because the functor $\|-\|:\cE\arr\to \cE$ always preserves terminal object.

\smallskip
\noindent \eqref{lem:def-lex-mod:1}$\Leftrightarrow$\eqref{lem:def-lex-mod:2}.
We denote by $s,t:\cE\arr\to \cE$ the source and target functors, they are both lex functors.
Moreover the pair $(s,t):\cE\arr\to \cE^2$ creates limits.
Since $\cL:\cE\arr\to \cE$ is the identity on the source
and $\cR:\cE\arr\to \cE$ is the identity on the target.
Hence they are lex functors if and only if $t\circ \cL=s\circ \cR=\|-\|:\cE\arr\to \cE$ is lex.

\smallskip
\noindent \eqref{lem:def-lex-mod:2}$\Rightarrow$\eqref{lem:def-lex-mod:4}.
A map $f$ is in $\cL$ if and only if $\cR(f)$ is an isomorphism.
Let $f_i$ be a diagram in $\cL$ and $f$ its limit in $\cE\arr$.
If $\cR$ is lex, the map $\cR(f) = \lim\cR(f_i)$ is an isomorphism, and hence $f$ is in $\cL$.

\smallskip
\noindent \eqref{lem:def-lex-mod:4}$\Rightarrow$\eqref{lem:def-lex-mod:2}.
Conversely, we want to show that $\cL:\cE\arr\to \cE$ is lex.
Let $f_i$ be a diagram in $\cE$ and $f$ its limit.

Let $f_i=p_iu_i$ the $(\cL,\cR)$-factorization of each $f_i$ 
Then the maps $u= \lim u_i$ or $p= \lim p_i$ define a factorization $f=pu$, which we want to prove is the $(\cL,\cR)$-factorization of $f$.
Since the factorization is unique, it is enough to show that $p\in \cR$ and $u\in \cL$.
Being the right class of a factorization system, $\cR$ is always stable by limits, 
hence $p=\lim p_i$ is in $\cR$.
The conclusion follows from the hypothesis, saying that $u=\lim u_i$ is in $\cL$.

\smallskip
\noindent \eqref{lem:def-lex-mod:4}$\Leftrightarrow$\eqref{lem:def-lex-mod:4bis}
because $\cL$ always contains the identity map of $\cE$.
\end{proof}

\begin{lem}
\label{lem:lex-reflector}\label{lexreflectorlexact}
If $(\cL,\cR)$ is a lex factorization system on $\cE$, then the functor $\|-\|:\cE\to \cR[1]$ is left-exact.
\end{lem}
\begin{proof}
The inclusion $\cE\hookrightarrow\cE\arr$ sending $X$ to $X\to 1$ is always left-exact.
The composition $\cE\hookrightarrow\cE\arr\xto {\|-\|} \cE$ has values in the full subcategory $\cR[1] \subseteq\cE$.
Since $\cR[1]$ is closed under finite limits, the resulting functor $\cE\to \cR[1]$ is left-exact.
\end{proof}

\begin{proposition}
\label{propexcataremodal}
Every left-exact factorization system $(\cL,\cR)$ is a modality.
Moreover, the class $\cL$ satisfies 3-for-2.
\end{proposition}
\begin{proof}
The class $\cL$ is closed under finite limits by \Cref{lem:def-lex-mod}, and under base changes by the dual of \Cref{saturatedareclosedundercob}.
This shows that the factorization system  $(\cL,\cR)$ is a modality.
The class $\cL$ is closed under composition and it has the right cancellation property by \Cref{omnibus}.
Moreover, it has the left cancellation property by \Cref{colexrightcancel}.
This proves $\cL$ satisfies the 3-for-2 property.
\end{proof}

\begin{rem}
Because of \cref{propexcataremodal}, we shall often refer to a left-exact factorization system as a {\it left-exact modality}.
\end{rem}

 Let $\cE$ be a category with finite limits and
  $\phi:\cE\to \cE'\subseteq \cE$ be a left-exact reflector.
 Then the class of maps $\cL_\phi\subseteq \cE$
 inverted by the functor $\phi$
 is the left class of a factorization system
 $(\cL_\phi,\cR_\phi)$ by \Cref{factsysfromlex}.
 
 \begin{proposition}\label{LRphilex}
  Let $\cE$ be a category with finite limits and
  $\phi:\cE\to \cE'\subseteq \cE$ be a left-exact reflector.
  Then the factorization system
 $(\cL_\phi,\cR_\phi)$
 is a left-exact modality.
 \end{proposition}
 
\begin{proof}
By definition, $\cL_\phi$ is the class of maps $f\in \cE$ inverted by the functor  $\phi:\cE\to \cE'$.
The class $\cL_\phi$ is closed under finite limits, since the functor $\phi$ preserves finite limits.
Hence the factorization system $(\cL_\phi,\cR_\phi)$ is left-exact by \Cref{lem:def-lex-mod} 
and a modality by \Cref{propexcataremodal}. 
\end{proof}

\begin{lemma}[{\cite[Thm 3.4.11]{AS}}]
\label{finitelimclosure} 
Let $\cE$ be a category with finite limits and let $\cL\subseteq \cE$ be a class of maps which contains the isomorphisms, is closed under composition and base changes.
If $\Delta(\cL)\subseteq \cL$, then the class $\cL$ is closed under finite limits.
\end{lemma}
 
\begin{proof} 
It suffices to prove that the full subcategory
$\underline{\cL}\subseteq \cE\arr$
is closed under pullbacks. 
Let $f_1\to f_0 \leftarrow f_2$ be a pullback diagram in $\cL$. We will decompose it into the composition of three pullbacks:
\[
\begin{tikzcd}
A_1 \ar[d,equal] \ar[r]   & A_0 \ar[d,"{f_0}"]       & A_2 \ar[d,equal] \ar[l] \ar[d,bend left=100,phantom, "(a)" description]\\
A_1 \ar[d, "f_1"']  \ar[r]        & B_0 \ar[d,equal] & A_2 \ar[d,equal] \ar[l] \ar[d,bend left=100,phantom, "(b)" description]\\
B_1 \ar[d,equal]  \ar[r]  & B_0 \ar[d,equal] & A_2 \ar[d,"f_2"] \ar[l] \ar[d,bend left=100,phantom, "(c)" description]\\
B_1 \ar[r]                & B_0              & B_2 \ar[l]
\end{tikzcd}
\]
The diagram 
\[
\begin{tikzcd}
A_1\times_{A_0}A_2 \ar[r]\ar[d,"g"'] \pbmark & A_0\ar[d,"\Delta f_0"]\\
A_1\times_{B_0}A_2 \ar[r]\ar[d] \pbmark & A_0\times_{B_0}A_0\ar[d]\\
A_1\times A_2 \ar[r] &  A_0\times A_0
\end{tikzcd}
\]
proves that the pullback $g$ of $(a)$ is a base change of $\Delta f_0$ which is in $\cL$ by assumption. So $g$ is in $\cL$. 
Similarly, the pullback of $(b)$ is in $\cL$ because it is a base change of $f_1$, and the pullback of $(c)$ is in $\cL$ because it is a base change of $f_2$.
Then the pullback of $f_1\to f_0 \leftarrow f_2$ is in $\cL$ because $\cL$ is closed under by composition.
\end{proof}

The notion of left-exact modality admits many equivalent
characterizations (for example, at least 13 are given in \cite{RSS}).
We choose the following ones:

\begin{theorem}
\label{lexcharacterization}
Let $\cE$ be a category with finite limits.
Then the following conditions on a modality $(\cL, \cR)$ in $\cE$ are equivalent:
\begin{enumerate}
\item \label{lexcharacterization1} the modality $(\cL, \cR)$ is left-exact;
\item \label{lexcharacterization2} the class $\cL$ has the 3-for-2 property;
\item \label{lexcharacterization3} $u\in \cL\Rightarrow \Delta(u)\in  \cL$;
\item \label{lexcharacterization4} the class $\cL$ is closed under finite limits;
\item \label{lexcharacterization5} the functor $u^\star:\cR[B]\to \cR[A]$
is an equivalence of categories for every map $u:A\to B$
in $\cL$;
\item \label{lexcharacterization6} the functor $u_\sharp:\cR[A]\to \cR[B]$
is an equivalence of categories for every map 
$u:A\to B$ in $\cL$;
\item \label{lexcharacterization7} a map $f:X\to Y$ belongs to $\cR$
if and only if the following naturality square is cartesian,
\begin{equation}\label{natsquareforrightclass}
    \begin{tikzcd}
X \ar[r]\ar[d,"f"']   & \|X\| \ar[d,"{\|f\|}"]\\
Y \ar[r] & \|Y\|
  \end{tikzcd}
\end{equation}

\item \label{lexcharacterization8} a map $f\in \cE$ belongs to $\cL$
if and only if the map $\|f\|$ is invertible;

\end{enumerate}
\end{theorem}

\begin{proof}
\noindent
\eqref{lexcharacterization1}$\Rightarrow$\eqref{lexcharacterization2}
follows from \Cref{propexcataremodal}.

\smallskip
\noindent
\eqref{lexcharacterization2}$\Rightarrow$\eqref{lexcharacterization3}
If a map $u:A\to B$ belongs to $\cL$, then so
does the projection $p_1:A\times_B A\to A$, since the latter is a
base change of the former.  Hence the diagonal $\Delta(u):A\to A\times_B A$
belongs to $\cL$ by 3-for-2, since $p_1\Delta(u)=1_A$ 
belongs to $\cL$.
  
\smallskip
\noindent
\eqref{lexcharacterization3}$\Rightarrow$\eqref{lexcharacterization4}
follows from \Cref{finitelimclosure}.
 
\smallskip
\noindent
\eqref{lexcharacterization4}$\Rightarrow$\eqref{lexcharacterization1}
follows from \Cref{lem:def-lex-mod}.

\smallskip
\noindent
\eqref{lexcharacterization3}$\Rightarrow$\eqref{lexcharacterization5}
If a map $u:A\to B$ belongs to $\cL$,
then $\Delta(u)$ belongs to $\cL$ by (3). Hence the functor
$u^\star:\cR[B]\to  \cR[A]$
is an equivalence of categories by \Cref{basechangesequiv}.

\smallskip
\noindent
\eqref{lexcharacterization5}$\Leftrightarrow$\eqref{lexcharacterization6} 
We have $u_\sharp \vdash u^\star$ by \Cref{propusharpleftadjoint}.
Hence the functor $u_\sharp$ is an equivalence if and only if the functor $u^\star $ is an equivalence.

\smallskip
\noindent
\eqref{lexcharacterization6}$\Rightarrow$\eqref{lexcharacterization7}
The canonical map $u:=\eta(Y)=\cL(Y\to 1):Y\to \| Y\| $
belongs $\cL$ by \Cref{reflectionR(1)}.
Hence the
functor $u_\sharp:\cR[Y]\to \cR [\|Y\|]$
is an equivalence of categories by (6). 
It then follows from \Cref{corollpusharpleftadjoint} 
that the square
\eqref{natsquareforrightclass} is cartesian
for every map $f:X\to Y$ in $\cR$.
Conversely, if the square
\eqref{natsquareforrightclass} is cartesian,
then $f$ belong to $\cR$, since the map $\|f\|$
belongs to $\cR$ by \Cref{reflectionR(1)},
and the class $\cR$ is closed under base changes
by \Cref{omnibus}.

\smallskip
\noindent
\eqref{lexcharacterization7}$\Rightarrow$\eqref{lexcharacterization8}
If $f\in \cE$ belongs to $\cL$,
then $\|f\|$ is invertible by
\Cref{reflectionR(1)}.
Conversely, if the map $\|f\|$ is invertible,
let us show that $f\in \cL$.
For this, let us choose a decomposition $f=pu:X\to E\to Y$
with $u\in \cL$ and $p\in \cR$. The map
$\|u\|$ is invertible by \Cref{reflectionR(1)},
since $u\in \cL$.
Hence the map $\|p\|$ is invertible,
since $\|p\|\|u\|=\|f\|$ is invertible by hypothesis.
But the following naturality square is cartesian by (7), since $p\in \cR$.
\begin{equation}
    \begin{tikzcd}
E \ar[rr,"\eta(E)"]\ar[d,"p"']   && \|E\| \ar[d,"{\|p\|}"]\\
Y \ar[rr,"\eta(Y)"] && \|Y\|
  \end{tikzcd}
\end{equation}
Thus, $p$ is invertible, since $\|p\|$ is invertible.
It follows that $f=pu\in \cL$, since $u\in \cL$.

\smallskip
\noindent
\eqref{lexcharacterization8}$\Rightarrow$\eqref{lexcharacterization2}
The class $\cL$ has the 3-for-2 property,
since a map $f\in \cE$ belongs to $\cL$
if and only if the map $\|f\|$ is invertible
\end{proof}

\begin{theorem}
\label{thm:recognition}
Let $\cE$ be a topos and let $(\cL,\cR)$ be a modality in $\cE$ generated by a class of maps $\Sigma\subseteq \cE$.
Then the modality $(\cL,\cR)$ is left-exact if and only if $\Delta(\Sigma)\subseteq \cL$.
\end{theorem} 
\begin{proof}
The condition $\Delta(\Sigma)\subseteq \cL$ is necessary by \cref{lexcharacterization}.
Let us see that it is sufficient.
The class $\cA \subseteq \cE$ of fiberwise $\cR$-equivalences is acyclic by \cref{saturationRequivalence}.
Let us show that $\cL\subseteq \cA$.
Every map $u\in \Sigma$ is a fiberwise $\cR$-equivalences by \cref{basechangesequiv}, since we have $\Sigma\cup \Delta(\Sigma) \subseteq \cL$ by hypothesis.
Thus, $\Sigma\subseteq \cA$ and hence $\Sigma\ac\subseteq \cA$, since the class $\cA$ is acyclic.
But we have $\cL=\Sigma\ac$, since the modality $(\cL,\cR)$ is generated by $\Sigma$.
This proves that  $\cL\subseteq \cA$.
It follows by  \cref{basechangesequiv} that $\Delta(\cL)\subseteq \cL$ and hence that the modality $(\cL,\cR)$ is left-exact by \cref{lexcharacterization}. 
\end{proof}

\begin{corollary}
\label{cor:mono-gen}
If a modality $(\cL,\cR)$ in a topos is generated by a set of monomorphisms, then it is left-exact.
\end{corollary} 
\begin{proof}
The diagonal $\Delta(u)$ of a monomorphism $u:A\to B$ is invertible.
Thus, $\Delta(\Sigma)\subseteq \cL$, since every isomorphism is in $\cL$ by \cref{omnibus}.
\end{proof}

Let $\cE$ be a category with finite limits.
If  $\phi:\cE\to \cE'\subseteq \cE$ is a left-exact reflector, then the factorization system 
$(\cL_\phi,\cR_\phi)$ of \Cref{factsysfromlex} is a left-exact modality by \Cref{LRphilex}. 
The following proposition shows that every left-exact factorization system $(\cL,\cR)$ is of the form  $(\cL_\phi,\cR_\phi)$ for a unique reflector $\phi:\cE\to \cE'\subseteq \cE$.
 
\begin{proposition}
\label{lexlocvslexmod}
The correspondence $\phi\mapsto (\cL_\phi,\cR_\phi)$ is a bijection between the left-exact reflectors  $\phi:\cE\to \cE'\subseteq \cE$ and the left-exact modalities $(\cL,\cR)$ in $\cE$.
The inverse correspondence takes a left-exact modality $(\cL,\cR)$ to the reflector $\|-\|:\cE\to \cR[1]\subseteq \cE$.
\end{proposition}

\begin{proof} 
If  $\phi:\cE\to \cE'\subseteq \cE$ is a left-exact reflector, then factorization system $(\cL_\phi,\cR_\phi)$ is a left-exact modality by \Cref{LRphilex}.
Let us show that $\cR_\phi(1)=\cE'$.
An object $X\in \cE$ belongs to $\cR_\phi(1)$ 
if and only if the map $X\to 1$ belongs to $\cR_\phi$, 
if and only if the following naturality square 
\begin{equation}
\begin{tikzcd}
X\ar[d] \ar[r, "\eta(X)"] &\phi(X) \ar[d]  \\
1 \ar[r] &1 
\end{tikzcd}
\end{equation}
is cartesian by \Cref{factsysfromlex}, since $\phi(1)=1$.
Thus, $X$ belongs to $\cR_\phi(1)$ 
if and only if the map $\eta(X)$ is invertible
if and only if $X$ belongs to $\cE'$.
This shows that $\cR_\phi(1)=\cE'$.
It follows that the reflector $\|-\|:\cE\to \cR_\phi(1)$ is equal to $\phi: \cE\to \cE'$.
Conversely, if $(\cL,\cR)$ is a left-exact modality in $\cE$, then the reflector $\|-\|:\cE\to \cR[1]$ is left-exact by \Cref{lexreflectorlexact}.
If $\phi:=\|-\|:\cE\to \cR[1]\subseteq \cE$, then $\cL_\phi=\cL$ by \Cref{lexcharacterization}\eqref{lexcharacterization8}.
Thus, $(\cL,\cR)=(\cL_\phi,\cR_\phi)$.
\end{proof}

\begin{rem}
Notice that no accessibility assumption is needed in \cref{lexlocvslexmod}.
This is to be compared with \cref{thm:bij-congruence-lexloc}.
\end{rem}

\subsection{Congruences}
\label{sec:congruences}

\begin{definition}[Congruence]
\label{defcongruenceclass}
Let $\cE$ be a topos.
We shall say that a class of maps $\cL\subseteq \cE$ is a {\it congruence} if the following conditions hold:
\begin{enumerate}[label=\roman*)]
\item $\cL$ contains the isomorphisms and is closed under composition;
\item ${\cL}$ is closed under colimits;
\item ${\cL}$ is closed under finite limits.
\end{enumerate}
Equivalently, a class of maps $\cL\subseteq \cE$ is a congruence if it is saturated and closed under finite limits.
\end{definition}

\begin{definition}
\label{defcongruencegen} 
Every class of maps $\Sigma\subseteq \cE$ in a topos $\cE$ is contained in a smallest congruence $\Sigma\cong \subseteq \cE$.
We shall say that ${\Sigma}\cong$ is the congruence \emph{generated} by the class of maps $\Sigma\subseteq \cE$.
We shall say that a congruence $\cL\subseteq \cE$ is of {\it small generation} if $\cL={\Sigma}\cong$ for a set of maps $\Sigma\subseteq \cE$.
\end{definition}

\begin{proposition} 
\label{charac-cong} \label{congruenceisstrongsat} \label{congclosediag}
In a topos, the following conditions on a class ${\cL}$ are equivalent:
\begin{enumerate}
\item \label{charac-cong:1} $\cL$ is a congruence;
\item \label{charac-cong:2} $\cL$ is strongly saturated and closed under base change;
\item \label{charac-cong:3} $\cL$ is acyclic and closed under diagonals ($\Delta(\cL)\subseteq \cL$).
\end{enumerate}
\end{proposition}

\begin{proof}
\eqref{charac-cong:1} $\Rightarrow$ \eqref{charac-cong:2}
The class $\cL$ contains the isomorphisms and the full subcategory $\underline{\cL}\subseteq \cE\arr$ is closed under finite limits and finite colimits.
It follows that $\cL$ has the left and the right cancellation properties by \cref{colexrightcancel} and its dual.
Thus $\cL$ is strongly saturated since it is closed under colimits.
Moreover,  the class $\cL$ is closed under base change by the dual of \cref{saturatedareclosedundercob}.

\smallskip
\noindent \eqref{charac-cong:2} $\Rightarrow$ \eqref{charac-cong:3}
If a map $u:A\to B$ belongs to $\cL$, then so is the projection $p_1:A\times_BA\to A$, since $\cL$ is closed under base change.
The class $\cL$ has the 3-for-2 property since it is strongly saturated, hence we have $\Delta(u)\in \cL$, since $p_1\Delta(u)=1_A\in \cL$.

\smallskip
\noindent \eqref{charac-cong:3}  
$\Rightarrow$ \eqref{charac-cong:1}
It suffices to show that the class $\cL$ is closed under 
pullbacks in $\cE\arr$. But $\cL$ contains
the isomorphisms, and it is closed
under composition and base change, since $\cL$
is acyclic. Thus $\cL$ is closed under 
pullbacks in $\cE\arr$ by \Cref{finitelimclosure},
since $\Delta(\cL)\subseteq \cL$.
\end{proof}

\begin{remark}
In the case where $\cL = \Sigma\ac$, \cref{charac-cong} can be deduced from \cref{lexcharacterization}, since then $\cL$ is the left class of a modality by \cref{thmgenerationmodality}.
\end{remark}

\begin{exmps}
\label{exmpcongruence}
Let $\cE$ be a topos.

\begin{exmpenum}
\item \label{exmpcongruence1} The classes $\Iso$ and $\All$ are respectively the smallest and the largest congruences (for the inclusion relation).

\item \label{exmpcongruence2} By \cref{cor:mono-gen}, if $\Sigma$ is a class of monomorphism, then $\Sigma\ac$ is a congruence. We shall see in \cref{sheaves4mono} that in this case $\Sigma\ac=\Sigma\cong$.

\item \label{exmpcongruence4} Let $\phi:\cE\to \cF$ be a cocontinuous left-exact functor between topoi.
Then for any congruence $\cW$ in $\cF$ the class $\phi^{-1}(\cW) =\{f\in \cE\ |\ \phi(f)\in \cW\}$ is a congruence in $\cE$.
In particular, the class $\cL_\phi:= \phi^{-1}(\Iso)$ of maps inverted by $\phi$ is a congruence.

\item \label{exmpcongruence3} The class $\Conn_\infty$ of \oo connected maps is a congruence. We know from \cref{exmp-acyclic:3} that $\Conn_\infty$ is acyclic. Recall that a map is \oo connected if and only if all its iterated diagonal are surjective maps \cite[Proposition 6.5.1.18]{HTT}, then it is clear that $\Delta(\Conn_\infty)\subseteq\Conn_\infty$. This proves that $\Conn_\infty$ is a congruence by \Cref{charac-cong}\eqref{charac-cong:3}. This congruence is the one associated to the hypercompletion of $\cE\to \cE^h$ by \cref{exmpcongruence4} \cite[Lemma 6.5.2.10 and Proposition 6.5.2.13]{HTT}. 

\end{exmpenum}
\end{exmps}

\begin{proposition}
\label{congurencevsmod}
A factorization system $(\cL,\cR)$ in a topos $\cE$ is a left-exact modality if and only if the class $\cL$ is a congruence.
\end{proposition}

\begin{proof}
If the factorization system $(\cL,\cR)$ is a left-exact modality, then the class  $\cL$ is saturated by \Cref{leftorthissat}
and it is closed under finite limits by 
\Cref{lexcharacterization}. Thus, $\cL$ 
is a congruence.
 Conversely, if the left class $\cL$ is a congruence, then the factorization system $(\cL,\cR)$  is left
exact by \cref{lem:def-lex-mod}, since $\cL$ is closed under finite limits. It is thus a modality by 
\Cref{propexcataremodal}.
\end{proof} 

\begin{lemma}
\label{congruenceofafunctor}
Let $\phi:\cE\to \cF$ be a cocontinuous left-exact functor between topoi.
Then the class $\cL_\phi$ of maps inverted by $\phi$ is a congruence of small generation.
\end{lemma}

\begin{proof} 
The class $\cL_\phi$ is strongly saturated by \cref{stronglysatcocont} since the functor $\phi$
is cocontinuous, and it is closed under finite limits, since the functor $\phi$ preserves finite limits.
This proves it is a congruence.
As a strongly saturated class, $\cL_\phi$ is of small generation by \cref{stronglysatcocont}.
Therefore we have $\cL_\phi=\Sigma^{ss}$ for a set of maps $\Sigma\subseteq \cL_\phi$.
Moreover, we have $\Sigma^c\subseteq \cL_\phi$, since $\Sigma\subseteq \cL_\phi$ and $\cL_\phi$ is a congruence.
On the other hand, $\Sigma^{ss}\subseteq \Sigma^{c}$, since the class $\Sigma^{c}$ is strongly saturated by \cref{congruenceisstrongsat}.
Thus $\Sigma^{c}=\Sigma^{ss}=\cL_\phi$.
\end{proof}

\subsubsection{Lurie's results on congruences}
\label{sec:Lurie}

\Cref{congruenceisstrongsat} shows that the notion of congruence defined above is equivalent to that of strongly saturated class of maps closed under base change considered by Lurie in \cite{HTT}.
This notion is introduced to construct the cocontinuous left-exact localization $\cE\to \LOC \cE \Sigma \cclex$ of a topos $\cE$ generated by a set of maps $\Sigma$.
The condition of closure under base change in Lurie's definition is motivated by the following fact.

\begin{lemma}[{\cite[Proposition 6.2.1.1]{HTT}}]
\label{Lurieprop1} 
Let $\cE$ be a category with finite limits
 and $\rho:\cE\to \cE'\subseteq \cE$ 
be a reflector. Then the functor $\rho$
is left-exact if and only if the 
class of maps $\cW_\rho$  inverted by $\rho$
is closed under base change.
\end{lemma} 

From there Lurie considers the smallest strongly saturated class closed under base change $\overline\Sigma$ containing $\Sigma$ (notice that, by \Cref{congruenceisstrongsat}, we have $\overline\Sigma=\Sigma \cong$).
Then the idea is to show that $\LOC \cE \Sigma \cclex = \Loc \cE {\overline\Sigma}$, but we need to know that $\overline\Sigma$ is of small generation as a saturated class.
This is the purpose of the following result.

\begin{lemma} [{\cite[Proposition 6.2.1.2]{HTT}}] 
\label{Lurielemma1}
A congruence $\cW$ in a topos $\cE$ is of small generation as a congruence
if and only if it is of small generation as a strongly saturated class.
\end{lemma} 

By combining \cref{Lurielemma1} with a result of generation of localization of presentable categories (see \cref{localizationofpresentable}), Lurie gets the following theorem.

\begin{theorem}[{\cite[Propositions 5.5.4.15, 6.2.1.1 and 6.2.1.2 together]{HTT}}]
\label{Lurietheorem2}
Let $\Sigma$ is a set of maps in a topos $\cE$.
Then the full subcategory $\Loc \cE {\overline\Sigma}$ is reflective,  
is a topos, and the reflector $\rho:
\cE\to \Loc \cE {\overline\Sigma}$
is a cocontinuous left-exact
localization $\cE\to \LOCcclex \cE \Sigma$.
Moreover, $\overline\Sigma$ is the class of maps inverted by $\rho$.
\end{theorem}

The following theorem shows that the notion of congruence plays a role similar to that of Grothendieck topologies in controlling left-exact localizations.
Contrary to the case of 1-topoi, it is not known whether all left-exact localizations of topoi are accessible.
Therefore, a condition of small generation must be imposed.
For $\cE$ a fixed topos, we consider
the class $\mathsf{LexLoc_{acc}}(\cE)$ of (isomorphism classes of) accessible cocontinuous left-exact localizations of $\cE$
and the class $\mathsf{Cong_{sg}}(\cE)$ of congruences of small generation in $\cE$.
The correspondence $\phi\mapsto \cL_\phi$ of \Cref{congruenceofafunctor} defines a morphism $\mathsf{LexLoc_{acc}}(\cE) \to \mathsf{Cong_{sg}}(\cE)$.
Conversely, if $\cW=\Sigma\cong$ is a congruence of small generation, then $\cW$ is also of small generation as a strongly saturated class  by \Cref{Lurielemma1}.
Hence the left-exact localization $\phi_\cW:\cE\to \Loc \cE {\Sigma\cong} = \LOC \cE \Sigma \cclex$ is accessible in the sense of \cref{def:accessible} and we get a function
$\mathsf{Cong_{sg}}(\cE) \to \mathsf{LexLoc_{acc}}(\cE)$.

\begin{thm}[{\cite[Propositions 5.5.4.2 and 6.2.1.1 together]{HTT}}]
\label{thm:bij-congruence-lexloc}
The functions $\phi\mapsto \cL_\phi$ and $\cW \mapsto \phi_\cW$ define inverse bijections
\[
\mathsf{LexLoc_{acc}}(\cE) \simeq\mathsf{Cong_{sg}}(\cE)\, .
\]
\end{thm}

\medskip
\Cref{Lurietheorem2} says that the left-exact localization $\cE\to \Loc \cE {\overline\Sigma}$ 
is a localization 
$\cE\to \Loc \cE {\Sigma'}$ 
with respect to some set of maps $\Sigma'$ in $\cE$.
A fundamental problem is then to provide an explicit description of $\Sigma'$ in terms of the original set $\Sigma$.
Such a description is well-known in the case of 1-topos theory.
The main goal of this paper, and the purpose of the next section, is to provide a solution for higher topoi.

\subsubsection{Generation of congruences}

We now study the question of constructing the congruence $\Sigma\cong$ generated by a class of maps $\Sigma$.
For a class $\Sigma$ of maps in a topos $\cE$, we write
\[
\Sigma\diag = \{ \Delta^i u \, | \, u \in \Sigma, i\geq 0 \}
\]
for the \emph{diagonal closure} of $\Sigma$.
That is, the class of all diagonals of all maps in $\Sigma$.
If $\Sigma$ is a set, then so is $\Sigma\diag$.

\begin{proposition} 
\label{congruencegenerated}
If $\Sigma$ is a class of maps in a topos $\cE$, then $\Sigma\cong=(\Sigma\diag)\ac$.
\end{proposition}

\begin{proof}
We first prove the result when $\Sigma$ is a set.
In this case $\Lambda:=\Sigma\diag$ is also a set.
Hence the pair $(\Lambda\ac,\Lambda^\fperp)$ is a modality by \cref{thmgenerationmodality}.
Furthermore, the modality $(\Lambda\ac,\Lambda^\fperp)$ is left-exact by \cref{thm:recognition}, since $\Delta(\Lambda) \subseteq \Lambda \subseteq \Lambda\ac$.  
So by \cref{congurencevsmod} the class $\Lambda\ac$ is a congruence.
From $\Sigma\subseteq \Lambda\ac$ it therefore follows $\Sigma\cong\subseteq \Lambda\ac$.
Conversely, by \cref{congclosediag} the class $\Sigma\cong$ is closed under diagonals. 
As $\Sigma \subseteq \Sigma\cong$ we have $\Lambda\subseteq \Sigma\cong$. 
But then $\Lambda\ac\subseteq \Sigma\cong$, since the class $\Sigma\cong$ is acyclic by \cref{congruenceisstrongsat}.
This shows that $\Sigma\cong=\Lambda\ac=(\Sigma\diag)\ac$.

We now extend the result for $\Sigma$ a class.
Let $\kappa$ be the inaccessible cardinal bounding the size of small objects (\cref{sec:convents}).
Then $\Sigma$ is a $\kappa$-filtered union of subsets $\Sigma_i$.
Moreover, we can always assume that $\Delta(\Sigma_i)\subset\Sigma_i$.
A $\kappa$-filtered union of congruences (or acyclic classes) is still a congruence (or acyclic).
Hence $\Sigma\cong = \bigcup_i \Sigma_i\cong=\bigcup_i(\Sigma_i\diag)\ac = (\Sigma\diag)\ac$.
\end{proof}

\begin{cor}
\label{cor:generation-congruence}
Let $\cG$ be a class of generators in a topos $\cE$.
For $\Sigma$ a class of maps in $\cE$, we have 
\[
\Sigma\cong = (\Sigma\diag)\ac = ((\Sigma\diag)\bc)\sat\,,
\]    
where $(\Sigma\diag)\bc$ is the class of all $\cG$-base changes of $\Sigma\diag$.
\end{cor}
\begin{proof}
The first equality is \cref{congruencegenerated}. 
The second comes from \cref{cor:generation-acyclic-class}.
\end{proof}

\begin{proposition}
\label{prop:gen-lex-mod}
Let $\Sigma$ be a set of maps in a topos $\cE$.
Then the pair $(\Sigma\cong, (\Sigma\diag)^\fperp)$ is a left-exact modality.
It is the smallest left-exact modality $(\cL,\cR)$ such that $\Sigma\subseteq\cL$.
\end{proposition}

\begin{proof}
Let us put $\Lambda:=(\Sigma\diag)$.
The pair $(\cL,\cR):=(\Lambda\ac,\Lambda^\fperp)$ is a modality by \cref{thmgenerationmodality}.
Furthermore, the modality $(\Lambda\ac,\Lambda^\fperp)$ is left-exact by \cref{thm:recognition}, since $\Delta(\Lambda) \subseteq \Lambda \subseteq \Lambda\ac$. 
If $(\cL,\cR)$ is a left-exact modality such that $\Sigma\subseteq\cL$,
then $\cL$ is a congruence by \cref{congurencevsmod} and $\Sigma\cong \subseteq\cL$.
The minimality statement follows.
\end{proof}

\subsection{Higher sheaves and higher sites}
\label{sec:higher-sheaves}

\begin{definition}[$\Sigma$-sheaf]
\label{defn:sigmasheaf}
Let $\Sigma$ be a class of maps in a topos $\cE$.
We say that an object $X \in \cE$ is a \emph{$\Sigma$-sheaf} if the map $\Map {u'} X : \Map {B'} X \to \Map {A'} X$ is invertible for every base change $u' : A' \to B'$ of a map $u \in \Sigma\diag$.
We write $\Sh \cE \Sigma$ for the full subcategory of $\Sigma$-sheaves.
Equivalently, $\Sigma$-sheaves are the modal objects for the class $\Sigma\diag$, and $\Sh \cE \Sigma = \Mod \cE {\Sigma\diag}$.
\end{definition}

The notion of $\Sigma$-sheaf defined above depends on the class of \emph{all} base changes of the maps in $\Sigma$.
The following lemma shows how to reduce this.
%
Let $\cE$ be a topos and $\cG$ a set of generators.
Recall the notion of $\cG$-base change from \cref{defG-basechange}

\begin{lemma}
\label{lemmadef:sigmasheaf}
An object $X \in \cE$ is a $\Sigma$-sheaf if and only if it is local with respect to the class $(\Sigma\diag)\bc$ of $\cG$-base changes of the maps in $\Sigma\diag$.
Moreover, when $\Sigma$ is a set, so is $(\Sigma\diag)\bc$.
\end{lemma}

\begin{proof}
We write $\Lambda:= (\Sigma\diag)\bc$ for short, it is clear that $\Lambda$ is a set if $\Sigma$ is.
The class $\Lambda$ is closed under $\cG$-base change by transitivity of base change.
Thus, $\Lambda^\upvdash=\Lambda^{\fperp}$ by \cref{saturatedacyclic}.
By definition, an object $X\in \cE$ is a $\Sigma$-sheaf if and only if it is modal with respect to the set $\Sigma\diag$.
However, every base change of a map in $\Lambda$ is a base change of a map in $\Sigma$, and conversely.
Thus, $X$ is a $\Sigma$-sheaf if and only if it is modal with respect to the set $\Lambda$.
But this happens if and only if $X$ is local with respect to $\Lambda$, since $\Lambda^\upvdash=\Lambda^{\fperp}$.
\end{proof}

As an example, let $\Prsh \cK=[\cK\op,\cS]$ be the category of presheaves on a small category $\cK$.
The category $\Prsh \cK$ is presentable and generated by the set of representable presheaves $R(\cK)\subseteq\Prsh \cK$.
Hence given a set of maps $\Sigma \subseteq \Prsh \cK$, in order to check that a presheaf $X$ on $\cK$ is a $\Sigma$-sheaf, it suffices to check that is it local with respect to the {\it set} of all $R(\cK)$-base changes of the maps in $\Sigma\diag$.

\begin{theorem}
\label{univlexloc}
Let $\Sigma$ be a set of maps in a topos $\cE$.
The full subcategory $\Sh \cE \Sigma\subseteq\cE$ admits an accessible left-exact reflector $\cE\to \Sh \cE \Sigma$ which coincides with the cocontinuous left-exact localization $\cE\to \LOCcclex \cE \Sigma$.
In particular $\Sh \cE \Sigma$ is a topos.
Moreover, the class of maps inverted by the reflector $\rho$ is the congruence $\Sigma\cong$ generated by $\Sigma$.
\end{theorem}

\begin{rem}
 It follows that we have canonical equivalences
\begin{align*}
\Sh \cE \Sigma
&= \Mod \cE {\Sigma\diag}
= \Loc \cE {(\Sigma\diag)\bc}\\
&= \LOC \cE {(\Sigma\cong)}
= \LOCcc \cE {((\Sigma\diag)\bc)}
= \LOCcclex \cE \Sigma
\end{align*}
as categories under $\cE$.
\end{rem}

\begin{proof}
Let us put $\Lambda:=\Sigma\diag$.
By definition, an object $X$ is a $\Sigma$-sheaf if and only if it is modal with respect to $\Lambda$.
In other words, we have $\Sh \cE \Sigma =\Mod \cE \Lambda$.
The pair $(\cL,\cR):=(\Lambda\ac,\Lambda^\fperp)$ is a left-exact modality by \cref{prop:gen-lex-mod}
It follows that the reflector $\rho=\|-\|:\cE\to \cR[1]=\Sh \cE \Sigma$ is left-exact by \cref{lexcharacterization} and hence that the category $\Sh \cE \Sigma$ is a topos by \cref{Rezkcharacterization}.
Moreover, the class of maps inverted by the reflector $\rho=\|-\|:\cE\to \cR[1]$ is equal to $\cL=\Lambda\ac$ by \cref{lexcharacterization}.
But \cref{congruencegenerated} says that $\Lambda\ac=\Sigma\cong$. 
So the class of maps inverted by the reflector $\rho$ is the congruence $\Sigma\cong$ generated by $\Sigma$.
By \cref{prop:loc-vs-reflection}, it follows that $\Sh \cE \Sigma = \LOC \cE {(\Sigma\cong)}$.
Recall from \cref{cor:generation-congruence} that 
$\Sigma\cong = ((\Sigma\diag)\bc)\sat$.
Then the equivalence 
$\LOCcc \cE {((\Sigma\diag)\bc)} = \LOC \cE {(\Sigma\cong)}$
follows from \cref{lem:equiv-ccloc}.

We are left to prove that $\Sh \cE \Sigma = \LOCcclex \cE \Sigma$.
Suppose that $\cF$ is a topos and that $\phi:\cE\to \cF$ is a cocontinuous left-exact functor   which inverts every map in $\Sigma$.
The class of maps $\cW\subseteq \cE$ inverted by the functor $\phi$ is a congruence by \cref{congruenceofafunctor}, hence $\Sigma\cong\subseteq \cW$.
Using $\Sh \cE \Sigma = \LOC \cE {(\Sigma\cong)}$, there exists a unique functor $\psi:\Sh \cE \Sigma \to \cF$ such that $\psi\circ \rho =\phi$.
\[
\begin{tikzcd}
\cE \ar[drr,"{\phi}"' ] \ar[rr, "{\rho}"] && \Sh \cE \Sigma
\ar[d, "{\psi}"] \\
&& \cF
\end{tikzcd}
\]
The functor $\psi$ is cocontinuous left-exact, since the functors $\psi \circ \rho$ and $\rho$ are cocontinuous left-exact.
\end{proof}

\begin{definition}
\label{def:topo-cong}
We shall say that a congruence $\cW$ is {\it topological} if there exists a class $\Sigma \subset \Mono$ of monomorphisms such that $\cW=\Sigma\cong$.
\end{definition}

A cocontinuous left-exact localization $\phi:\cE\to \cE'$ is topological in the sense of \cite[Definition 6.2.1.4]{HTT} if and only if the associated congruence $\cL_\phi$ of \cref{congruenceofafunctor} is topological.

\medskip
The following result says that $\Sigma$ can always be chosen a set in \cref{def:topo-cong}, and that $\Sigma$-sheaves are simply the $\Sigma$-modal objects.

\begin{proposition}
\label{lem:mono-ac=cong}
If $\Sigma$ is a class of monomorphisms in a topos $\cE$, then $\Sigma\cong = \Sigma\ac$.
In particular, we have $\Sh \cE \Sigma = \Mod \cE \Sigma$.
\end{proposition}
\begin{proof}
Since $\Sigma$ consists of monomorphisms, the diagonal $\Delta(u)$ of a map in $u\in \Sigma$ is invertible.
Thus, $\Sigma\diag=\Sigma\cup \cI$, where $\cI$ is a set of isomorphisms.
Using \Cref{congruencegenerated}, this proves that $\Sigma\cong = \Sigma\ac$.
The equality $\Mod \cE {\Sigma\diag} = \Mod \cE \Sigma$ follows.
\end{proof}

\begin{proposition}
\label{sheaves4mono}
If $\Sigma$ is a class of monomorphisms in a topos $\cE$, the congruence $\Sigma\cong$ is of small generation and is the left class of a left-exact modality.
In particular, the full subcategory $\Sh \cE \Sigma\subseteq\cE$ admits a left-exact reflector $\cE\to \Sh \cE \Sigma$ which coincides with the cocontinuous left-exact localization $\cE\to \LOCcclex \cE \Sigma$.
\end{proposition}
 
\begin{proof}
For the small presentation, recall from \cref{congruenceisstrongsat} that any congruence is a strongly saturated class stable by base change, then the result is \cite[Proposition 6.2.1.5]{HTT}.
For the modality statement, we need to prove that $(\Sigma\cong,\Sigma^\fperp)$ is a factorization system.
Using the small generation hypothesis, this can be done by a small object argument.
The class $\Sigma\cong$ is stable by finite limits by definition of congruences.
Hence the factorization system is left-exact by \cref{lem:def-lex-mod}.
Finally, the left-exact reflector of $\cE\to \Sh \cE \Sigma$ is given by the reflector $\|-\|:\cE\to \cR[1]$ associated to the left-exact modality $(\cL,\cR)=(\Sigma\cong,\Sigma^\fperp)$ by \cref{lem:lex-reflector}.
The end of the statement is \cref{univlexloc}.
\end{proof}

In \cite{Rezk:topos}, a \emph{model site} is defined to be a pair $(\cK,\Sigma)$ where $\cK$ is a small category, $\Sigma$ is a class of maps in $\Prsh \cK$ satisfying the condition that the localization generated by $\Sigma$ is left-exact.
The results of this article, however, provide us with concrete methods of presenting topoi by means of generators and relations, and thus lead us to propose a
modification of the notion of higher site.
Specifically, we lift the requirement that the localization generated by $\Sigma$ be left-exact, since the definition of $\Sigma$-sheaf provides us with a means of recognizing the local objects for the left-exact localization which $\Sigma$ generates.
Our notion of site also extends that of~\cite{TV05} which is only suited for topological left-exact localizations.

\begin{definition}[Site]
\label{defsite}
A \emph{site} is defined to be a pair $(\cK,\Sigma)$, where $\cK$ is a small category and $\Sigma$ is a set of maps in $\Prsh \cK$.
We shall say that a presheaf $X$ on $\cK$ is a \emph{sheaf} on $(\cK,\Sigma)$ if $X$ is local with respect to the $R(\cK)$-base changes of the maps in $\Sigma\diag$.
We shall denote the category of sheaves on $(\cK,\Sigma)$ by $\Shsite \cK \Sigma$.
\end{definition}

Every topos $\cE$ is a left-exact localization of a presheaf category $\Prsh \cK$ \cite[Definition 6.1.0.4]{HTT}.
If $\cL$ is the class of maps inverted by the localization functor $\rho:\Prsh \cK\to \cE$, then we have $\cE = \Loc {\Prsh \cK} \cL$.
The class $\cL$ is a congruence of small generation by \cref{congruenceofafunctor}.
Hence we have $\cL=\Sigma\cong$ for a set of maps $\Sigma \subseteq \Prsh \cK$.
But we have $\Sigma\cong=(\Sigma\diag)\ac$ by \cref{congruencegenerated}.
Thus, a presheaf $X\in \Prsh \cK$ is local with respect to $\cL$ if and only if it is a $\Sigma$-sheaf.
Thus, $\Loc {\Prsh \cK} \cL = \Sh {\Prsh \cK} \Sigma = \Shsite \cK \Sigma$.
This shows that the topos $\cE$ is equivalent to the category of sheaves on the site $(\cK,\Sigma)$.

\medskip
We now compare our notion of sheaves with the one of \cite[Definition 6.2.2.6]{HTT}, defined with respect to a Grothendieck topology.
For $C\in \cK$, let $\hat C$ denote the functor $\Map - C :\cK\op\to \cS$.
A sieve of $C$, is a monomorphism $S\hookrightarrow \hat C$ in $\Prsh \cK$.
A Grothendieck topology $J$ on $\cK$ (see \cite[Definition 6.2.2.1]{HTT}) is, for each object $C\in \cK$, a collection $J(C)$ of sieves of $C$ (called covering sieves) satisfying a set of conditions which includes that $J$ is closed under base changes: if $f:C\to D$ is a map in $\cK$ and $S\hookrightarrow \hat D$ is in $J(D)$, then $f^\star(S)\hookrightarrow \hat C$ is in $J(C)$.
A presheaf $X\in \Prsh \cK$ is a {\it sheaf} for the topology $J$ if it is local with respect to all sieves in $J$.
We denote the category of sheaves for $J$ by $\Shsite \cK J$.
The following lemma shows that our notion of sheaf extends the one of \cite{HTT}.

\begin{lemma}
\label{lem:site=site}
Let $\tau$ be a Grothendieck topology on a small category $\cK$, and let $\Sigma=\bigcup_C J(C)$ be the set of all covering sieves in $\Prsh \cK$, then $\Shsite \cK J = \Shsite \cK \Sigma$.
\end{lemma}
\begin{proof}
By definition, $\Shsite \cK \tau = \Loc {\Prsh\cK} \Sigma$.
The presheaf category $\Prsh \cK$ is generated by the set of representable presheaves $R(\cK)\subseteq\Prsh \cK$.
We have seen that the set of maps $\Sigma\subseteq  \Prsh \cK$ is closed under $R(\cK)$-base changes by definition of a topology.
Hence we have $\Sigma^\upvdash =\Sigma^\fperp$ by \cref{saturatedacyclic}.
It follows by \cref{proporthvslocal} and \cref{defmodalobject} that 
$\Loc {\Prsh\cK} \Sigma = \Mod {\Prsh\cK} \Sigma$.
But we have $\Mod {\Prsh\cK} \Sigma = \Sh {\Prsh\cK} \Sigma$ by \cref{sheaves4mono}.
Thus, $\Shsite \cK \tau =\Sh {\Prsh\cK} \Sigma =\Shsite \cK \Sigma$.
\end{proof}
 
\begin{proposition}
\label{prop:any-topos}
Any topos $\cE$ can be presented as a category of sheaves on a site $(\cK,\Sigma)$.
\end{proposition}
\begin{proof}
By definition of topoi, $\cE$ is an accessible cocontinuous left-exact localization of some category $\phi:\Prsh \cK \to \cE$.
Let $\cW\subseteq \Prsh \cK$ be the class of maps inverted by $\phi$, it is a congruence by \cref{exmpcongruence4}.
Moreover, it is of small generation as a strongly saturated class by the hypothesis of accessibility.
Let $\Sigma \subseteq \cW$ be a set of maps such that $\Sigma\ssat = \cW$.
The statement will be proven if we show that $\Sigma\cong = \cW$.
Since $\cW$ is a congruence, we have $\Sigma\cong \subseteq \cW$.
Conversely, we have $\Sigma\ssat\subseteq \Sigma\cong$ since any congruence is strongly saturated by \cref{congruenceisstrongsat}.
\end{proof}

\section{Applications}
\label{sec:examples}

This final section gives some applications of \cref{univlexloc} and the notions of congruence and  $\Sigma$-sheaves.
We introduce first the topoi $\cS[X] = \fun{\Fin}\cS$ and $\cS[X\pointed] = \fun{\Fin\pointed}\cS$, where $\Fin$ is the category of finite spaces and $\Fin\pointed$ the category of pointed finite spaces.
We will prove that $\cS[X]$ is freely generated by an object $X$, 
and that $\cS[X\pointed]$ is freely generated by an object $X\pointed$. 
We will then consider several of their localizations. 

All these topoi will represent certain functors defined on the opposite category of topoi.
Precisely, we define the category $\Toposop$ as the category whose objects are topoi, and morphisms are the cocontinuous left-exact functors.
In the notations of \cite{HTT}, $\Toposop = \mathrm{\cL\cT op}$.
In the notations of \cite{AJ}, $\Toposop = \mathsf{Logos}$.
If $\cE$ and $\cF$ are two topoi, we shall denote by $\fun \cE \cF \cclex$ the category of cocontinuous left-exact functors between them.
Any topos $\cE$ defines a functor $\Toposop \to \CAT$ sending $\cF$ to the category $\fun \cE \cF \cclex$.
We shall say that a functor $F:\Toposop \to \CAT$ is representable by a pair $(\cE,\eta)$, where $\cE$ is a topos and $\eta\in F(\cE)$, if the natural transformation
\begin{align*}
\fun \cE \cF \cclex &\longrightarrow F(\cF)\\
(f:\cE\to \cF)  &\longmapsto F(f)(\eta)
\end{align*}
induce an equivalence $\fun \cE \cF \cclex \simeq F(\cF)$ of categories for all $\cF$, that is an equivalence of functors $\fun \cE - \cclex \simeq F$.
The pair $(\cE,\eta)$ representing $F$ is unique up to canonical isomorphism.
If $(\cE',\eta')$ is another one, then the equivalences 
$F(\cE') \simeq \fun \cE {\cE'} \cclex$ and $F(\cE) \simeq \fun {\cE'} \cE \cclex$
send $\eta'$ and $\eta$ to inverse equivalences $\cE \to \cE'$ and $\cE \to \cE'$.
By abuse, we shall sometimes say that a functor is represented by a topos $\cE$, leaving implicit the choice of $\eta$ (this choice hopefully being clear from the context).

\subsection{The topoi classifying objects and pointed objects}

We introduce first the topos classifying objects. 
Let $\Fin$ be the category of finite spaces. 
We write $\cS[X]$ for the topos $\fun \Fin \cS$.
This topos comes equipped with the distinguished object $X:\Fin \to \cS$ given by the canonical inclusion.
The functor $X$ is represented by the one point space $1\in \Fin$.
More generally, for $K\in \Fin$, we denote by $X^K$ the functor represented by $K$.
In particular, if 0 is the initial object of $\Fin$ (the empty space), $X^0$ is the terminal object of $\cS[X]$.
Let $\cG$ be the set of all $X^K$ for $K\in \Fin$, then $\cG$ generates $\cS[X]$ as a presentable category.

\begin{proposition}[Object classifier]
\label{prop:SX}

For any topos $\cE$ and any $E\in \cE$, there is a unique cocontinuous left-exact functor $\cS[X] \to \cE$ sending $X$ to $E$.
More precisely, the evaluation at $X$ induces a functor $\fun {\cS[X]} \cE \cclex \to \cE$ which is an equivalence of categories.
In other words, the forgetful functor
\begin{align*}
\bA:\Toposop &\longrightarrow \CAT\\
\cE &\longmapsto \cE
\end{align*}
is representable by the pair $(\cS[X], X)$.
\end{proposition}

\begin{proof}
Let us prove that the evaluation functor $\fun {\cS[X]} \cE \cclex \to \cE$ is an equivalence of categories by constructing an explicit inverse functor.
Recall that $\cS[X] = \fun \Fin \cS = \Prsh {\Fin\op}$.
Any object $E\in \cE$ defines a functor $E:1\to \cE$ where 1 is the terminal category.
Using that the category $\Fin\op$ is the free completion of 1 for finite limits, the functor $E$ is equivalent to a lex functor $E\lex:\Fin\op\to \cE$.
The Yoneda embedding $\Fin\op\to \fun {\Fin}\cS$ is a lex functor. 
The left Kan extension of a lex functor $\Fin\op \to \cE$ with values in a topos is again a lex functor (see \cite[Proposition 6.1.5.2]{HTT} or \cite[Thm 2.1.4]{Anel-Lejay:topos-exp}) and induces an equivalence between lex functors $\Fin\op\to \cE$ and cocontinuous lex functors $\cS[X] \to \cE$.
We leave to reader the verification that the resulting functor $\cE\to \fun {\cS[X]} \cE\cclex$ is an inverse to the evaluation at $X$.
\end{proof}

\begin{rem}
\label{rem:objet-univ}
The previous result says that $\cS[X]$ is freely generated by the object $X$.
For this reason, $\cS[X]$ is often called the ``object classifier'' and $X$ is often called the ``universal object''.
From there, any left-exact localization of $\cS[X]$ will distinguish objects with some special property.
In an analogy with commutative algebra, the topos $\cS[X]$ can be compared to the free ring $\ZZ[x]$ and the universal object $X$ to the variable element $x$.
The left-exact localizations of $\cS[X]$ are then analogues of the quotients of the ring $\ZZ[x]$ and the congruences of $\cS[X]$ are analogues of ideals in $\ZZ[x]$.
In the same way that we can impose $x$ to satisfy a relation $P(x)=0$ by taking the quotient ring $\ZZ[x]/(P(x))$,
we can impose $X\in \cS[X]$ to satisfy some relation by taking a left-exact localizations of $\cS[X]$.
We shall study below the localizations of $\cS[X]$ generated by forcing $X$ to become contractible,  $n$-truncated, and $n$-connected.
\end{rem}

We now introduce the topos classifying pointed objects. 
A pointed object in a topos $\cE$ is be a pair $(E,e)$ where $E$ is an object of $\cE$ and $e:1\to E$ a map from the terminal object of $\cE$.
We define the category of pointed objects in $\cE$ as the coslice $\cE\pointed :=1\backslash\cE$.
Any cocontinuous left-exact functor $\cE\to\cF$ between topoi preserves the terminal object and induces a functor $\cE\pointed \to \cF\pointed$.
This defines a functor
\begin{align*}
\bA\pointed :\Toposop &\longrightarrow \CAT\\
\cE &\longmapsto \cE\pointed .
\end{align*}

Let $\Fin\pointed$ be the category of pointed finite spaces.
The topos $S[X\pointed] :=\fun {\Fin\pointed} \cS$ is equipped with a distinguished object $X\pointed:\Fin\pointed\to \cS$ which is the functor forgetting the base point.
Since $\Fin\pointed = 1\backslash \Fin$, we have $\cS[X\pointed] = \fun{1\backslash\Fin}\cS = \fun{\Fin}\cS\slice X = \cS[X]\slice X$ (slice topos).
Let us denote the objects of $\cS[X]\slice X$ by pair $(F,\phi)$ where $\phi:F\to X$ is a map in $\cS[X]$.
With this notation, the terminal object of $\cS[X]\slice X$ is $(X,id)$,
and the object $X\pointed$ corresponds to $(X\times X,p_1)$. 
The canonical morphism $\cS[X]\to \cS[X\pointed]$ sending $X$ to $X\pointed$ corresponds to $p^*:\cS[X]\to \cS[X]\slice X$, the pullback along $p:X\to 1$.
The base point of $X\pointed$ is given by the diagonal $\Delta X:X\to X\times X$ viewed as morphism $(X,id) \to (X\times X,p_1)$.
The object $X\pointed$ is represented by $S^0\in \Fin\pointed$, whereas $1\in \Fin\pointed$ represents the terminal object of $\cS[X\pointed]$.
The base point $1\to X\pointed$ is represented by the map $S^0\to 1$.

Let $\cE$ be a topos and $E$ an object in $\cE$.
The category $\cE\slice E$ is still a topos and the cartesian product with $E$ induces a cocontinuous left-exact functor $E^*:\cE\to \cE\slice E$.
The object $E^*(E)=(E\times E,p_1)$ is equipped with a canonical section given by the diagonal map $\Delta E:E\to E\times E$.
This canonical section has a universal property \cite[Proposition 6.3.5.5]{HTT}.
For any topos $\cF$, any cocontinuous left-exact functor $\phi:\cE\to \cF$, and any map $e:1\to \phi(E)$, there exists a unique cocontinuous left-exact functor $\phi':\cE\slice E \to \cF$ such that $\phi = \phi'E^*$ and $e = \phi'(\Delta E)$.

\begin{proposition}[Pointed object classifier]
\label{prop:SXpointed}
The functor $\bA\pointed$
is representable by the pair $(\cS[X\pointed], X\pointed)$.
\end{proposition}
\begin{proof}
By the universal property of the slice topos $\cS[X\pointed]=\cS[X]\slice X$, a cocontinuous left-exact functor $\cS[X\pointed]\to \cE$ is equivalent to a cocontinuous left-exact functor $\phi:\cS[X]\to \cE$ together with a map $1\to \phi(X)$ in $\cE$.
The result follows.
\end{proof}

\subsection{The topos classifying contractible objects}
\label{sec:contr}

We shall first illustrate our notion of $\Sigma$-sheaves and \cref{univlexloc} by looking at a simple example of left-exact localization.
An object $E$ of a topos $\cE$ is {\it contractible} if the canonical map $E\to 1$ is an equivalence, that is if $E$ is a terminal object.
Since there is a unique terminal object in a topos, the functor classifying contractible objects is the constant functor
\begin{align*}
\bone:\Toposop &\longrightarrow \CAT\\
\cE &\longmapsto 1
\end{align*}
Because $\cS$ is the initial object in $\Toposop$, this functor is represented by the pair $(\cS,1)$ where 1 is the terminal object in $\cS$.

The condition of being contractible is geometric in the sense that a cocontinuous left-exact functor preserves terminal objects.
By the considerations of \cref{rem:objet-univ}, the condition of contractibility can be enforced universally by forcing the universal object $X\in \cS[X]$ to become contractible, that is, by quotienting $\cS[X]$ by the congruence generated by the single map $X\to 1$.

It is easy to see that this localization of $\cS[X] = \fun \Fin \cS$ is given by the evaluation functor $\epsilon^*:\cS[X]\to \cS$ at the one point space $1\in\Fin$.
The object 1 being terminal in $\Fin$, this functor coincides with the colimit functor $\cS^{\Fin}\to \cS$.
Hence $\epsilon^*$ has a right adjoint $\epsilon_*$ given by the constant sheaf functor, and this latter functor is fully faithful since $\epsilon^*\epsilon_*= id$.
This proves that $\epsilon^*$ is a left-exact localization and that the corresponding full subcategory of $\cS[X]$ is the category of constant functors.

\begin{proposition}
\label{prop:contractible-gen}
The cocontinuous left-exact localization $\epsilon^*:\cS[X]\to \cS$ is universal for inverting the map $X\to 1$, that is
\[
\LOC{\cS[X]}{(X\to 1)} = \cS.
\]
In particular, it classifies contractible objects and the universal contractible object is $1\in \cS$.
\end{proposition}

It is instructing to re-prove \cref{prop:contractible-gen} using \cref{univlexloc}.
According to this theorem, the statement that $\LOC{\cS[X]}{(X\to 1)} = \cS$ is equivalent to the statement that the full subcategory of constant functors, coincides with the subcategory of $\Sigma$-sheaves for $\Sigma=\{X\to 1\}$.
Let $\cW$ be the congruence associated to $\epsilon^*$.
It is the class of maps $F\to G$ in $\cS[X]$ such that $F(1)\simeq G(1)$.
We have $\Sigma \subseteq\cW$ and then $\Sigma\cong \subseteq\cW$.
This proves that $\cW$-sheaves (constant functors) are contained in $\Sigma$-sheaves.
To prove the reverse inclusion, we need to identify the $\Sigma$-sheaves.
We are going to compare them along the way with the $\Sigma$-local and $\Sigma$-modal objects.
Notice that $1=X^0$ (where 0 is the initial object of $\Fin$) and that the map $X\to 1$ is dual to the map $0 \to 1$ in $\Fin$.
Then a functor $F:\Fin\to \cS$ is a {\it $\Sigma$-local object} if and only if the map $F(0)\to F(1)$ is invertible.
The set of $\cG$-base changes of $\Sigma$ is the set of maps $X^{K\sqcup 1}\to X^K$ for $K\in \Fin$, and 
$F$ is a {\it $\Sigma$-modal object} if and only if $F(K\sqcup 1)\simeq F(K)$ for all $K\in \Fin$.
We have $\Delta^{\infty}(\Sigma) = \{X \to X^{S^n}\,|\,n\geq -1\}$ with $S^{-1}=0$.
A natural transformation $X^K \to X^{S^n}$ is equivalent to a map $S^n\to K$ in $\Fin$.
Let $K\to K\sqcup_{S^n}1$ be the cofiber of this map attaching an $(n+1)$-cell to $K$.
The set of $\cG$-base changes of $\Delta^{\infty}(\Sigma)$ is then the set of maps $X^{K\sqcup_{S^n}1} \to X^K$ for all maps $S^n\to K$.
Finally, $F$ is a {\it $\Sigma$-sheaf} if and only if $F(K\sqcup_{S^n}1)\simeq F(K)$ for all maps $S^n\to K$.
Since any map $A\to B$ in $\Fin$ can be built by a finite sequence of maps $K\to K\sqcup_{S^n}1$, this proves that $\Sigma$-sheaves are functors $\Fin\to \cS$ inverting every map.
Hence they are equivalent to functors $|\Fin|\to \cS$ where $|\Fin|$ is the groupoid reflection of $\Fin$.
But $|\Fin|=1$, and this proves that $\Sigma$-sheaves are constant functors.

\begin{rem}
\label{rem:ex:contr}
As a corollary of \cref{prop:contractible-gen}, we get that the pair $(\Fin\op,\{X\to 1\})$ is a higher site in the sense of \cref{defsite} presenting the topos $\cS$ (even though the topos $\cS$ has a simpler presentation).
It can be proved that the localization $\cS[X]\to \cS$ of \cref{prop:contractible-gen} is {\it not} topological (we will see in a future paper that its topological part is the topos classifying \oo connected objects).
Hence the site $(\Fin\op,\{X\to 1\})$ is not associated to a Grothendieck topology on $\Fin\op$.
\end{rem}

\subsection{The topos classifying \texorpdfstring{$n$-truncated}{n-truncated} objects}

An object $E$ of a topos $\cE$ is {\it $n$-truncated} (for some $-2\leq n<\infty$) if the map $\Delta^{n+2}E:E\to E^{S^{n+1}}$ is an equivalence.
When $n=-2$, we have $S^{-1}=\emptyset$ and we recover the notion of contractible object of \cref{sec:contr}.
This condition involves only finite limits and is therefore preserved by cocontinuous left-exact functors. 
Following \cref{rem:objet-univ}, the condition can be enforced universally by forcing the universal object $X\in \cS[X]$ to become $n$-truncated.
This is done by localizing by the congruence generated by the map $\Sigma=\{X\to X^{S^{n+1}}\}$.
The resulting left-exact localization of $\cS[X]$ will classify $n$-truncated objects.
We are going to describe it using our methods.

Let $\cS\truncated n\subseteq\cS$ be the full subcategory spanned by $n$-truncated spaces.
The inclusion functor has a left adjoint $P_n:\cS\to \cS\truncated n$ (Postnikov truncation).
As a localization of cocomplete categories, 
$P_n$ is generated by the single map $S^{n+1}\to 1$, that is the class of maps inverted by $P_n$ is the strongly saturated class $\{S^{n+1}\to 1\}\ssat \subset \cS$.
We denote by $\Fin\truncated n$ the full subcategory of $\cS\truncated n$ spanned by the $n$-truncations of finite spaces (i.e. by the $P_nX$ for $X\in \Fin$), 
and by $q_n:\Fin\to \Fin\truncated n$ the corresponding restriction of $P_n$.
It can be proven that $q_n$ is a localization (but no longer reflective).
More precisely, $q_n$ preserves finite colimits and inverts universally the map $S^{n+1}\to 1$ in the category of categories with finite colimits \cite[Example 3.2.2]{AS}.

The functor $q_n$ induces a restriction functor $q_n^* : \fun {\Fin\truncated n}\cS \to \fun {\Fin}\cS$ having a left adjoint $(q_n)_!$ and a right adjoint $(q_n)_*$.
The functor $q_n^*$ is fully faithful because $q_n$ is a localization.
And the functor $(q_n)_!$ is left-exact because it is the left Kan extension of the left-exact functor $q_n\op:\Fin\op \to (\Fin\truncated n)\op$ (see \cite[Proposition 6.1.5.2]{HTT} or \cite[Thm 2.1.4]{Anel-Lejay:topos-exp}).
Consequently, $\fun {\Fin\truncated n} \cS$ is a left-exact localization of $\fun {\Fin} \cS$.
We define the object $X\truncated n$ to be the inclusion functor $\Fin\truncated n\to \cS$ and we denote $\fun {\Fin\truncated n} \cS$ by $\cS[X\truncated n]$.
The object $X\truncated n$ is represented by $1\in \Fin\truncated n$.
Since $q_n(1)=1$, we have $X\truncated n = (q_n)_!(X)$.

\begin{lemma}
\label{lem:SXleqn}
The cocontinuous left-exact localization $(q_n)_!:\cS[X] \to \cS[X\truncated n]$ is universal for inverting the map $X\to X^{S^{n+1}}$.
\end{lemma}
\begin{proof}
Using \cref{univlexloc}, the result is equivalent to proving that the image of $\cS[X\truncated n]$ by the fully faithful functor $q_n^*$ coincides with the category of $\Sigma$-sheaves.
We start by identifying $\Sigma$-sheaves.
The map $X\to X^{S^{n+1}}$ is dual to the map $S^{n+1}\to 1$ in $\Fin$.
The $\Sigma$-local objects are then functors $F:\Fin\to \cS$ such that $F(1)\simeq F(S^n)$.
The $\cG$-base changes of $X\to X^{S^{n+1}}$ are the maps $X^{K\sqcup_{S^{n+1}}1}\to X^K$, for any $K$ and any map $S^{n+1}\to K$.
The $\Sigma$-modal objects are then functors such that $F(K)\simeq F(K\sqcup_{S^{n+1}}1)$.
The diagonal closure of $\Delta^{n+2}X$ is the set of all map $X\to X^{S^k}$ for $k\geq n+1$.
The corresponding $\cG$-base changes are the maps $X^{K\sqcup_{S^k}1}\to X^K$, for any $K$, any $k\geq n+1$ and any map $S^k\to K$.
The $\Sigma$-sheaves are then functors such that $F(K)\simeq F(K\sqcup_{S^k}1)$.

For any $K\in\Fin$, the map $K\to P_nK$ can be obtained as a the colimit of a chain (indexed by some ordinal) of cofiber maps $K'\to K'\sqcup_{S^k}1$ for $k\geq n+1$.
For $F:\Fin\to \cS$, we denote again by $F$ its extension $\cS\to \cS$ commuting with filtered colimits.
With this convention, if $F$ is a $\Sigma$-sheaf, we must have $F(K)\simeq F(P_nK)$ (because nonempty ordinals are filtered).
This proves that the unit map $F\to q_n^*(q_n)_!F$ is invertible in $\fun {\Fin} \cS$ and that $\Sigma$-sheaves are included in the image of $q_n^*$.
Conversely, if $F$ is in this image, the map $F(K)\to F(K\sqcup_{S^k}1)$ is equivalent to
$F(P_nK)\to F(P_n(K\sqcup_{S^k}1))$.
But the map $P_nK \to P_n(K\sqcup_{S^k}1)$ is an equivalence since $k\geq n+1$.
This proves that the image of $q_n^*$ coincides with the category of $\Sigma$-sheaves.
\end{proof}

\begin{rem}
\label{rem:ex:trunc}
As a consequence of \cref{lem:SXleqn}, the pair $(\Fin\op,\{X\to X^{S^{n+1}}\})$ is a higher site (\cref{defsite}) presenting the topos $\cS[X\truncated n]$.
When $n=-2$ we recover the site of \cref{rem:ex:contr}.
Equivalently, we could have used the site $(\Fin\op,\{X\to P_nX\})$ where $P_nX$ is the (Postnikov) $n$-truncation of $X$ in $\cS[X]$.
It can be proved that the left-exact localization $\cS[X]\to \cS[X\truncated n]$ of \cref{lem:SXleqn} is not topological.
We will see in a future paper that its topological part is the topos classifying  objects such that the map $X\to X^{S^{n+1}}$ (or equivalently the map $X\to P_nX$) is \oo connected.
Hence the sites $(\Fin\op,\{X\to X^{S^{n+1}}\})$ and $(\Fin\op,\{X\to P_nX\})$ are not associated to a Grothendieck topology on $\Fin\op$.
\end{rem}

For a topos $\cE$, we denote by $\cE\truncated n\subseteq\cE$ the full subcategory spanned by $n$-truncated objects.
Any cocontinuous left-exact functor $\cE\to\cF$ between topoi preserves the construction of $\Delta^{n+2} X$ and isomorphisms, hence induces a functor $\cE\truncated n \to \cF\truncated n$.

\begin{proposition}[$n$-truncated object classifier]
\label{prop:SXleqn}
The functor
\begin{align*}
\bA\truncated n : \Toposop &\longrightarrow \CAT\\
\cE &\longmapsto \cE\truncated n
\end{align*}
is representable by the pair $(\cS[X\truncated n], X\truncated n)$.
\end{proposition}
\begin{proof}
By definition of $n$-truncated objects, this equivalent to prove that the localization 
$\cS[X]\to \cS[X\truncated n]$ is generated by the map $\Delta^{n+2}X:X\to X^{S^{n+1}}$.
But this is \Cref{lem:SXleqn}.
\end{proof}

A similar result is true in the pointed object case.
For a topos $\cE$, we denote by $\cE\ptruncated n$ the category of pointed and $n$-truncated objects in $\cE$.
This defines a functor
\begin{align*}
\bA\ptruncated n:\Toposop &\longrightarrow \CAT\\
\cE &\longmapsto \cE\ptruncated n
\end{align*}
We define $\cS[X\ptruncated n]:= \fun {\Fin\ptruncated n} \cS$, where $\Fin\ptruncated n$ is the category of pointed and $n$-truncated finite spaces;
and the object $X\ptruncated n:\Fin\ptruncated n \to \cS$ as the forgetful functor.
We leave to the reader the proof that $\bA\ptruncated n$ is represented by the pair $(\cS[X\ptruncated n], X\ptruncated n)$.

\subsection{The topoi classifying \texorpdfstring{$n$-connected}{n-connected} objects and pointed \texorpdfstring{$n$-connected}{n-connected} objects}
\label{sec:n-connected-classifier}

An object $E$ in a topos $\cE$ is $n$-connected if the following equivalent conditions are satisfied:
\begin{enumerate}[label=(\arabic*)]
\item \label{item1} its $n$-truncation $P_nE$ is contractible;
\item \label{item2} its diagonals $\Delta^kE:E\to E^{S^{k-1}}$, for $0\leq k\leq n+1$, are surjective.
\end{enumerate}

Each of these conditions provides generators for the left-exact localization of $\cS[X]$ classifying $n$-connected objects.
For \ref{item1} this is the single map $P_nX\to 1$ (which is $n$-truncated), 
for \ref{item2} this is the set $\Sigma = \{\im {\Delta^kX}\,|\,0\leq k\leq n+1\}$
where, for a map $f:X\to Y$, $\im f$ is the inclusion of the image of $f$ in $Y$.
In particular, the congruence $\Sigma\cong$ is topological in the sense of \cref{def:topo-cong}.

Let $i_n:\Fin\connected n\subseteq\Fin$ be the full subcategory spanned by $n$-connected finite spaces.
This inclusion induces a fully faithful functor 
$(i_n)_*:\fun {\Fin\connected n} \cS \to \fun {\Fin} \cS$,
and a reflector $i_n^*:\fun {\Fin} \cS \to \fun {\Fin\connected n} \cS$.
Moreover this localization is left-exact since $i_n^*$ has a left adjoint $(i_n)_!$.
We define $\cS[X\connected n]:=\fun {\Fin\connected n} \cS$ and denote by 
$X\connected n$ the canonical inclusion $\Fin\connected n\to \cS$.
The functor $X\connected n$ is represented by $1\in \Fin\connected n$ and we have $X\connected n = i_n^*(X)$ where $X$ is the inclusion $\Fin \to \cS$.

For a topos $\cE$, we denote by $\cE\connected n\subseteq\cE$ the full subcategory spanned by $n$-connected objects.
Any cocontinuous left-exact functor $\cE\to\cF$ between topoi preserves the diagonal maps 
$\Delta^{k}E:E\to E^{S^{k-1}}$ (for $0\leq k\leq n+1$) and surjective maps, hence it induces a functor $\cE\connected n \to \cF\connected n$.
Our purpose in this section is to prove the following result.

\begin{proposition}[$n$-connected object classifier]
\label{prop:connected-classifier}
The functor
\begin{align*}
\bA\connected n:\Toposop &\longrightarrow \CAT\\
\cE &\longmapsto \cE\connected n
\end{align*}
is representable by the pair $(\cS[X\connected n], X\connected n)$.
\end{proposition}

We are going to prove first the pointed version of the same result.
This situation is simpler because, in this case, the inclusion $i_n:\cS\pointed\connected n \to \cS\pointed$ of pointed $n$-connected spaces in pointed spaces has a right adjoint $C_n$, which sends a finite pointed space $K$ to its {\it $n$-connected cover}.
We define $C_nK$ by the pullback
\[
\begin{tikzcd}
C_nK\ar[r]\ar[d] \pbmark & K\ar[d]\\
1 \ar[r] & P_nK
\end{tikzcd}
\]
where $P_nK$ is the $n$-truncation of $K$.
Equivalently, $C_nK$ can be defined either as the middle object of the $(n-1)$-connected--$(n-1)$-truncated factorization of canonical map $1\to K$.
The functor $P_n:\cS\pointed \to \cS\pointed$ preserves filtered colimits since a filtered colimit of $n$-truncated spaces is $n$-truncated.
Thus, the functor $C_n:\cS\pointed \to \cS\pconnected n$ also preserves filtered colimits.

\begin{lemma}
\label{lem:connected-ind}
$\cS\pconnected n = \mathrm{Ind}\left(\Fin\pconnected n\right)$.
\end{lemma}
\begin{proof}
This can be seen by modelling pointed spaces by pointed simplicial sets.
Recall that a simplicial set $(X_n)$ is $n$-reduced if $X_k=1$ for $0\leq k\leq n$.
In particular, any such simplicial set is canonically pointed.
Any $n$-connected space is the homotopy type of an $n$-reduced simplicial set.
Then, the result follows from the fact that any $n$-reduced simplicial set is a filtered colimit of finite $n$-reduced simplicial subsets.
\end{proof}

\begin{lemma}
\label{lem:connected-kan-ext}
The left Kan extension $(i_n)_!:\fun {\Fin\pconnected n} \cS \to \fun {\Fin\pointed} \cS$ is left-exact.
\end{lemma}
\begin{proof}
Using that 
$\cS\pointed = \mathrm{Ind}\left(\Fin\pointed\right)$
and
$\cS\pconnected n = \mathrm{Ind}\left(\Fin\pconnected n\right)$ (\cref{lem:connected-ind}),
we get equivalences of categories
$
\fun {\Fin\pconnected n} \cS = \fun {\cS\pconnected n}\cS _\mathrm{filt}$
and
$\fun {\Fin\pointed} \cS = \fun {\cS\pointed}\cS _\mathrm{filt}$
where $\fun -- _\mathrm{filt}$ refers to the category of functors preserving filtered colimits.
Both adjoint functors $i_n :\cS\pconnected n \rightleftarrows \cS\pointed: C_n$ 
preserve filtered colimits.
Hence they induce an adjunction $C_n^*\dashv i_n^*$
\[
\begin{tikzcd}
\fun {\Fin\pconnected n} \cS = \fun {\cS\pconnected n}\cS _\mathrm{filt}
\ar[r, shift left, "C_n^*"]&
\fun {\cS\pointed}\cS _\mathrm{filt} =
\fun {\Fin\pointed} \cS 
\ar[l, shift left,"i_n^*"]
\end{tikzcd}
\]
which coincides with the adjunction $(i_n)_!\dashv i_n^*$.
Therefore $(i_n)_! = C_n^*$ and the result is proved since $C_n^*$ is left-exact.
\end{proof}

\begin{lemma}
\label{lem:gen-retract}
Let $r:X\rightleftarrows Y:s$ be a retraction pair ($rs=1$), then the congruences generated by $r$ and by $s$ coincide.
\end{lemma}
\begin{proof}
Any congruence satisfies the ``3-for-2" condition.
Considering the relation $rs=1$, we get $r\in \{s\}\cong$ and $s\in \{r\}\cong$.
\end{proof}

\begin{lemma}
\label{lem:gen-surjection}
Given a cartesian square
\[
\begin{tikzcd}
X'\ar[r,two heads]\ar[d,"f'"'] \pbmark & X\ar[d,"f"]\\
Y' \ar[r,two heads] & Y
\end{tikzcd}
\]
where $g$ is a surjection, the congruences generated by $f$ and $f'$ coincide.
\end{lemma}
\begin{proof}
Since congruences are closed under base change, we have $f'\in \{f\}\cong$ and $\{f'\}\cong\subseteq\{f\}\cong$.
Congruence are local classes by \cref{prop-localclass} since they are acyclic classes.
Thus, $f\in \{f'\}\cong$ since
the base change of $f$ along the surjection $g$ belongs to $\{f'\}\cong$.
This proves the reverse inclusion.
\end{proof}

\begin{proposition}
\label{prop:connected-pointed-gen}
The maps $C_nX\pointed \to X\pointed$ and $P_n(X\pointed)\to 1$ generate the same congruence in $\cS[X\pointed]$.
\end{proposition}
\begin{proof}
We consider the cartesian diagram
\[
\begin{tikzcd}
C_nX\pointed\ar[r,"f'"]\ar[d] \pbmark & X\pointed\ar[d,"p", two heads]\\
1 \ar[r,"f"] & P_nX\pointed\ .
\end{tikzcd}
\]
Using \Cref{lem:gen-retract,lem:gen-surjection}, the congruences generated by $P_nX\pointed\to 1$, $1\to P_nX\pointed$, and $C_nX\pointed\to X\pointed$ coincide.
\end{proof}

\begin{lemma}
\label{lem:connected-iso}
The map $(i_n)_!i_n^*(X\pointed)\to X\pointed$ is canonically isomorphic to the map $C_nX\pointed\to X\pointed$.
\end{lemma}
\begin{proof}
The equivalence 
$\fun {\Fin\pointed} \cS = \fun {\cS\pointed}\cS _\mathrm{filt}$ 
from the proof of \cref{lem:connected-kan-ext},
takes a functor $ F:\Fin\pointed\to \cS$
to a functor $\tilde F:\cS\pointed\to \cS$.
With this notation, for any $J\in \Fin\pointed$, we have
\[
\big((i_n)_!i_n^* F\big)(J)
= \big(C_n^*i_n^*\tilde F\big)(J)
\ =\ \tilde F(C_nJ)
\]
and the adjunction counit $(i_n)_!i_n^* F\to F$ 
corresponds is the map $\tilde F(C_nJ)\to \tilde F(J) = F(J)$ induced by the canonical map $C_nJ\to J$ in $\cS\pointed$.
\end{proof}

For a topos $\cE$, we denote by $\cE\pconnected n$ the category of pointed $n$-connected objects in $\cE$.
Any cocontinuous left-exact functor $\cE\to\cF$ induces a functor $\cE\pconnected n \to \cF\pconnected n$.
We define $\cS[X\connected n\pointed]:=\fun {\Fin\pconnected n} \cS$ and denote by 
$X\pconnected n$ the functor forgetting the base point $\Fin\pconnected n\to \cS$.
Since $\cE\pconnected n= \cE\pointed \times_\cE \cE\connected n$, the functor
\begin{align*}
\bA\pconnected n:\Toposop &\longrightarrow \CAT\\
\cE &\longmapsto \cE\pconnected n
\end{align*}
is the fiber product $\bA\pointed \times_\bA \bA\connected n$.

\begin{proposition}[Pointed $n$-connected object classifier]
\label{prop:pointed-connected-classifier}
The functor $\bA\pconnected n$ is representable by the pair $(\cS[X\connected n\pointed], X\connected n\pointed)$.
\end{proposition}
\begin{proof}
Using \cref{prop:SXpointed}, the $\bA\pconnected n$ is represented by the left-exact localization of $\cS[X\pointed]$ forcing $X\pointed$ to become $n$-connected.
Using \cref{prop:connected-pointed-gen}, this left-exact localization 
is generated by $C_nX\pointed \to X\pointed$.
We need to prove that it coincide with the left-exact localizations
$i_n^*:\fun {\Fin\pointed} \cS \to \fun {\Fin\pointed\connected n} \cS$.
By \cref{univlexloc}, this will be the case if the two corresponding congruences are the same.

Let $\cW$ be the congruence associated to $i_n^*$ of \Cref{exmpcongruence4}.
We are going to show that $\cW = \Sigma\cong$ for $\Sigma = \{C_nX\pointed \to X\pointed\}$.
The functor $(i_n)_!$ is fully faithful since $(i_n)_*$ is.
Therefore, $i_n^*$ is also a coreflection and 
the image by $i_n^*$ of every counit map $(i_n)_!i_n^*(F)\to F$ is invertible.
Hence, we have $\Tau \subseteq \cW$ where $\Tau$ be the class of all these counits.
Since $\cW$ is a congruence, we have $\Tau\cong\subseteq\cW$.
Since $\cW$ is the closure of $\Tau$ for the 3-for-2 condition, we have in fact $\Tau\cong=\cW$.
Let us see now that $\Sigma\cong = \Tau\cong $.
By \Cref{lem:connected-iso}, the maps $C_nX\pointed \to X\pointed$ and $(i_n)_!i_n^*X\pointed \to X\pointed$ coincide and we have $\Sigma \subseteq \Tau$.
Thus, $\Sigma\cong \subseteq \Tau\cong$.
Any functor $F:\Fin\pointed \to \cS$ is a colimit of representable and any representable is a finite limit of $X\pointed$.
Both $(i_n)_!$ and $i_n^*$ are cocontinuous and left-exact functors (\cref{lem:connected-kan-ext}), hence $\Tau\subseteq \Sigma\cong$ because a congruence is closed under finite limits and colimits.
\end{proof}

We now turn to the proof  \cref{prop:connected-classifier}, for which we will need the following two technical lemmas.

\begin{lemma}
\label{lem:po-loc}
Let $f:\cE\to \cF$ be a cocontinuous left-exact functor between topoi.
For $\Sigma$ a set of maps in a topos $\cE$, we denote $f(\Sigma)$ its image by $f$.
Then the diagram
\[
\begin{tikzcd}
\cE \ar[r]\ar[d] 
& \cF \ar[d]
\\
\LOCcclex {\cE}{\Sigma} \ar[r] 
&  \LOCcclex {\cF}{f(\Sigma)}
\end{tikzcd}
\]
is a pushout in $\Toposop$.
\end{lemma}
\begin{proof}
Given a commutative diagram in $\Toposop$
\[
\begin{tikzcd}
\cE \ar[r]\ar[d] 
& \cF \ar[d]
\\
\LOCcclex {\cE}{\Sigma} \ar[r] 
&  \cG
\end{tikzcd}
\]
the commutation says exactly that $f(\Sigma)$ is inverted in $\cG$.
See also \cite[Proposition 6.3.4.6]{HTT}.
\end{proof}

\begin{lemma}[{\cite[Remark 6.3.5.8]{HTT}}]
\label{lem:po-slice}
Let $f:\cE\to \cF$ be a cocontinuous left-exact functor between topoi and $E\in \cE$
Then the diagram
\[
\begin{tikzcd}
\cE \ar[r]\ar[d] 
& \cF \ar[d]
\\
\cE \slice E \ar[r] 
&  \cF \slice {f(E)}
\end{tikzcd}
\]
is a pushout in $\Toposop$.
\end{lemma}

\begin{proof}[Proof of \cref{prop:connected-classifier}]

We consider the map $p_n : X\to P_nX$ in $\cS[X]$.
The left-exact localization $\LOCcclex {\cS[X]} {p_n}$ is the topos classifying $n$-connected objects.
We will prove that $\LOCcclex {\cS[X]} {p_n} = \cS[X\connected n]$ by proving the existence of a commutative square
\begin{equation}
\label{eqn:conn-square}
\begin{tikzcd}
\cS[X] \ar[rr,"\text{localization}"]\ar[d,"\text{localization}"'] 
&& \cS[X\connected n] \ar[d,"\text{conservative}"]
\\
\LOC {\cS[X]}{p_n} \cclex \ar[rr,"\text{conservative}"'] 
&& \cS[X\pointed\connected n]
\end{tikzcd}
\end{equation}
The result will then follow from the unicity of the localization--conservative factorization of a functor (or by the orthogonality properties of localization and conservative functors, see \cref{factsysexemp5}).

\medskip
An object $E$ in a topos $\cE$ is called {\it inhabited} if $p:E\to 1$ is surjective, that is if $p^*:\cE\to \cE\slice E$ is a conservative functor (\cref{defsurjcoverage}).
If $\cE= \cS^C$ is a diagram category, the object $E$ is inhabited if and only if for any $c\in C$, $E(c)$ is a non-empty space.
In particular, for any $n\geq -1$, the object $X\connected n$ is inhabited in $\cS[X\connected n]=\fun {\Fin\connected n}\cS$.
Hence the canonical continuous left-exact functor
\[
\cS[X\connected n]\longrightarrow \cS[X\connected n]\slice {X\connected n} = \cS[X\pointed\connected n]
\]
is conservative.
The image of $p_n$ by $\cS[X]\to \cS[X\pointed]$ is the map $p_n\pointed : X\pointed\to P_nX\pointed$.
Using \cref{lem:po-loc}, we have the following pushout in $\Toposop$
\begin{equation}
\label{eqn:po-conn}
\begin{tikzcd}
\cS[X] \ar[rr]\ar[d] 
&& \cS[X\pointed] \ar[d]
\\
\LOCcclex {\cS[X]}{p_n} \ar[rr] 
&&  \LOCcclex {\cS[X\pointed]}{(p_n\pointed)} \pomarkk
\end{tikzcd}    
\end{equation}
By \cref{prop:pointed-connected-classifier}, we have
\[
\LOCcclex {\cS[X\pointed]}{(p_n\pointed)}
= \cS[X\pointed\connected n].
\]
Using that $\cS[X\pointed]=\cS[X]\slice X$ in \eqref{eqn:po-conn} and \cref{lem:po-slice}, we have also
\[
\LOCcclex {\cS[X\pointed]}{(p_n\pointed)} = \big(\LOCcclex {\cS[X]}{p_n}\big)\slice X
\]
The image of $X$ in $\LOCcclex {\cS[X\pointed]}{(p_n\pointed)}$ is $n$-connected, hence inhabited.
This proves that the bottom map of \eqref{eqn:po-conn} is conservative.

All the maps of the diagram \eqref{eqn:conn-square} have been constructed and proven to be localizations or conservative functors.
We need only to verify that \eqref{eqn:conn-square} is commutative.
Thanks to the universal property of $\cS[X]$ (\cref{prop:SX}), it is enough to verify that the images of the object $X$ are isomorphic.
But this is straightforward.
\end{proof}

\bibliographystyle{alpha}

\begin{thebibliography}{XXXX}

\bibitem[AJ21]{AJ} M. Anel and A. Joyal.
\newblock {\it Topo-logie}. In M. Anel \& G. Catren (Eds.), {\it New Spaces for Mathematics: Formal and Conceptual Reflections }, page 155--257.
\newblock Cambridge University Press (2021).

\bibitem[ALS]{AS} M. Anel and C. Leena Subramaniam, 
{\it Small object arguments, plus-construction, and left-exact localization}.
arXiv:2004.00731 (2020)

\bibitem[AL19]{Anel-Lejay:topos-exp}
M.~Anel and D.~Lejay.
\newblock Exponentiable infinity-topoi (version 2).
\newblock (2019).
\newblock \href{http://mathieu.anel.free.fr/mat/doc/Anel-Lejay-Exponentiable-topoi.pdf}{On Anel's homepage}.

\bibitem[ABFJ1]{ABFJ1} M. Anel, G. Biedermann,
E. Finster, and A. Joyal, {\it Goodwillie's Calculus of functors and Higher
    Topos Theory}. Journal of Topology 11.4 (2018): 1100--1132.

\bibitem[ABFJ2]{ABFJ2} M. Anel, G. Biedermann, E. Finster, and A. Joyal, 
{\it A Generalised Blakers-Massey theorem}. Journal of Topology 13.4 (2020): 1521--1553.



%
%
%
%
  
\bibitem[Cis19]{Cisinski}
D.-C. Cisinski.
\newblock {\em Higher Categories and Homotopical Algebra}.
\newblock Cambridge Studies in Advanced Mathematics. Cambridge University
  Press, 2019.

\bibitem[DF95]{DF95} E. D. Farjoun, {\it Cellular spaces, null
    spaces and homotopy localization}. Number 1621-1622, Lecture Notes
  in Math. (1995).

\bibitem[Joy]{Joy} A. Joyal, {\it The theory of quasi-categories}. Notes (2008).

\bibitem[HTT]{HTT} J. Lurie, {\it Higher Topos Theory}.  Version March 10, (2012).

\bibitem[MV]{MV} F. Morel and V. Voevodsky, {\em
    $\mathbf A^1$-homotopy theory of schemes}. Publications Math\'ematiques de
  l'Institut des Hautes \'Etudes Scientifiques 90.1 (1999): 45--143.

\bibitem[TV05]{TV05} B. To\"{e}n and G. Vezzosi, {\em
  Homotopical Algebraic Geometry I. Topos Theory}.  Adv. Math., 193(2): 257--372, 2005.  

\bibitem[Raptis]{Raptis} G. Raptis, {\it Some characterizations of
    acyclic maps}. Journal of Homotopy and Related Structures 14
  (2019), no. 3, 773--785.
  
\bibitem[Re05]{Rezk:topos} C. Rezk, 
{\it Toposes and Homotopy Toposes} (2005).
\url{https://www.math.uiuc.edu/rezk}

\bibitem[Re19]{Rezk:course} C. Rezk, 
{\it Lectures on Higher Topos theory}, Leeds (2019).
\url{https://faculty.math.illinois.edu/~rezk/leeds-lectures-2019.pdf}
  
 
\bibitem[RSS]{RSS} E. Rijke, M. Shulman, and B. Spitters, 
{\it Modalities in homotopy type theory}.
Logical Methods in Computer Science 16 (2020), no. 1, Paper No. 2, 79 pp.
  

\bibitem[RV21]{RV}
E.~Riehl and D.~Verity.
\newblock Elements of $\infty$-category theory.
\newblock 2021.
\newblock \href{https://emilyriehl.github.io/files/elements.pdf}{Link}.


\bibitem[HoTT]{HoTT} The Univalent Foundations Program, {\it Homotopy
    Type Theory: Univalent Foundations of
    Mathematics}. \url{https://homotopytypetheory.org/book}
  
\end{thebibliography}

\end{document}